\documentclass[11pt]{amsart}
\usepackage[dvips]{color}
\usepackage{amsmath}
\usepackage{amsxtra}
\usepackage{amscd}
\usepackage{amsthm}
\usepackage{amsfonts}
\usepackage{amssymb}
\usepackage{eucal}		
\usepackage{epsfig}
\usepackage{graphics}
\usepackage[matrix,arrow]{xy}
\usepackage{enumitem}

\theoremstyle{plain}
\newtheorem{theorem}{Theorem}[section]

\newtheorem{thm}[theorem]{Theorem}
\newtheorem{cor}[theorem]{Corollary}
\newtheorem{prop}[theorem]{Proposition}
\newtheorem{lem}[theorem]{Lemma}
\newtheorem{conj}[theorem]{Conjecture}
\newtheorem{defi}[theorem]{Definition}
\newtheorem{example}[theorem]{Example}
\newtheorem{rem}[theorem]{Remark}

\newcommand{\Glie}{\mathfrak{g}}

\newcommand{\U}{\mathcal{U}}

\newcommand{\ZZ}{\mathbb{Z}}
\newcommand{\CC}{\mathbb{C}}

\newcommand{\C}{\mathbb{C}}

\newcommand{\g}{\mathfrak{g}}
\newcommand{\bo}{\mathfrak{b}}
\newcommand{\tb}{\mathbf{\mathfrak{t}}}
\newcommand{\ga}{\overline{\alpha}}

\newcommand{\Psib}{\mbox{\boldmath$\Psi$}}
\newcommand{\Psibs}{\scalebox{.7}{\boldmath$\Psi$}}
\newcommand{\nc}{\newcommand}
\nc{\on}{\operatorname}
\nc{\la}{\lambda}
\nc{\wh}{\widehat}
\nc{\wt}{\widetilde}
\nc{\sw}{{\mathfrak s}{\mathfrak l}}
\nc{\ghat}{\wh{\g}}
\nc{\hhat}{\wh{\h}}
\nc{\mc}{\mathcal}
\nc{\bi}{\bibitem}
\nc{\pa}{\partial}
\nc{\ppart}{(\!(t)\!)}
\nc{\pparl}{(\!(\la)\!)}
\nc{\zpart}{(\!(z^{-1})\!)}
\nc{\n}{{\mathfrak n}}
\nc{\ol}{\overline}
\nc{\mb}{\mathbf}
\nc{\bb}{{\mathfrak b}}
\nc{\su}{\wh\sw_2}
\nc{\h}{{\mathfrak h}}
\nc{\can}{\on{can}}
\nc{\ntil}{\wt{\n}}
\nc{\pone}{{\mathbb P}^1}
\nc{\bs}{\backslash}
\nc{\al}{\alpha}
\nc{\gt}{{\mathfrak g}'}
\nc{\ds}{\displaystyle}

\theoremstyle{definition}

\begin{document}

\author{David Hernandez}

\subjclass[2020]{Primary 17B37 (16T99 17B10 82B23 13F60)}

\keywords{Shifted quantum affine algebras and truncations, quantum affine Borel algebras, category $\mathcal{O}$, fusion product, categorification of cluster algebras, Baxter polynomials, Langlands interpolation}

\address{Universit\'e de Paris and Sorbonne Universit\'e, CNRS, IMJ-PRG, IUF, F-75006 Paris, France.}

\email{david.hernandez@u-paris.fr}

\begin{abstract} 
We develop the representation theory of shifted quantum affine algebras $\mathcal{U}_q^\mu(\hat{\Glie})$ and of their truncations
which appeared in the study of quantized K-theoretic Coulomb branches of 3d $N = 4$ SUSY quiver gauge theories.
Our approach is based on novel techniques, which are new in the cases of shifted Yangians or 
ordinary quantum affine algebras as well : realization in terms of asymptotical subalgebras of the quantum affine algebra $\mathcal{U}_q(\hat{\Glie})$, induction and restriction
functors to the category $\mathcal{O}$ of representations of the Borel subalgebra $\mathcal{U}_q(\hat{\bo})$ of $\mathcal{U}_q(\hat{\Glie})$, relations between truncations and Baxter polynomiality in quantum integrable models, 
parametrization of simple modules via Langlands dual interpolation.
We first introduce the category $\mathcal{O}_\mu$ of representations of
$\mathcal{U}_q^\mu(\hat{\Glie})$ and we classify its simple objects. 
Then we establish the existence of fusion products and we get a ring structure on the sum of the 
Grothendieck groups $K_0(\mathcal{O}_\mu)$. We classify simple finite-dimensional representations of $\mathcal{U}_q^\mu(\hat{\Glie})$ 
and we obtain a cluster algebra structure on the Grothendieck ring of finite-dimensional representations.
%We establish a necessary condition for a simple representation to descend to a truncation, which is also sufficient 
%for $\Glie = sl_2$.
We prove a truncation has only a finite number of 
simple representations and we introduce a related partial ordering on simple modules. Eventually, we state a conjecture on the parametrization of simple modules of a non simply-laced truncation 
in terms of the Langlands dual Lie algebra. We have several evidences, including a general result 
for simple finite-dimensional representations.
\end{abstract}

\begin{title}
{Representations of shifted quantum affine algebras}
\end{title}

\maketitle

\tableofcontents

\section{Introduction}

Shifted quantum affine algebras and their truncations arose \cite{FT} in the study of quantized $K$-theoretic Coulomb branches of 
3d $N = 4$ SUSY quiver gauge theories in the sense of Braverman-Finkelberg-Nakajima \cite{bfn} which are at the center of 
current important developments (see for instance \cite{Nreview, Fi} and references therein). A presentation of (truncated) 
shifted quantum affine algebras by generators and relations was given by Finkelberg-Tsymbaliuk \cite{FT}. 
Their rational analogs, 
the shifted Yangians, and their truncations, appeared in type $A$ in the context of the representation theory of finite $W$-algebras \cite{brkl}, 
then in the study of quantized affine Grassmannian slices \cite{kwwy1} for general types and in the study of quantized Coulomb branches of 3d $N = 4$ SUSY quiver gauge theories \cite{bfn} for simply-laced types and \cite{nw} for non simply-laced types.

Shifted quantum affine algebras $\mathcal{U}_q^{\mu_+,\mu_-}(\hat{\Glie})$ can be seen as variations of Drinfeld-Jimbo quantum affine algebras $\mathcal{U}_q(\hat{\Glie})$ 
in their Drinfeld presentation, but depending on two coweights $\mu_+$, $\mu_-$ 
of the underlying simple Lie algebra $\mathfrak{g}$. 
In the particular case $\mu_+ = \mu_- = 0$, the algebra $\mathcal{U}_q^{0,0}(\hat{\Glie})$ is a central extension of the ordinary quantum affine algebra $\mathcal{U}_q(\hat{\Glie})$. 
The algebra $\mathcal{U}_q(\hat{\Glie})$ and its representations have been under intense study in recent years, the reader may refer to the recent ICM talks \cite{K3, o} 
for recent important developments.

The truncations depend on additional parameters, including a dominant coweight $\lambda$. In this context, these parameters $\lambda$ and 
$\mu = \mu_+ + \mu_-$ can be interpreted as parameters for generalized slices of the affine Grassmannian $\overline{\mathcal{W}}_\mu^\lambda$ (usual slices when $\mu$
is dominant) or its symplectic dual in the sense of \cite{blpw}, a Nakajima quiver variety $\mathcal{M}_{\lambda,\mu}$. 
Up to isomorphism, $\mathcal{U}_q^{\mu_+,\mu_-}(\hat{\Glie})$ only depends on $\mu$ and will be denoted 
simply by $\mathcal{U}_q^\mu(\hat{\Glie})$.

For simply-laced types, representations of shifted Yangians and related Coulomb branches have been 
intensively studied, see \cite{brkl, ktwwy, ktwwy2} and references therein. 
For non simply-laced types, representations of quantizations of Coulomb branches have been
studied by Nakajima and Weekes \cite{nw} by using the method originally developed in \cite{N4} for simply-laced types.

In the present paper, we develop the representation theory of shifted quantum affine algebras with an approach based on several novel techniques : 

\smallskip

\begin{enumerate}[label=(\arabic*)]

\item  for $\mu$ anti-dominant, realization in terms of the asymptotic algebra introduced in \cite{HJ}, which is 
a subalgebra of the ordinary quantum affine algebra $\mathcal{U}_q(\hat{\Glie})$,

\item  induction and restriction
functors to the category $\mathcal{O}$ of representations of the Borel subalgebra $\mathcal{U}_q(\hat{\bo})$ of $\mathcal{U}_q(\hat{\Glie})$, 

\smallskip

\item relations of truncations with Baxter polynomiality in quantum integrable models, 

\smallskip

\item parametrization of simple modules via Langlands dual interpolating $(q,t)$-characters.

\end{enumerate}

\smallskip

We underline that these techniques, and a large part of our results, are also new for ordinary quantum affine algebras
or shifted Yangians. Hence our study goes beyond a trigonometric version of known results for Yangians.

Let us explain our results. We first relate these representations to $q$-oscillator algebras and to the quantum affine Borel algebra $\mathcal{U}_q(\hat{\bo})$. It is known since \cite{blz} that certain representations of the $q$-oscillator algebra 
give rise to representations of the quantum affine Borel algebra $\mathcal{U}_q(\hat{\bo})$ of the quantum affine algebra $\mathcal{U}_q(\hat{sl}_2)$.
 For general untwisted types, the category $\mathcal{O}$ of representations of the quantum affine Borel algebra $\mathcal{U}_q(\hat{\bo})$ was 
introduced and studied in \cite{HJ}. Some representations in this category extend to a representation of the 
whole quantum affine algebra $\mathcal{U}_q(\hat{\Glie})$, but many do not, including the prefundamental representations constructed in \cite{HJ} and whose transfer-matrices
have remarkable properties for the corresponding quantum integrable systems \cite{FH}.
However, it was first observed in \cite{HJ} that for some of the simple representations in $\mathcal{O}$, the structure of representation of $\mathcal{U}_q(\hat{\bo})$ can 
be extended to a larger subalgebra of $\mathcal{U}_q(\hat{\Glie})$, the asymptotic algebra $\tilde{\mathcal{U}}_q(\hat{\Glie})$. 
It was observed in \cite{Z} that, in the Yangian case, certain examples of simple representations in an analog of the category $\mathcal{O}$ can be extended to a shifted Yangian. We will prove that all antidominant shifted quantum affine algebras can be realized in terms of $\tilde{\mathcal{U}}_q(\hat{\Glie})$.

This picture motivated us to introduce a category $\mathcal{O}_\mu$ of representations of $\mathcal{U}_q^\mu(\hat{\Glie})$ which is an analog of the ordinary category $\mathcal{O}$. In particular, the category $\mathcal{O}_0$ contains the category of finite-dimensional representations
of the quantum affine algebra $\mathcal{U}_q(\hat{\Glie})$, but for some other $\mu$, there are much more finite-dimensional representations in the 
categories $\mathcal{O}_\mu$ which seem to be a natural extension of the ordinary representation theory of quantum affine algebras.
These categories $\mathcal{O}_\mu$ are the main categories studied in the present paper. 

\smallskip

Our main results are the following  : 

\smallskip

\begin{enumerate}[label=(\arabic*)]

\item The classification of simple representations in the category $\mathcal{O}_\mu$.

\smallskip

\item The classification of simple finite-dimensional representations of $\mathcal{U}_q^\mu(\hat{\Glie})$.

\smallskip

\item A ring structure on the sum of Grothendieck groups $K_0(\mathcal{O}_\mu)$ from fusion products.

\smallskip

\item A cluster algebra structure on the Grothendieck ring of finite-dimensional 
representations of shifted quantum affine algebras.

\smallskip 

\item A $q$-characters formula for simple finite-dimensional 
representations of $\mathcal{U}_q^\mu(\hat{\Glie})$ in terms of the $q$-character 
of simple representations of $\mathcal{U}_q(\hat{\bo})$.

\smallskip

\item The rationality and polynomiality of remarkable Cartan-Drinfeld series on 
simple representations in $\mathcal{O}_\mu$, 
using Baxter polynomiality of quantum integrable models.

\smallskip

\item The proof of the finiteness of the number of simple isomorphism classes for truncations for 
general types and their complete classification for $\Glie = sl_2$.

\smallskip

\item A new partial ordering on simple representations of $\mathcal{U}_q^\mu(\hat{\Glie})$.

\smallskip

\item The statement of a conjecture to explicitly parametrize simple representations of non simply-laced truncated shifted quantum affine algebras involving the Langlands dual $\mathcal{U}_q({}^L\hat{\Glie})$.

\smallskip

\item The proof that simple finite-dimensional representations descend to a truncation as in the conjecture.

\end{enumerate}

\smallskip

Let us comment on related structures and on previous results.

For simply-laced types, simple representations of truncated shifted Yangians have been parametrized in terms of Nakajima monomial crystals \cite{ktwwy2}. Combining with \cite{N4}, this implies an analog statement for simply-laced shifted quantum affine algebras. 
This is a fundamental motivation for our conjecture (9) in non simply-laced types. 
We have several strong evidences, including a complete result in type $B_2$, and a general result for finite-dimensional representations as mentioned above.

Based on \cite{N4}, Nakajima-Weekes \cite{nw} gave a bijection between more general simple representations of a non simply-laced quantized Coulomb branch and
those for simply-laced types (they are parametrized by the same canonical base). 
Nakajima explained to the author this bijection preserves finiteness of dimension and category $\mathcal{O}$. Thus, 
combining with \cite{ktwwy2}, this gives an explicit parametrization of simple representations in category $\mathcal{O}$ of truncated non 
simply-laced shifted Yangians and quantum affine algebras. 
After using the comparison between simply-laced and twisted $q$-characters \cite{H3}, one can 
consider a possible relation between the two parametrizations. 
In small examples discussed in a correspondence between Nakajima and the author, this different method seems to give the same parametrization as our result.
Note also that results (7) above can be obtained by this method from simply-laced types and an equivalence of representations of truncated shifted Yangians/quantum affine algebras. For this last point, in the formulation of \cite{N4, nw}, once the spectral parameters are specialized, the algebras relevant to homological and $K$-theoretic Coulomb branches 
become isomorphic by Riemann-Roch theorem (there should be also other approaches to this last problem in some cases, in the spirit of \cite{GTL}).

One of the aims of the last part of the paper is to understand truncations and their representations, 
uniformly, from the direct methods developed in the first parts. 
We show that it is also relevant to use the theory of quantum integrable models we have previously studied.
We derive an explicit parametrization using the direct algebraic and transfer-matrices 
approaches. 

Let us recall that to each finite-dimensional representation $V$ of $\mathcal{U}_q(\hat{\Glie})$, and 
more generally to each representation of $\mathcal{U}_q(\hat{\mathfrak{b}})$ in the category $\mathcal{O}$, 
is attached a transfer-matrix $t_V(z)$ which is a formal power series in a formal parameter $z$ with coefficients 
in $\mathcal{U}_q(\hat{\Glie})$ (the transfer-matrix is defined via the $R$-matrix construction).
Given another simple finite-dimensional representation $W$ of $\mathcal{U}_q(\hat{\Glie})$, we get a family
of commuting operators on $W[[z]]$. This is a quantum integrable model generalizing the $XXZ$-model. 
As conjectured in \cite{Fre} and established in \cite{FH}, the spectrum of this system, that is the eigenvalues
of the transfer-matrices, can be described in terms of certain polynomials, generalizing Baxter's polynomials 
associated to the $XXZ$-model. These 
Baxter's polynomials are obtained from 
the eigenvalues of transfer-matrices associated to prefundamental representations of $\mathcal{U}_q(\hat{\mathfrak{b}})$. 
Moreover, this Baxter polynomiality implies the polynomiality of certain series of Cartan-Drinfeld elements acting on 
finite-dimensional representations \cite{FH}. We relate this result to the structures of representations 
of truncated shifted quantum affine algebras (results (6), (8), (10) above).

In non-simply-laced types, the parametrization (9) does not involve directly the 
monomial crystal or the $q$-character of a standard module, but a "mixture" between the 
$q$-characters of Langlands dual standard modules obtained from interpolating $(q,t)$-characters. 
The latter were defined in \cite{FH0} as an 
incarnation of Frenkel-Reshetikhin deformed $\mathcal{W}$-algebras interpolating between 
$q$-characters of a non simply-laced quantum affine algebra and its Langlands dual. 
They lead \cite{FH0, FHR} to the definition of an interpolation between the Grothendieck ring 
$\text{Rep}(\U_q(\hat{\Glie}))$ of finite-dimensional representations of
 $\U_q(\hat{\Glie})$ (at $t = 1$) and the Grothendieck ring $\text{Rep}(\U_t({}^L\hat{\Glie}))$ 
of finite-dimensional representations of the Langlands dual
 algebra quantum affine algebra $\U_t({}^L\hat{\Glie})$ (at $q = \epsilon$ a certain root of $1$) : 
$$\xymatrix{ &\mathfrak{K}_{q,t}\ar@{-->}[dr]^{q = \epsilon}\ar@{-->}[dl]^{t = 1}&   
\\ \text{Rep}(\U_q(\hat{\Glie})) & &  \text{Rep}(\U_t({}^L\hat{\Glie}))
}$$
Here $\mathfrak{K}_{q,t}$ is the ring of interpolating $(q,t)$-characters. To describe our parametrization (9), we 
found it is relevant to use a different specialization of interpolating $(q,t)$-characters that we call Langlands dual $q$-characters 
(with $t = 1$ for variables but $q = \epsilon$ for coefficients).

The interpolating $(q,t)$-characters are closely related to the deformed $\mathcal{W}$-algebras which appeared in \cite{Wd} in the context of the geometric Langlands program. 
Note also that the parametrization in \cite{ktwwy2} for simply-laced types can be understood in the context of symplectic duality (more precisely from the equivariant version of the Hikita conjecture \cite{hik} for the symplectic duality formed by an affine Grassmannian slice 
and a quiver variety). Hence the statement of our conjecture has also as main motivations the symplectic duality and the Langlands duality. 

Let us discuss another application of our approach in the context of cluster algebra theory (result (4) above). 
The cluster
algebra $\mathcal{A}(Q)$ attached to a quiver $Q$ is a commutative ring 
with a distinguished set of generators called cluster variables and obtained
inductively via the Fomin-Zelevinsky procedure of mutation \cite{FZ1}. 
A monoidal category $\mathcal{C}$ is said to be a monoidal categorification of $\mathcal{A}(Q)$ if
there exists a ring isomorphism 
$\mathcal{A}(Q)\stackrel{\sim}{\rightarrow} K_0(\mathcal{C})$, with some additional properties, 
in particular with cluster variables corresponding to classes of certain (prime) simple modules (see \cite{HL0}). 
Various examples of monoidal categorifications have been established. It was proved in \cite{HL} that 
the Grothendieck rings of certain categories of representations of $\mathcal{U}_q(\hat{\mathfrak{b}})$ have a cluster 
algebra structure. Using (3) and induction/restriction functors, we obtain a cluster algebra structure on the Grothendieck ring of finite-dimensional 
representations of shifted quantum affine algebras (result (4)).

\smallskip

We expect our results and conjectures will lead to further developments in the representation theory of shifted quantum affine algebras. We also expect our results to extend to twisted shifted quantum affine algebras. 

\medskip

The paper is organized as follows.

\smallskip

In Section \ref{finitetype}, we consider finite-type analogs of shifted quantum affine algebras and we underline the relation with the $q$-oscillator algebras. 

In Section \ref{sqaa}, we recall the definition of shifted quantum affine algebras $\mathcal{U}_q^\mu(\hat{\Glie})$. In the $sl_2$-case, 
we consider evaluation morphisms to the $q$-oscillator algebra (Proposition \ref{evalm}) which give examples of evaluation representations. 
For general types, we prove for $\mu$ anti-dominant that $\mathcal{U}_q^\mu (\hat{\Glie})$ can be realized from the asymptotic algebra and that it contains a subalgebra isomorphic to the quantum affine Borel algebra $\mathcal{U}_q(\hat{\bo})$ (Proposition \ref{suba}).

In Section \ref{catosh}, we introduce the category $\mathcal{O}_\mu$ of representations of the shifted quantum affine algebra 
$\mathcal{U}_q^\mu(\hat{\Glie})$ and we classify its simple objects (Theorem \ref{param}). The category $\mathcal{O}^{sh} = \bigoplus_{\mu\in\Lambda}\mathcal{O}_\mu$ is the sum 
of the abelian categories $\mathcal{O}_\mu$.
We study shift functors induced by shift homomorphisms. 

In Section \ref{fusionr}, we construct the fusion product of representations of shifted quantum affine algebras by using the deformed Drinfeld coproduct and 
the renomalization procedure introduced in \cite{H2} (Theorem \ref{fp}). This leads to the definition of a ring structure on the Grothendieck groups $K_0(\mathcal{O}^{sh})$. We establish a simple module in $\mathcal{O}^{sh}$ is a quotient of a fusion product of various prefundamental ones (Corollary \ref{role}).
Along the way we consider analogs of Frenkel-Reshetikhin $q$-characters of representations of shifted quantum affine algebras. We establish $q$-characters of simple representations
satisfy a triangularity property with respect to Nakajima partial ordering (Theorem \ref{partialo}).

In Section \ref{fdr}, we classify the simple finite-dimensional representations of shifted quantum affine algebras (Theorem \ref{fdclas}). 

In Section \ref{irf}, we define and study induction and restriction functors relating the category $\mathcal{O}$ of representations
of the quantum affine Borel algebra $\mathcal{U}_q(\hat{\bo})$ and the categories $\mathcal{O}_\mu$ of representations of shifted quantum affine algebras.

In Section \ref{charclus}, we establish a $q$-characters formula for simple finite-dimensional 
representations of shifted quantum affine algebras in terms of the $q$-characters 
of certain simple representations of $\mathcal{U}_q(\hat{\bo})$ 
in the category $\mathcal{O}$ (Theorem \ref{charqf}). Then, we prove 
the results in \cite{HL} imply a description 
of simple finite-dimensional representations of $\mathcal{U}_q^\mu(\hat{sl}_2)$ (Theorem \ref{fact}), 
isomorphisms of Grothendieck rings between categories of representations
 of $\mathcal{U}_q^\mu(\hat{\Glie})$ associated to dominant and anti-dominant $\mu$ 
(Theorem \ref{exdu}), 
and a cluster algebra structure on the Grothendieck ring of finite-dimensional 
representations of shifted quantum affine algebras (Theorem \ref{clsh}).

In Section \ref{cds}, we recall Cartan-Drinfeld series $Y_i^\pm(z)$, $T_i^\pm(z)$ introduced respectively in \cite{Fre} and in \cite{HJ} 
in the study of transfer-matrices of representations of 
quantum affine algebras. We establish (Theorem \ref{newpol}) the rationality of $Y_i^\pm(z)$ (resp. the polynomiality 
of $(T_i^\pm(z))^{\mp 1}$) on a simple representation of a shifted quantum affine algebra in the category $\mathcal{O}_\mu$, 
using Baxter polynomiality of quantum integrable models \cite{FH}.

In Section \ref{tsqaa}, we recall the definition of truncated shifted quantum affine algebras and we explain how 
it appears naturally  in terms of the Cartan-Drinfeld series $Y_i^\pm(z)$, $T_i^\pm(z)$. We establish (Proposition \ref{rata}) a necessary and sufficient condition 
for the defining series to have a rational action on a simple representation.

In Section \ref{dttt}, we start investigating which simple representations descend to truncated shifted quantum affine algebras.
We establish a necessary condition  (Proposition \ref{maint}). It implies that a 
truncated shifted quantum affine algebra has only a finite number of isomorphism classes of simple representations (Theorem \ref{finsimp}).
Then we introduce a partial ordering $\preceq_{\mathcal{Z}}$ on simple modules, related to the Langlands dual Lie algebra ${}^L\Glie$. 
We prove a related triangularity property 
for simple representations of truncated shifted quantum affine algebra (Theorem \ref{thpartial}).
In the $sl_2$-case we characterize simple representations of a truncated  shifted quantum affine algebra (Theorem \ref{carsl22}).

In Section \ref{maincsec}, 
we state a conjecture (Conjecture \ref{mainc}) on the parametrization of simple modules of non simply-laced truncated shifted quantum affine algebras. 
It is given in terms Langlands dual $q$-characters that we introduce.
 We establish in type $B_2$ that our parametrization indeed gives representations of the 
truncated shifted quantum affine algebra (Proposition \ref{tdeux}). In general, we establish that a simple finite-dimensional
representation of a shifted quantum affine algebra descends to a truncation as in Conjecture \ref{mainc} (Theorem \ref{truncfd}). 
The proof of this last result is based on Baxter polynomiality.

\medskip

{\bf Acknowledgments : } The author is grateful to Sasha Tsymbaliuk for a very careful reading of a first version of this 
paper and for numerous precious comments.
The author is grateful to Hiraku Nakajima for useful discussions and correspondences, in particular for a 
question on prefundamental representations from which this project started and for explanations of the consequences
of results in \cite{N4}. He is also grateful to Vyjayanthi Chari, 
Pavel Etingof, Christof Geiss, Bernard Leclerc, Alex Weekes and Huafeng Zhang for useful discussions.
The author was supported by the European Research Council under the European Union's Framework Programme H2020 with ERC Grant Agreement number 647353 Qaffine.

\section{$q$-oscillator algebras}\label{finitetype}

We first consider finite-type analogs of shifted quantum affine algebras and we underline the relation with the $q$-oscillator algebras. 

\subsection{Definition} Let $\Glie$ be a simple finite-dimensional of rank $n$ and $I = \{1,\cdots, n\}$. We denote by 
$(\omega_i)_{i\in I}$, $\{\alpha_i\}_{i\in I}$, $\{\alpha_i^\vee\}_{i\in I}$, $\{\omega_i^\vee\}_{i\in I}$ the fundamental weights, the simple roots, simple
coroots and fundamental coweights respectively. $P$ is integral weight lattice and $P_\mathbb{Q} = P\otimes \mathbb{Q}$. 
The set of dominant weights is denoted by $P^+$. The Cartan matrix is $C = (\alpha_j(\alpha_i^\vee))_{i,j\in I}$ and $r_1,\cdots, r_n > 0$ integers are 
minimal so that we have a symmetric matrix
$$B = \text{diag}(r_1,\cdots, r_n)C.$$ 
Let $q\in\mathbb{C}^*$ be not a root of unity. For $i\in I$, we set $q_i = q^{r_i}$. The quantum Cartan matrix $C(q)$ is 
defined for $i\neq j\in I$ by 
$$C_{i,i}(q) = [2]_{q_i}\text{ and } C_{i,j}(q) = [C_{i,j}]_q,$$
with the standard $q$-number notation $[m]_u = \frac{u^m - u^m}{u - u^{-1}}$ for $m\in\mathbb{Z}$, $u\in\mathbb{C}^*\setminus\{-1,1\}$.

Set $\tb^*=\bigl(\C^\times\bigr)^I$, and endow it with a group structure by pointwise multiplication. 
We define a group morphism $\overline{\phantom{u}}:P_\mathbb{Q} \longrightarrow \tb^*$ by setting  
$\overline{\omega}(i)=q_i^{\omega(\alpha_i^\vee)}$. We shall use the standard partial ordering  on $\tb^*$:
$$\omega\preceq \omega' \quad \text{if $\omega' \omega^{-1}$ is a product of $\{\ga_i\}_{i\in I}$}.$$

We consider the following generalization of $q$-oscillator algebras. For $J,K\subset I$, 
the algebra $\mathcal{U}_q^{J , K}(\Glie)$ is defined by the same generators $k_i^{\pm 1}$, $e_i$, $f_i$ as for the quantum group $\mathcal{U}_q(\Glie)$ but with the modified relation
$$[e_i , f_i] = \frac{ \delta_{i\notin K} k_i - \delta_{i\notin J} k_i^{-1} }{q_i - q_i^{-1}}\text{ for $i\in I$.}$$ 
The other relations are the same, that is for $i,j\in I$  
$$k_i k_j = k_j k_i\text{ , }k_i e_j = q_i^{C_{i,j}} e_j k_i\text{ , }k_i f_f = q_i^{-C_{i,j}}f_j k_i,$$
and if $i\neq j$, $[e_i,f_j] = 0$ and  
$$\sum_{0\leq r\leq 1 - C_{i,j}}(-1)^r \begin{bmatrix}1 - C_{i,j}\\ r \end{bmatrix}_{q_i}  e_i^{1 - C_{i,j} - r}e_je_i^r = 
0 = \sum_{0\leq r\leq 1 - C_{i,j}}(-1)^r \begin{bmatrix}1 - C_{i,j}\\ r \end{bmatrix}_{q_i}  f_i^{1 - C_{i,j} - r}f_jf_i^r.$$
We used here the standard $q$-binomials from the standard $q$-factorial notations.

\begin{rem} For $gl_{n+1}$, certain contracted algebras are introduced in \cite{t1, t2} and are very close the 
algebras $\mathcal{U}_q^{J, K}(sl_{n+1})$ above. This was pointed out to the author by Huafeng Zhang.
\end{rem}

\subsection{$q$-oscillator algebras in the $sl_2$-case}
We recover the usual $q$-oscillator algebras 
$$\mathcal{U}_q^+(sl_2) = \mathcal{U}_q^{\emptyset,\{1\}}(sl_2)\text{ and }\mathcal{U}_q^-(sl_2) = \mathcal{U}_q^{\{1\},\emptyset}(sl_2).$$ 
$\mathcal{U}_q^\pm(sl_2)$ is generated by $e$, $f$, $k$, $k^{-1}$ with relations
$$ke = q^2 ek\text{ , }kf = q^{-2}fk\text{ , }kk^{-1} = k^{-1}k = 1\text{ , }[e,f] = \frac{\pm k^{\pm 1}}{q - q^{-1}}.$$
Note that exchanging $e$ and $f$, $k$ and $k^{-1}$, we have an isomorphism 
\begin{equation}\label{isompm}\mathcal{U}_q^+(sl_2)\simeq \mathcal{U}_q^-(sl_2).\end{equation}

\begin{rem}\label{basic} (i) The quantum Boson algebra $\mathcal{B}_q(sl_2)$ of Kashiwara \cite{K} is isomorphic to the subalgebra of $\mathcal{U}_q^+(sl_2)$ 
generated by $f$ and 
$$e' = (q - q^{-1})k^{-1}e\text{ as }e'f-q^2 fe' = 1.$$ 

(ii) We denote $\mathcal{U}_{q,loc}^\pm(sl_2)$ the algebra $\mathcal{U}_q^\pm(sl_2)$ localized at the Casimir central element
$$C_\pm = ef + \frac{q^{\mp 1}k^{\pm 1}}{(q - q^{-1})^2} = fe + \frac{q^{\pm 1} k^{\pm 1}}{(q - q^{-1})^2}.$$

(iii) The algebra $\mathcal{U}_q^\pm (sl_2)$ has a natural triangular decomposition. In particular the Borel subalgebra $\mathcal{U}_q(\bo)\subset \mathcal{U}_q(sl_2)$ generated by $e$, $k$, $k^{-1}$ is a subalgebra of $\mathcal{U}_q^\pm(sl_2)$.\end{rem}

\begin{prop} There is an anti-isomorphism $S : \mathcal{U}_q^+(sl_2)\rightarrow \mathcal{U}_q^-(sl_2)$ defined by 
$$S(e) = - e k\text{ , }S(f) = - k^{-1}f\text{ , }S(k) =  k^{-1}\text{ , }S(k^{-1}) = k.$$
\end{prop}
Composing with the isomorphism (\ref{isompm}), it gives also anti-automorphisms
$$S^\pm : \mathcal{U}_q^\pm(sl_2)\rightarrow \mathcal{U}_q^\pm(sl_2).$$

\begin{proof} It suffices to check the relations are preserved. For the first three relations, it follows from
the fact that the usual antipode is well-defined. For the last one, one has :
$$[S(f),S(e)] = [f,e] = \frac{k^{-1}}{q - q^{-1}} = S([e,f]).$$
\end{proof}

\begin{prop} There are algebra morphisms
$$\Delta_\pm : \mathcal{U}_q(sl_2)\rightarrow \mathcal{U}_q^\pm(sl_2)\otimes \mathcal{U}_q^\mp(sl_2),$$
$$\Delta_\pm (e) = e\otimes 1 + k^{\mp 1}\otimes e\text{ , }\Delta_\pm (f) = f\otimes k^{\pm 1} + 1 \otimes f\text{ , }\Delta_\pm (k) = k\otimes k\text{ , }\Delta_\pm (k^{-1}) = k^{-1}\otimes k^{-1}.$$
The same formulas define left-comodule and right-comodule structures :
$$\mathcal{U}_q^\pm (sl_2) \rightarrow \mathcal{U}_q(sl_2)\otimes \mathcal{U}_q^\pm (sl_2)\text{ , }
\mathcal{U}_q^\mp (sl_2) \rightarrow \mathcal{U}_q^\mp(sl_2)\otimes \mathcal{U}_q (sl_2).$$
\end{prop}

\begin{proof}  It suffices to check the compatibility with the defining relations of $\mathcal{U}_q(sl_2)$. For the first three standard relations which are
the same as for the quantum groups, it follows from the fact that the usual coproduct is well-defined. For the last one, one has :
$$[\Delta_\pm(e),\Delta_\pm(f)] = [e\otimes 1 + k^{\mp 1}\otimes e , f\otimes k^{\pm 1} + 1 \otimes f ]
= [e\otimes 1, f\otimes k^{\pm 1}] + [k^{\mp 1}\otimes e , 1\otimes f] $$
$$= \frac{\pm k^{\pm 1}}{q  -q^{-1}}\otimes k^{\pm 1} 
\mp k^{\mp 1}\otimes \frac{k^{\mp 1}}{q - q^{-1}} =  \frac{k\otimes k - k^{-1}\otimes k^{-1}}{q - q^{-1}} 
= \Delta_\pm([e,f]).$$
\end{proof}

Note that composing with the isomorphim (\ref{isompm}), we also get algebra morphisms
$$\mathcal{U}_q(sl_2)\rightarrow \mathcal{U}_q^+(sl_2)\otimes \mathcal{U}_q^+(sl_2)$$
which can be considered as analogs of a coproduct for $\mathcal{U}_q^+(sl_2)$. The author believes these maps are known but could
not find them in the literature.

\subsection{Representations of $\mathcal{U}_q^\pm (sl_2)$}

Let $\mathcal{C}_0$ be the category of $\mathcal{U}_q(sl_2)$-modules and $\mathcal{C}_1$ be the category of $\mathcal{U}_q^+(sl_2)$-modules.
It is well known that $\mathcal{C}_0$ has a monoidal structure
$$\otimes : \mathcal{C}_0\times \mathcal{C}_0\rightarrow \mathcal{C}_0.$$
From the algebra morphisms above we get bi-functors
$$\mathcal{C}_0\times \mathcal{C}_1\rightarrow \mathcal{C}_1\text{ , }\mathcal{C}_1\times \mathcal{C}_0\rightarrow \mathcal{C}_1\text{ , }\mathcal{C}_1\times \mathcal{C}_1\rightarrow \mathcal{C}_0,$$
which make the category $\mathcal{C}_0\oplus \mathcal{C}_1$ into a $\mathbb{Z}/2\mathbb{Z}$-graded monoidal category. In particular we have a ring structure on
$$K_0(\mathcal{C}_0) \oplus K_0(\mathcal{C}_1).$$

\begin{rem} (i) $\mathcal{C}_0$ admits a left and a right action by the category $\mathcal{C}_1$.

(ii) For $V^\pm$ a representation of $\mathcal{U}_q^\pm(sl_2)$, $V^+\otimes V^-$ and $V^-\otimes V^+$ are $\mathcal{U}_q(sl_2)$-modules.

(iii) We expect to get a $(\mathbb{Z}/2\mathbb{Z})^n$-graded monoidal category for a general $\Glie$, although we will not use it in this paper. We plan to come back to it in a forthcoming paper.
 \end{rem}

From the triangular decomposition (see (iii) in Remark \ref{basic}), $\mathcal{U}_q^+ (sl_2)$ (resp. $\mathcal{U}_q^-(sl_2)$) 
has a Verma module $V(\gamma)$ (resp. $W(\gamma)$) associated to each $\gamma\in\CC^*$ eigenvalue of $k$. These representations $V(\gamma) = \text{Vect}(v_r)_{r\geq 0}$, $W(\gamma) = \text{Vect}(w_r)_{r\geq 0}$ can be explicitly described : 
$$e.v_r = \delta_{r > 0}v_{r - 1}\text{ , }f.v_r = \gamma q^{-r} \frac{[r+1]_q}{q - q^{-1}}v_{r + 1}\text{ , }k.v_r = \gamma q^{-2r}v_r,$$
$$e.w_r = \delta_{r > 0}w_{r - 1}\text{ , }f.w_r = - \gamma^{-1} q^{r} \frac{[r+1]_q}{q - q^{-1}}w_{r + 1}\text{ , }k.w_r = \gamma q^{-2r}w_r.$$

\begin{rem} $V(\gamma)$ is also a representation of $\mathcal{U}_{q,loc}^+(sl_2)$ as the action of $C_+$ is $\frac{q\gamma\text{Id}}{(q - q^{-1})^2}$.
\end{rem}

For $L$ a representation of $\mathcal{U}_q^\pm(sl_2)$, we can define its weight space $L_{\gamma }= \{v\in V|k.v = \gamma v\}$ for $\gamma\in\mathbb{C}^*$. 
We have a corresponding category $\mathcal{O}_\pm$ of $\mathcal{U}_q^\pm(sl_2)$-modules 
(defined as for the ordinary category $\mathcal{O}$, see Section \ref{remcato}) and the corresponding character morphisms : 
$$\chi : K_0(\mathcal{O}_\pm) \rightarrow \mathcal{E}\text{ , }\chi(L) = \sum_{\gamma}\text{dim}(L_\gamma)[\gamma],$$
where $\mathcal{E} \subset \ZZ^{\CC^*}$ is a ring of formal power series as for the category $\mathcal{O}$ of $\mathcal{U}_q(sl_2)$ 
(see \cite[Section 3.4]{HJ} for instance). 

\begin{prop}\label{nonze} The Verma modules of $\mathcal{U}_q^\pm(sl_2)$ are irreducible and exhaust all simple modules of the category $\mathcal{O}_\pm$.
The non-zero representations of $\mathcal{U}_q^\pm(sl_2)$ are infinite-dimensional.
\end{prop}

\begin{proof} In $V(\gamma)$ (resp. $W(\gamma)$), the highest weight vectors $\mathbb{C}. v_0$ (resp. $\mathbb{C}.w_0$) are the only primitive vectors. 
This implies $V(\gamma)$ and $W(\gamma)$ are simple. 

By standard arguments, a simple finite-dimensional representation of  $\mathcal{U}_q^\pm(sl_2)$ is a quotient of a Verma module, the second point follows.
\end{proof}

\begin{example} (i)
For $\gamma,\beta\in\mathbb{C}^*$, the $\mathcal{U}_q(sl_2)$-modules $V(\gamma)\otimes W(\beta)$ and $W(\beta)\otimes V(\gamma)$ have the same character
 $$\sum_{r,r'\geq 0}[\gamma\beta q^{-2(r+r')}] = \sum_{r\geq 0}\chi(M(\gamma\beta q^{-2r}))$$
where $M(\lambda)$ is the Verma module of $\mathcal{U}_q(sl_2)$ of highest weight $\lambda$. 
For $\gamma,\beta$ so that $\gamma\beta\notin q^{\mathbb{Z}_{\geq 0}}$, these representations are semi-simple and
$$V(\gamma)\otimes W(\beta)\simeq W(\beta) \otimes V(\gamma)  \simeq \bigoplus_{r\geq 0}M(\gamma\beta q^{-2r}).$$
Indeed, the weight space associated to $\gamma\beta q^{-2r}$ has dimension $r + 1$. Hence $e$ is not injective 
on this weight space which contains a primitive vector generating $M(\gamma\beta q^{-2r})$.

(ii) For $\gamma\in \mathbb{C}^*$ and $V$ a fundamental $\mathcal{U}_q(sl_2)$-module of highest weight $[q]$, we have
$$V\otimes V(\gamma)\simeq V(\gamma q) \oplus V(\gamma q^{-1}).$$
\end{example}

\section{Shifted quantum affine algebras}\label{sqaa}

We recall, for $\mu$ in the coweight lattice, the definition of the shifted quantum affine algebra $\mathcal{U}_q^\mu(\hat{\Glie})$ in the sense of Finkelberg-Tsymbaliuk \cite{FT} and its first properties. In the $sl_2$-case, 
we consider evaluation morphisms to the $q$-oscillator algebras of the previous section which
give examples of evaluation representations (Proposition \ref{evalm}). For general types, we prove that for $\mu$ anti-dominant, the 
shifted quantum affine algebra $\mathcal{U}_q^\mu (\hat{\Glie})$ can be realized from a subalgebra of the ordinary quantum affine algebra, the asymptotic algebra introduced in \cite{HJ}. We also prove that for $\mu$ anti-dominant, it
contains a subalgebra isomorphic to the quantum affine Borel subalgebra $\mathcal{U}_q(\hat{\bo})$ of the quantum affine algebra (Proposition \ref{suba}).

\subsection{Definition and first properties}
For $\hat{\Glie}$ an untwisted affine algebra, recall the Drinfeld presentation of the quantum affine algebra
$\mathcal{U}_q(\hat{\Glie})$ (established in \cite{dr, bec, da2}, see for instance \cite{CP}). 

Let $\Lambda = \bigoplus_{i\in I}\mathbb{Z}\omega_i^\vee$ be the coweight lattice (that is the weight lattice of the Langlands dual Lie algebra ${}^L\Glie$).
It contains the set $\Lambda^+$ of dominant coweights $\omega^\vee\in\Lambda$ so that $\alpha_i(\omega^\vee)\geq 0$ for $i\in I$.

Let $\mu_+, \mu_- \in \Lambda$. The shifted quantum affine algebra $\mathcal{U}_q^{\mu_+, \mu_-}(\hat{\Glie})$ is defined in \cite[Section 5]{FT} 
by Drinfeld generators $x_{i,m}^\pm$, $\phi_{i,m}^\pm$, $h_{i,r}$ with $i\in I, m\in\mathbb{Z}, r\in\mathbb{Z}\setminus\{0\}$ and the same relations as for the ordinary quantum affine algebra except that
$$\sum_{r\in\ZZ} \phi_{i,\pm r}^\pm z^{\pm r} =  \phi_i^\pm (z)  =  z^{\mp \alpha_i(\mu_{\pm})} \phi_{i,\mp \alpha_i(\mu_{\pm})}^\pm \text{exp}(\pm (q_i - q_i^{-1}) \sum_{r > 0} h_{i,\pm r} z^{\pm r}),$$
with $\phi_{i,\mp \alpha_i(\mu_{\pm})}^\pm$ which are invertible and satisfy the same relations as the ordinary $k_i^{\pm 1}$. 

Explicitly, for $i,j\in I$, $r,r'\in\mathbb{Z}$, $m\in\mathbb{Z}\setminus\{0\}$, we have
\begin{equation}\label{un}[\phi_{i,r}^\pm,\phi_{j,r'}^\pm] = [\phi_{i,r}^\pm,\phi_{j,r'}^\mp] = 0,\end{equation}
\begin{equation}\label{deux}\phi_{i,-\alpha_i(\mu_+)}^+ x_{j,r}^{\pm} = q_i^{\pm C_{i,j}}x_{j,r}^{\pm} \phi_{i,- \alpha_i(\mu_+)}^+\text{ and }\phi_{i,\alpha_i(\mu_-)}^- x_{j,r}^{\pm} = q_i^{\mp C_{i,j}}x_{j,r}^{\pm} \phi_{i,\alpha_i(\mu_-)}^-,\end{equation}
\begin{equation}\label{hd}[h_{i,m},x_{j,r}^{\pm}] = \pm \frac{1}{m}[m C_{i,j}]_{q_i}  x_{j,m+r}^{\pm},\end{equation}
\begin{equation}\label{trois}[x_{i,r}^+,x_{j,r'}^-] = \delta_{i,j}\frac{\phi^+_{i,r+r'}- \phi^-_{i,r+r'}}{q_i-q_i^{-1}},\end{equation}
\begin{equation}\label{hdd}x_{i,r+1}^{\pm}x_{j,r'}^{\pm} - q^{\pm B_{i,j}}x_{j,r'}^{\pm}x_{i,r+1}^{\pm}
=q^{\pm B_{i,j}}x_{i,r}^{\pm}x_{j,r'+1}^{\pm}-x_{j,r'+1}^{\pm}x_{i,r}^{\pm},\end{equation}
and for $i\neq j$, $r', r_1,\cdots, r_s\in\mathbb{Z}$ where $s = 1 - C_{i,j}$, 
\begin{equation}\label{seq}\underset{\pi\in \Sigma_s}{\sum}\underset{0\leq r \leq s}{\sum}(-1)^r\begin{bmatrix}s\\r\end{bmatrix}_{q_i}x_{i,r_{\pi(1)}}^{\pm}\cdots x_{i,r_{\pi(r)}}^{\pm}x_{j,r'}^{\pm}x_{i,r_{\pi(r+1)}}^{\pm}\cdots x_{i,r_{\pi(s)}}^{\pm}=0.\end{equation}
The relations may be written in terms of currents $x_i^\pm(z) = \sum_{r\in\mathbb{Z}} x_{i,r}^\pm z^r$, $\delta(z) = \sum_{r\in\mathbb{Z}}z^r$ :
$$[\phi_i^\pm(z),\phi_j^\pm(w)] = [\phi_i^\pm(z),\phi_j^\mp(w)] = 0,$$
\begin{equation}\label{phix}\phi_i^{\epsilon}(z) x_j^\pm(w) =   \frac{q^{\pm B_{i,j}}w - z}{w - q^{\pm B_{i,j}}z}  x_j^\pm(w) \phi_i^\epsilon(z) \text{ for $\epsilon = +$ or $-$},\end{equation}
\begin{equation}\label{pmz} [x_i^+(z), x_j^-(w)] = \frac{\delta_{i,j}}{q_i - q_i^{-1}}\left[ \delta \left( \frac{w}{z} \right) \phi_i^+(z)    -  \delta\left( \frac{z}{w} \right)  \phi_i^-(z) \right],\end{equation}
$$ (w - q^{\pm B_{i,j}} z)   x_i^\pm(z) x_j^\pm(w)      =       (q^{\pm B_{i,j}} w - z)            x_j^\pm(w) x_i^\pm(z),$$
$$\underset{\pi\in \Sigma_s}{\sum}\underset{0\leq r \leq s}{\sum}(-1)^r\begin{bmatrix}s\\r\end{bmatrix}_{q_i}x_i^{\pm}(w_{\pi(1)})\cdots x_i^\pm(w_{\pi(r)})x_j^{\pm}(z)x^\pm_i(z_{\pi(r+1)})\cdots x_i^\pm(z_{\pi(s)})=0\text{ for $i\neq j$.}$$

\begin{rem}\label{com} (i) Up to isomorphism, $\mathcal{U}_q^{\mu_+, \mu_-}(\hat{\Glie})$ depends only on $\mu = \mu_+ + \mu_-$, see \cite[Section 5.(i)]{FT}. 
We will simply denote $\mathcal{U}_q^\mu(\hat{\Glie}) = \mathcal{U}_q^{0,\mu}(\hat{\Glie})$.

(ii) For $i\in I$, the product 
$$\phi_{i,-\alpha_i(\mu_+)}^+\phi_{i,\alpha_i(\mu_-)}^-$$ 
is central. The quantum loop algebra $\mathcal{U}_q(\hat{\Glie})$ is the quotient of $\mathcal{U}_q^{0, 0}(\hat{\Glie})$ by identifying 
$\phi_{i,0}^+\phi_{i,0}^-$ with $1$ for $i\in I$, see \cite{FT}.

(iii) The algebra $\mathcal{U}_q^{\mu_+ , \mu_-}(\hat{\Glie})$ has a triangular decomposition analogous to the Drinfeld triangular 
decomposition of $\mathcal{U}_q(\hat{\Glie})$ (see \cite[Proposition 2]{FT} and \cite[Theorem 2]{H}). Each triangular factor can 
be presented by their Drinfeld generators and the relations of $\mathcal{U}_q^{\mu_+ , \mu_-}(\hat{\Glie})$ involving these
generators (there are no hidden relations). The subalgebra generated by the $\phi_{i,m}^\pm$, $(\phi_{i,\alpha_i(\mu_{\pm})}^\pm)^{\mp 1}$, with $i\in I$, $m\in\mathbb{Z}$,  is commutative and called 
the Cartan-Drinfeld subalgebra.

(iv) Consequently, for $\mu_+,\mu_-\in - \Lambda^+$ and $J_\pm = \{i\in I|\alpha_i(\mu_\mp)\neq 0\}$, the $q$-oscillator 
algebra $\mathcal{U}_q^{J_+,J_-}(\Glie)$ of the previous section is a subalgebra of 
$$\mathcal{U}_q^{\mu_+,\mu_-}(\hat{\Glie})/<\phi_{i,0}^+\phi_{i,0}^- = 1, i\notin J_+\cup J_->.$$

(v) The algebra $\mathcal{U}_q^{\mu_+, \mu_-}(\hat{\Glie})$ has a natural $\mathbb{Z}$-grading defined so that 
$$\text{deg}(x_{i,m}^\pm) = \text{deg}(\phi_{i,m}^{\pm}) = m\text{ and }\text{deg}(h_{i,r}) = r
\text{ for  $i\in I$, $m\in\mathbb{Z}$, $r\in\mathbb{Z}\setminus\{0\}$.}$$ 
In particular, for $a\in\CC^*$, there is an algebra automorphism $\tau_a$ of $\mathcal{U}_q^{\mu_+ , \mu_-}(\hat{\Glie})$ such that for $i\in I$, $m\in\mathbb{Z}$ and $r\in\mathbb{Z}\setminus\{0\}$ :
$$\tau_a(x_{i,m}^\pm) = a^m x_{i,m}^\pm\text{ , }\tau_a(\phi_{i,m}^\pm) = a^m \phi_{i,m}^\pm\text{ , }\tau_a(h_{i,r}) = a^r h_{i,r}.$$

(vi) $z^{-1}$ in the notations of \cite{FT} is $z$ here. 

(vii) In type $A$, $RTT$ realizations have been established in \cite{FT, FPT} when $\mu\in - \Lambda^+$.
\end{rem}

\begin{example} Let $i\in I$. For $\mathcal{U}_q^{-\omega_i^\vee,0}(\hat{\Glie})$ (resp. $\mathcal{U}_q^{0 , -\omega_i^\vee}(\hat{\Glie})$),
the modified relations are : 
$$[x_{i,r}^+ , x_{i,- r}^-] = \frac{- \phi_{i,0}^-}{q_i - q_i^{-1}}\text{ (resp. }[x_{i,r}^+ , x_{i,- r}^-] = \frac{\phi_{i,0}^+}{q_i - q_i^{-1}})\text{ for $r \in\mathbb{Z}$,}$$ 
with the definition of the $\phi_{i,  r}^+$ (resp. $\phi_{i, - r}^-$) modified to : 
$$\phi_i^+ (z)  = z \phi_{i,1}^+ \text{exp}(  (q_i - q_i^{-1}) \sum_{r > 0} h_{i, r} z^{ r})$$
$$(\text{resp. }\phi_i^- (z)  = z^{- 1} \phi_{i,- 1}^- \text{exp}( - (q_i - q_i^{-1}) \sum_{r > 0} h_{i,- r} z^{- r})).$$
 \end{example}

\subsection{Antidominant case : relation to asymptotic and Borel algebras}\label{asymprel}

The asymptotic algebra $\tilde{\mathcal{U}}_q(\hat{\Glie})$ is defined in \cite{HJ} as the subalgebra of ordinary quantum affine algebra $\mathcal{U}_q(\hat{\Glie})$ generated by the 
$$x_{i,m}^+\text{ , }k_i^{-1}x_{i,m}^-\text{ , }k_i^{-1}\phi_{i, m}^\pm\text{ , }k_i^{-1}\text{ for $m\in\mathbb{Z}$, $i\in I$.}$$ 
These elements are denoted respectively by $\tilde{x}_{i,m}^+$, $\tilde{x}_{i,m}^-$, $\tilde{\phi}_{i,m}^\pm$, $\kappa_i$.
Note that $\tilde{\phi}_{i,0}^+ = 1$ and $\tilde{\phi}_{i,0}^- = \kappa_i^2$. We will denote 
$$\tilde{\phi}_i^\pm(z) = \sum_{m\geq 0}\tilde{\phi}_{i,\pm m}^\pm z^{\pm m}.$$

Consider an antidominant coweight $\mu\in - \Lambda^+$. The defining relations of the shifted quantum affine algebra 
and the presentation of the asymptotic algebra by generators and relations in \cite[Section 2.2]{HJ} imply the following.

\begin{prop}
The shifted quantum affine $\mathcal{U}_{q}^{\mu}(\hat{\Glie})$ is obtained from the quotient
$$\tilde{\mathcal{U}}_q(\hat{\Glie})/(\kappa_i = \phi_{i,-1}^- = \cdots = \phi_{i,\alpha_i(\mu) + 1}^- = 0,i\in I)$$
by localizing at the $\phi_{i,\alpha_i(\mu)}$ for all $i\in I$, and by adding, for $i\in I$ so that $\alpha_i(\mu) < 0$, the generators $k_i^{\pm 1}$ with the quasi-commutation relations (\ref{un}), (\ref{deux}).
\end{prop}

Let $\mathcal{U}_q(\hat{\bo})$ be the quantum affine Borel subalgebra of $\mathcal{U}_q(\hat{\Glie})$, in the sense of Drinfeld-Jimbo. It is generated by the subset $e_i$, $k_i^{\pm 1}$, $i\in I \sqcup\{0\}$, of the Chevalley generators of $\mathcal{U}_q(\hat{\Glie})$.

From the last Proposition, we have a natural algebra morphism 
$$\mathcal{U}_q(\hat{\bo})\cap \tilde{\mathcal{U}}_q(\hat{\Glie})\rightarrow \mathcal{U}_{q}^{\mu}(\hat{\Glie}).$$
But $\mathcal{U}_q(\hat{\bo})$ is obtained from $\mathcal{U}_q(\hat{\bo})\cap \tilde{\mathcal{U}}_q(\hat{\Glie})$ 
by localization at the $\kappa_i$, $i\in I$, and so the morphism extends to an algebra morphism
$$I_\mu : \mathcal{U}_q(\hat{\bo})\rightarrow \mathcal{U}_{q}^{\mu}(\hat{\Glie}).$$

\begin{prop}\label{suba} If $\mu\in - \Lambda^+$, then $\mathcal{U}_q^\mu (\hat{\Glie})$ contains a subalgebra
isomorphic to $\mathcal{U}_q(\hat{\bo})$.
\end{prop}

\begin{proof} Let us consider the triangular decomposition of $\mathcal{U}_q(\hat{\Glie})$ 
as in (iii) of Remark \ref{com}. It induces a triangular decomposition of the quantum affine Borel algebra
$\mathcal{U}_q(\hat{\bo})$ (this follows from \cite{bec, da}, see \cite[Section 2.3]{HJ} for instance) :
$$\mathcal{U}_q(\hat{\bo})\simeq \mathcal{U}_q^-(\hat{\bo})\otimes\mathcal{U}_q^0(\hat{\bo})\otimes\mathcal{U}_q^+(\hat{\bo}).$$
Now let us consider the analogous triangular decomposition of $\mathcal{U}_q^\mu (\hat{\Glie})$ :
$$\mathcal{U}_q^\mu (\hat{\Glie}) \simeq \mathcal{U}_q^{\mu,-} (\hat{\Glie})
\otimes \mathcal{U}_q^{\mu,0} (\hat{\Glie})\otimes \mathcal{U}_q^{\mu,+} (\hat{\Glie})$$
as in (iii) of Remark \ref{com}. The triangular decomposition statement 
is not only a linear isomorphism with the tensor product of the three algebras. In addition, it states that each of the three algebras 
can be presented by Drinfeld generators and the relations involving these generators only 
(no additional hidden relations are necessary).

Hence each triangular factor, $\mathcal{U}_q^{\mu,-} (\hat{\Glie})$, $\mathcal{U}_q^{\mu,0} (\hat{\Glie})$, $\mathcal{U}_q^{\mu,+} (\hat{\Glie})$, is isomorphic, as an algebra, 
to the corresponding triangular factor in $\mathcal{U}_q^0(\hat{\Glie})$. Hence $\mathcal{U}_q^{\mu,-} (\hat{\Glie})$
(resp. $\mathcal{U}_q^{\mu,0} (\hat{\Glie})$, $\mathcal{U}_q^{\mu,+} (\hat{\Glie})$) contains a
subalgebra isomorphic to $\mathcal{U}_q^-(\hat{\bo})$ (resp. $\mathcal{U}_q^0(\hat{\bo})$,
$\mathcal{U}_q^+(\hat{\bo})$). By construction, these three isomorphism are given by the restrictions of $I_\mu$ 
to the corresponding triangular factors. This implies that 
$I_\mu$ is injective as it is injective on each triangular factor. Consequently these three subalgebras generate a subalgebra isomorphic to $\mathcal{U}_q(\hat{\bo})$.

%As $\mu$ is anti-dominant, the commutations relations between 
%the Drinfeld generators of positive degrees are the same in $\mathcal{U}_q^0 (\hat{\Glie})$ and in
%$\mathcal{U}_q^\mu (\hat{\Glie})$. Hence we have a surjective ring morphism from $\mathcal{U}_q(\hat{\bo})$ 
%to the subalgebra of $\mathcal{U}_q^\mu(\hat{\Glie})$ generated by the three subalgebras. 

\end{proof}

The images $I_\mu(e_i)$ of the Chevalley generators $e_i$, $i\in I \sqcup \{0\}$, of $\mathcal{U}_q(\hat{\bo}))$, now seen in $\mathcal{U}_q^{\mu}(\hat{\Glie})$, will be denoted by the same symbols $e_i$.

\begin{example} For $\mu\in - \Lambda^+$, the subalgebra of $\mathcal{U}_q^\mu (\hat{sl_2})$ generated by $e_1 = x_{1,0}^+$, 
$e_0 = (\phi_{1,0}^+)^{-1}x_{1,1}^-$, $(\phi_{1,0}^+)^{\pm 1}$ is isomorphic to $\mathcal{U}_q(\hat{\bo})\subset \mathcal{U}_q(\hat{sl}_2)$.
\end{example}

\begin{rem} The algebra $\mathcal{U}_q(\hat{\bo})$ contains the $x_{i,m}^+$, $x_{i,r}^-$, 
$\phi_{i,m}^+$, $(\phi_{i,0}^+)^{-1}$ for $i\in I$, $m\geq 0$, $r > 0$, but is not generated by these elements (except in the $sl_2$-case).
\end{rem}

The same argument as for the proof of Proposition \ref{suba} implies the following.

\begin{prop}\label{mubo}
Let $\mu\in \Lambda$. The subalgebras of $\mathcal{U}_q(\hat{\bo})$ and of $\mathcal{U}_q^{\mu}(\hat{\Glie})$ generated by : 
$$x_{i,r}^+\text{ , } x_{i,s}^-\text{ , }\phi_{i, r}^+\text{ , }(\phi_{i,0}^+)^{- 1}$$
for $i\in I$, $r\geq 0$, $s > \text{Max}(0,\alpha_i(\mu))$, are isomorphic. 
\end{prop}

We will denote this algebra by $\mathcal{U}_q^\mu(\hat{\bo})$.

\subsection{Example - the algebras $\mathcal{U}_q^\pm(\hat{sl}_2)$}\label{evalex}

In the $sl_2$-case, we will simply denote 
$$\mathcal{U}_q^+(\hat{sl}_2) = \mathcal{U}_q^{0,-\omega_1^\vee}(\hat{sl}_2)\text{ and }\mathcal{U}_q^-(\hat{sl}_2) = \mathcal{U}_q^{-\omega_1^\vee,0}(\hat{sl}_2).$$ 
Note that we have $e_1 = x_{1,0}^+$ and $e_0 = k_1^{-1}x_{1,0}^-$ in $\mathcal{U}_q(\hat{\bo})$.

It is well known there exist evaluation morphisms $\mathcal{U}_q(\hat{\bo})\rightarrow \mathcal{U}_{q}^\pm(sl_2)$ 
(see \cite{blz}) but they can not be extended to $\mathcal{U}_q(\hat{sl}_2)$. We prove it can be extended to a shifted quantum affine algebra.

\begin{prop}\label{evalm} For $a\in\CC^*$, we have evaluation morphisms 
$$ev_a^\pm : \mathcal{U}_q^\pm(\hat{sl}_2)\rightarrow \mathcal{U}_{q,loc}^\pm(sl_2)$$ 
defined for $m\in \ZZ$, $r > 0$ by :
$$ev_a(x_m^+) = a^m e k^m\text{ , }ev_a(x_m^-) = a^m k^m f,$$
$$ev_a(\phi_r^+) = a^r[ek^r,f](q - q^{-1})\text{ , }ev_a(\phi_{-r}^-) = a^{-r} [ ek^{-r}, f] (q^{-1} - q),$$
$$ev_a(\phi_0^+) = \delta_{\pm, +}k\text{ , }ev_a(\phi_0^-) = \delta_{\pm,-}k^{-1}.$$
\end{prop}

\begin{rem} (i) Only the last two formulas differ from the Jimbo evaluation morphism \cite{J}
$$\mathcal{U}_q(\hat{sl}_2)\rightarrow \mathcal{U}_q(sl_2).$$

(ii) Closely related morphisms are introduced in \cite[Proposition 3.85 and Remark 3.89]{FPT} 
for  $\Glie = gl_{n+1}$ using an $RTT$-presentation.
\end{rem}

\begin{proof} First we have 
$$ev_a^\pm (\phi_{\mp 1}^\mp) = - q^{\pm 1}a^{\mp 1} ( q - q^{-1})^2 C_\pm k^{\mp 1}$$
which is invertible in $\mathcal{U}_{q,loc}^\pm(sl_2)$. 
We have to check the other defining relations of $\mathcal{U}_q^\pm(\hat{sl}_2)$ are compatible with the
defining formulas of $ev_a$. 
From the result for the standard evaluation morphism, this is clear for the formulas (\ref{un}), (\ref{deux}), (\ref{hd}), (\ref{trois}), (\ref{hdd})
except of $r + r' = 0$ (note that there are no Drinfeld-Serre relations (\ref{seq}) in the $sl_2$-case).
In the last case, we have for $r\in\ZZ$
$$[x_r^+ , x_{-r}^-] = \frac{\phi_0^+ - \phi_0^-}{q - q^{-1}}\mapsto \frac{\delta_{\pm, +}k - \delta_{\pm,-}k^{-1}}{q - q^{-1}} = [e,f] = [a^ e k^r,a^{-r}k^{-r}f]
=[ev_a(x_r^+), ev_a(x_{-r}^-)].$$
\end{proof}

\begin{example}\label{exval} Recall the Verma module $V(\gamma)$ of $\mathcal{U}_q^+(sl_2)$. Its evaluation at $q^2\gamma^{-1}$ satisfies
$$e_0.v_r = q^{-r + 2}\frac{[r+1]_q}{q - q^{-1}}v_{r+1}\text{ , }f_0.v_r = \gamma q^{-2}v_{r - 1}.$$
By twisting by the automorphisms $\tau_a$, we get a continuous family of representations that we denote by $L_{\gamma,a}^-$.
The action of $\phi_\pm(z)\in\mathcal{U}_q^+(\hat{sl_2})[[z^{\pm 1}]]$ is given by
$$\phi^\pm(z).v_j = \gamma q^{-2j} \frac{1 - q^2 z a}{(1 - q^{ 2 - 2j} az)(1 - q^{-2j} a z)} .v_j\in L_{\gamma,a}^-[[z^{\pm 1}]].$$
This matches formulas in \cite[Section 4.1]{HJ} : the restriction to $\mathcal{U}_q(\hat{\bo})$ is $[\gamma]\otimes L^{\bo, -}_a$
where $L^{\bo,-}_a$ is a prefundamental representation and $[\gamma]$ is $1$-dimensional (see Section \ref{remcato}). Note that 
for $m\geq 0$, $r > 0$, we recover the action on $L_{\gamma,a}^-$ from the action on $L_{1,a}^-$ by the twist 
$$x_r^- \mapsto \gamma x_r^-\text{ , }x_m^+\mapsto x_m^+\text{ , }\phi_m^\pm\mapsto \gamma \phi_m^\pm.$$
\end{example}

\begin{rem}\label{finc} The action of $\mathcal{U}_q(\hat{\bo})$ on $L^{\bo, -}_{1,a}$ can not be extended to $\mathcal{U}_q(\hat{sl}_2)$ (see \cite{HJ}). 
This implies there is no embedding $\mathcal{U}_q(\hat{sl}_2) \subset \mathcal{U}_q^+(\hat{sl}_2)$ 
which induces the embedding $\mathcal{U}_q(\hat{\bo})\subset\mathcal{U}_q^+(\hat{sl}_2)$ above. 
\end{rem}

\section{Category $\mathcal{O}^{sh}$}\label{catosh}

%It is known since the work in \cite{blz} that certain representations of the $q$-oscillator algebra associated to $sl_2$ give rise to representations of the quantum affine Borel subalgebra $\mathcal{U}_q(\hat{\bo})$ of the quantum affine algebra $\mathcal{U}_q(\hat{sl}_2)$ (see Section \ref{evalex}). For general untwisted types, 
%the category $\mathcal{O}$ of representations of the quantum affine Borel subalgebra of a quantum affine algebra was 
%introduced and studied in \cite{HJ}. Some representations in this category extend to a representation of the 
%full quantum affine algebra (so belong to the category $\hat{\mathcal{O}}$ of \cite{H2}), but many do not. 

%It was first observed in \cite{HJ} that for some of these representations, the structure of representation of the quantum affine Borel algebra can 
%be extended to a larger algebra. It is called the asymptotic algebra $\tilde{\mathcal{U}}_q(\hat{\Glie})$ and we will see this algebra is related to certain shifted quantum affine algebras. In the Yangian case, it was first observed in \cite{Z} that for certain simple representations in an analog of the category $\mathcal{O}$ (the prefundamental representations), the action can be extended to a shifted Yangian.

In this section we introduce categories $\mathcal{O}_\mu$ of representations of shifted quantum affine algebras and classify their simple objects (Theorem \ref{param}). The category $\mathcal{O}^{sh}$ is the sum of these abelian categories.
We also study shift functors induced by shift homomorphisms.

\subsection{Reminder - the category $\mathcal{O}$ for the quantum affine Borel algebra}\label{remcato} 
The category $\mathcal{O}$ of representations of $\mathcal{U}_q(\hat{\bo})$ is defined in \cite{HJ} as 
an analog of the ordinary category $\mathcal{O}$ (see \cite{kac}). It is the category of $\mathcal{U}_q(\hat{\bo})$-modules $V$ which are the sum of their weight spaces
$$V_\omega = \{v\in V| k_i. v = \omega(i) v,\forall i\in I\}\text{ where }\omega\in \tb^*,$$
such that the $V_\omega$ are finite-dimensional and there are a finite number of $\omega_1,\cdots, \omega_s\in \tb^*$ so that 
the weights of $V$, that is the $\omega$ so that $V_\omega\neq 0$, belong to $D(\omega_1)\cup\cdots \cup D(\omega_s)$ where
$$D(\omega_i) = \{\omega\in \tb^*|\omega\preceq \omega_i\}.$$

A series $\Psib=(\Psi_{i, m})_{i\in I, m\geq 0}$ of complex numbers such that 
$\Psi_{i,0}\neq 0$ for all $i\in I$ is called an $\ell$-weight. We also denote 
$$\Psib = (\Psi_i(z))_{i\in I}\text{ where }\Psi_i(z) = \underset{m\geq 0}{\sum} \Psi_{i,m} z^m.$$
A $\mathcal{U}_q(\hat{\mathfrak{b}})$-module $V$ is of highest $\ell$-weight $\Psib$ if there is $v\in V$ with $V =\mathcal{U}_q(\hat{\mathfrak{b}}).v$ and 
\begin{align*}
e_i\, v=0\quad (i\in I)\,,
\qquad 
\phi_{i,m}^+v=\Psi_{i, m}v\quad (i\in I,\ m\ge 0)\,.
\end{align*}
We recall here that the $\phi_{i,m}^+$, with $i\in I$, $m\geq 0$, are indeed in $\mathcal{U}_q(\hat{\bo})$.

For any $\ell$-weight $\Psib$, there exists a unique simple 
highest $\ell$-weight $\mathcal{U}_q(\hat{\bo})$-module $L^{\bo}(\Psib)$ of highest $\ell$-weight
$\Psib$. For example, for $i\in I$ and $a\in\CC^\times$, we have the prefundamental representations
\begin{align}
L_{i,a}^{\bo,\pm }= L^\bo(\Psib_{i,a}^{\pm 1})
\quad \text{where}\quad 
(\Psib_{i,a})_j(z) = \begin{cases}
(1 - za) & (j=i)\,,\\
1 & (j\neq i)\,.\\
\end{cases} 
\label{fund-rep}
\end{align}
The superscript $\bo$ is to distinguish with representations of shifted quantum affine algebras we will study. 
For $\omega\in \tb^*$, we have the $1$-dimensional representation, called constant representation
$$[\omega] = L^\bo(\Psib_\omega)
\quad \text{where}\quad 
(\Psib_\omega)_i(z) = \omega(i) \quad (i\in I).$$
\newcommand{\mfr}{\mathfrak{r}}
Let $\mfr$ be the group of rational $\ell$-weights $\Psib$, so that the $\Psi_i(z)$ are rational for $i\in I$ 
(the group structure is given by the ordinary multiplication of rational fractions).
\begin{thm}\label{class}\cite{HJ} The simple modules in the category 
$\mathcal{O}$ are the $L^\bo(\Psib)$ for $\Psib\in \mfr$.
\end{thm}
For $V$ in the category $\mathcal{O}$ and $\Psib$ an $\ell$-weight, we have the $\ell$-weight space 
\begin{align}
V_{\Psibs} =
\{v\in V\mid
\exists p\geq 0, \forall i\in I, 
\forall m\geq 0,  
(\phi_{i,m}^+ - \Psi_{i,m})^pv = 0\}.
\label{l-wtsp} 
\end{align}
A representation in the category $\mathcal{O}$ is the direct sum of its $\ell$-weight spaces. Moreover : 
\begin{thm}\cite{HJ} For $V$ in category $\mathcal{O}$, $V_{\Psib}\neq 0$ implies $\Psib\in\mfr$.\end{thm}

\subsection{Reminder - the category $\hat{\mathcal{O}}$ for the quantum affine algebra}
The category $\hat{\mathcal{O}}$ of $\mathcal{U}_q(\hat{\Glie})$-modules which are in the category $\mathcal{O}$ as $\mathcal{U}_q(\hat{\bo})$-modules 
was introduced in \cite{H}. The categories $\hat{\mathcal{O}}$ and $\mathcal{O}$ are monoidal (with respect to the ordinary Drinfeld-Jimbo coproduct) and there is a forgetful functor
$$\hat{f} : \hat{\mathcal{O}} \rightarrow \mathcal{O}.$$
The following was established by Bowman \cite{bo} and Chari-Greenstein \cite{CG} for finite-dimensional representations. 
The proof in \cite[Proposition 3.5]{HJ} can be adapted to the category $\hat{\mathcal{O}}$.

\begin{prop} For $V$ a simple representation in $\hat{\mathcal{O}}$, $\hat{f}(V)$ is simple.\end{prop}

\begin{rem}\label{remind}
(i) It is proved in \cite[Lemma 14]{H} and \cite[Lemma 3.9]{HJ} that for $i\in I$, the action of $\phi_i^+(z)$ and $\phi_i^-(z)$ coincide on a representation $V$ in $\hat{\mathcal{O}}$, 
seen as rational operators on each weight space (it follows directly from the existence of a polynomial $P(z)$ so that $P(z)(\phi_i^+(z) - \phi_i^-(z)) = 0$; 
this is also proved in \cite[Section 3.6]{GTL}). 

(ii) As $\phi_i^+(z)$ (resp. $\phi_i^-(z)$) is regular at $0$ (resp. $\infty$), 
this also implies that $\phi_i^+(z)$ has degree $0$ and that $\phi_i^+(0)\phi_i^+(\infty) = \text{Id}$. 
In particular, the simple representations $V$ in $\hat{\mathcal{O}}$ are parametrized by the highest $\ell$-weight $\Psib$ of $\hat{f}(V)$ : 
it is rational of degree $0$ with $\Psi_i(0)\Psi_i(\infty) = 1$ for $i\in I$. The converse statement is true by \cite{my} : these rational $\ell$-weights 
parametrize the simple representations in $\hat{\mathcal{O}}$.
\end{rem}

\subsection{The category $\tilde{\mathcal{O}}$ for the asymptotic algebra}
Recall the asymptotic algebra $\tilde{\mathcal{U}}_q(\hat{\Glie})$ which is a subalgebra of
$\mathcal{U}_q(\hat{\Glie})$ (see Section \ref{asymprel}). For our purposes, we will consider representations of $\tilde{\mathcal{U}}_q(\hat{\Glie})$ with $0$ as a possible eigenvalue of $\kappa_i$.
Hence we have to modify the axioms to introduce the proper notion of category $\mathcal{O}$ for this algebra.

A representation $V$ of $\tilde{\mathcal{U}}_q(\hat{\Glie})$ is said to be 
$\tb^*$-graded if there is a decomposition into a direct sum of finite-dimensional subspaces (the weight spaces)
$V = \bigoplus_{\omega\in \tb^*} V^{(\omega)}$ such that 
$$\tilde{x}_{i,r}^\pm V^{(\omega)}\subset V^{(\omega \overline{\alpha_i}^{\pm 1})}\text{ , }
\tilde{\phi}_{i,\pm m}^\pm V^{(\omega)}\subset V^{(\omega)}\text{ , }
\kappa_i V^{(\omega)}\subset V^{(\omega)}
\text{ for any $\omega\in \tb^*$, $i\in I$, $r\in\ZZ, m\geq 0$.}
$$
The weights of $V$ are the $\omega$ so that $V^{(\omega)}\neq 0$. 

The same argument as in Remark \ref{remind} (i) shows that for each $i\in I$, the action of $\tilde{\phi}_i^\pm(z)$ are rational on weight spaces and coincide as rational operators. The representation $V$ is the direct sum of its $\ell$-weight spaces 
corresponding to pseudo-eigenvalues $\Psi_i(z)$ of the $\tilde{\phi}_i^+(z)$. 
Besides, for $i\in I$, we have $\text{deg}(\Psi_i)\leq 0$ from the development of the rational function at $\infty$ and 
$\Psi_i(0) = 1$ from the development at $0$. Hence we have
$$V = \bigoplus_{\mu\in - \Lambda^+} V_{(\mu)}$$
where for $\mu\in - \Lambda^+$, $V_{(\mu)}$ is the sum of the $\ell$-weight spaces $V_{\Psib}$ with $\text{deg}(\Psi_i) = \alpha_i(\mu)$. 

For $\omega\in\tb^*$ and $\mu\in - \Lambda^+$, we set
$$V_{(\mu,\omega)} = \{v\in V_{(\mu)}|\tilde{\phi}_{i,\alpha_i(\mu)}^-.v = \omega(i) v, \forall i\in I\}.$$

\begin{defi}\label{o} A $\tb^*$-graded representation $V$ of $\tilde{\mathcal{U}}_q(\hat{\Glie})$ is said to be in the category $\tilde{\mathcal{O}}$ if : 

%(i) for any $i\in I$, there is $b_i\leq 0$ maximal so that the image of $\tilde{\phi}_{i,b_i}^-$ in $\text{End}(V)$ is non-zero. Moreover $\tilde{\phi}_{i,b_i}^-$ is invertible and diagonalizable on $V$.

%(ii) For each $\omega\in \tb^*$, the space $V_\omega=\{v\in V|\tilde{\phi}_{i,b_i}^-.v = \omega(i)v, \forall i\in I\}$ is finite-dimensional,

(i) We have $V = \bigoplus_{\mu\in - \Lambda^+, \omega\in\tb^*} V_{(\mu,\omega)}$.

(ii) The spaces $V_{(\omega)} = \bigoplus_{\mu\in - \Lambda^+} V_{(\mu,\omega)}$ are finite-dimensional.

(iii) There are a finite number of elements $\omega_1,\cdots, \omega_s\in \tb^*$, so that $V^{(\omega)}\neq 0$ or $V_{(\omega)}\neq 0$ implies $\omega\in\bigcup_{1\leq j\leq s} D(\omega_j)$.
\end{defi}

As for the category $\mathcal{O}$, a simple representation in $\tilde{\mathcal{O}}$ is determined up to isomorphism 
by its highest $\ell$-weight $\Psib$. It is rational, with $\text{deg}(\Psi_i)\geq 0$ and $\Psi_i(0) = 1$ for any $i\in I$.
We will prove that such $\Psib$ parametrize the simple representations in $\tilde{\mathcal{O}}$ (see Theorem \ref{simpasym}).

It is proved in \cite[Section 2.4]{HJ} that a $\tb^*$-graded\footnote{It is proved in \cite{HJ} for $Q$-graded representations, but the proof also works for $\tb^*$-graded representations.} representation $V$ gives a representation of $\mathcal{U}_q(\hat{\bo})$ such that 
\begin{align}
e_i\, v = \tilde{x}_{i,0}^+v,\quad e_0\, v = y\, v,
\quad
k_iv = \omega(i) v\,
\quad (i\in I,\ v\in V^{(\omega)})\,,
\label{vsigma1}
\end{align}
where $y\in \tilde{\mathcal{U}}_q(\hat{\Glie})$ is a certain distinguished element defined via iterated quantum brackets.
This defines a functor
$$\tilde{f} : \tilde{\mathcal{O}}\rightarrow \mathcal{O}.$$
This is how the prefundamental representations $L_{i,a}^{\mathfrak{b},-}$ are constructed in \cite{HJ}.

\begin{example} For $i\in I$, $a\in\mathbb{C}^*$, the prefundamental representation $L_{i,a}^{\bo,-}$ of $\mathcal{U}_q(\hat{\bo})$ is 
in the image of a module in $\tilde{\mathcal{O}}$ by the functor $\tilde{f}$ (see \cite{HJ}). 
\end{example}

\begin{example}\label{exom} Monoidal subcategories $\mathcal{O}^\pm$ of $\mathcal{O}$ were introduced in \cite{HL} in the context of monoidal categorification of cluster algebras. $\mathcal{O}^\pm$ is the subcategory of representations in $\mathcal{O}$ whose simple constituents have highest $\ell$-weight which is a product of various $\Psib_{i,a}\Psib_{i,aq_i^2}^{- 1}$, $\Psib_{i,a}^{\pm 1}$, $[\omega]$ for various $i\in I$, $a\in\CC^*$, $\omega\in \tb^*$. By \cite[Section 7.2]{HL}, the simple representations in $\mathcal{O}^-$ are in the image of the functor $\tilde{f}$.
\end{example}

\subsection{The category $\mathcal{O}_\mu$ for the shifted quantum affine algebra}\label{debo}
For $\mu_+, \mu_-\in \Lambda$, we introduce the following category.

\begin{defi} The category $\mathcal{O}_{\mu_+,\mu_-}$ 
is the category of $\mathcal{U}_q^{\mu_+,\mu_-}(\hat{\Glie})$-modules 
$$V = \bigoplus_{\omega\in \tb^*} V_{\omega}^+ = \bigoplus_{\omega\in \tb^*} V_{\omega}^-$$
where for $\omega\in \tb^*$, the weight spaces
$$V_\omega^+ = \{v\in V | \phi_{i,- \alpha_i(\mu_+)}^+.v = \omega(i) v, \forall i\in I\}\text{ , }V_\omega^- = \{v\in V | \phi_{i, \alpha_i(\mu_-)}^-.v = (\omega(i))^{- 1} v, \forall i\in I\}$$
are finite-dimensional and there are a finite number of $\omega_1,\cdots, \omega_s\in \tb^*$ so that 
$V_\omega^+ \neq \{0\}$ or $V_\omega^- \neq \{0\}$ implies 
$$\omega \in D(\omega_1)\cup\cdots \cup D(\omega_s).$$ 
\end{defi}

\begin{rem} (i) In general $V_\omega^+$ and $V_{\omega}^-$ do not coincide as $\phi_{i,- \alpha_i(\mu_+)}^+$ and 
$\phi_{i, \alpha_i(\mu_-)}^-$ are not inverse one to each other.

(ii) By (i) in Remark \ref{com}, this category depends only on $\mu = \mu_+ + \mu_-$. We will be simply denote $\mathcal{O}_\mu = \mathcal{O}_{0,\mu}$.

(iii) By (ii) in Remark \ref{com}, the category $\hat{\mathcal{O}}$ is a full subcategory of $\mathcal{O}_0$.
\end{rem}

\begin{prop}\label{zero} Let $V$ be a representation in $\mathcal{O}_\mu$ (or in $\tilde{\mathcal{O}}$). 
For each weight space of $V$, there is a non-zero polynomial $P(z)$ so that for 
any $i\in I$, $P(z)(\phi_i^+(z) - \phi_i^-(z))$ and $P(z)x_i^\pm(z)$ are zero 
on this weight space.  The action of $\phi_i^+(z)$ and $\phi_i^-(z)$ are rational of degree $\alpha_i(\mu)$
on this weight space and coincide as rational operators.
\end{prop}

\begin{proof} The same argument as in Remark \ref{remind} (i) shows that for each $i\in I$, the action of $\phi_i^+(z)$ and 
$\phi_i^-(z)$ are rational on weight spaces of a representation in $\mathcal{O}_\mu$ and coincide as rational operators. In particular on each
weight space $\phi_i^+(z)$ is equivalent to $\phi_{i,\alpha_i(\mu)}^- z^{\alpha_i(\mu)}$ when $z\rightarrow \infty$, which implies
the degree is $\alpha_i(\mu)$ as $\phi_{i,\alpha_i(\mu)}^-$ is invertible. The statement for 
the $x_i^\pm(z)$ is proved as in \cite[Proposition 3.8]{H2} (it is proved there under the assumption the representation
is integrable, but the fact that weight spaces are finite-dimensional is only used there; an analogous result was also obtained
in \cite{BK}).
\end{proof}

As above, the representations in $\mathcal{O}_\mu$ are the direct sum of their $\ell$-weight spaces 
corresponding to pseudo-eigenvalues of the $\phi_i^+(z)$.
The simple representations are determined up to isomorphism by their highest 
$\ell$-weight $\Psib$ which is rational with $\text{deg} (\Psi_i) = \alpha_i(\mu)$.
We will denote by $\mathfrak{r}_\mu$ the set of such $\ell$-weights :
$$\mathfrak{r}_\mu = \{\Psib = (\Psi_i(z))_{i\in I} \in \mathfrak{r}|  \text{deg}(\Psi_i(z)) = \alpha_i(\mu) \}.$$

A representation in $\mathcal{O}_\mu$ is said to be of highest $\ell$-weight $\Psib\in\mathfrak{r}_\mu$ if it is generated 
by a vector $v$ such that 
$$x_{i,m}^+.v = 0\text{ and }\phi_{i,m}^\pm.v = \Psi_{i,m}^\pm v\text{ for $i\in I$, $m\in \mathbb{Z}$,}$$ 
where
$$\Psi_i(z) = \sum_{m\geq 0} \Psi_{i,m}^+ z^m = \sum_{m\geq - \alpha_i(\mu)} \Psi_{i,-m}^- z^{-m} \in\mathbb{C}(z).$$
Here, the second and third expressions are just expansions of the same rational function in $z$ and $z^{-1}$ respectively.

\begin{cor}\label{still} A simple representation in the category $\mathcal{O}_\mu$, $\mu \in - \Lambda^+$ (resp. in the category $\tilde{\mathcal{O}}$) is simple as a representation of $\mathcal{U}_q(\hat{\bo})$.\end{cor}

\begin{proof} Consider a simple representation $V$ in the category $\mathcal{O}_\mu$. From Proposition \ref{zero}, 
we have established that the action of the operators $x_{i,m}^+$ (resp. $x_{j,m}^-$) for $m\in\mathbb{Z}$ 
are determined by the action of these operators for $m\geq 1$, which are in $\mathcal{U}_q(\hat{\bo})$. Hence, $V$ is generated by its highest
weight vector and has no other primitive vector as a representation of $\mathcal{U}_q(\hat{\bo})$. So it is simple. \end{proof}

\begin{thm}\label{param}
For $\mu\in\Lambda$, the simple modules in the category $\mathcal{O}_{\mu}$ are parametrized by $\mathfrak{r}_\mu$.
\end{thm}

\begin{proof} We have seen in Section \ref{debo} that for $i\in I$, a simple representation in the category $\mathcal{O}_\mu$ has a highest $\ell$-weight $\Psib$ satisfying $\text{deg}(\Psi_i) = \alpha_i(\mu)$ and that $\phi_i^+(z)$, $\phi_i^-(z)$ coincide as rational operators on this representation. 
So we have to prove that there exists a simple representation in the category $\mathcal{O}_\mu$ for each such $\ell$-weight $\Psib$.
In contrast to the case of the category $\mathcal{O}$, we have all Drinfeld generators in $\mathcal{U}_q^{\mu}(\hat{\Glie})$. 
So we do not have to use the strategy in \cite{HJ} (asymptotic representation theory), but 
we can use arguments as for finite-dimensional representations in \cite{CP} 
(see also \cite[Theorem 3.6]{my}). We consider a representation $L$ of highest $\ell$-weight $\Psib$ 
(such a representation can be constructed from a Verma module of highest $\ell$-weight as in \cite{HJ} for instance). 
It suffices to prove its weight spaces $L_{\omega'}$, $\omega'\in\tb^*$, are finite-dimensional. 
Let $\omega = \Psib(0)\in\tb^*$ be the highest weight of $L$. 
As in \cite[Section 5, PROOF of (b)]{CP}, this is proved by induction on the height of 
$\omega' \omega^{-1}$ factorized as a product of simple roots. 
The first step in the proof is to establish that for any $j\in I$, $L_{\omega\overline{\alpha_j}^{-1}}$ is finite-dimensional 
(the induction starts on weights $\omega'$ so that $\omega' \omega^{-1}$ has height $2$).
By the properties of $\Psib$, there is a non-zero polynomial $P(z)$ such that for any $i\in I$, the operator $P(z) (\phi_i^+(z) - \phi_i^-(z))$ is $0$ on $L_\omega$.
For $i,j\in I$ and $s\in\mathbb{Z}$, by the relation (\ref{pmz}), 
$$x_{i,s}^+ (P(z)x_j^-(z)) = 0$$
on $L_\omega$. As $L$ is simple, we get $P(z)x_j^-(z) = 0$ on $L_\omega$. 
This implies that $L_{\omega\overline{\alpha_j}^{-1}}$ is finite-dimensional.
We finish the proof word by word as in \cite[Section 5, PROOF of (b)]{CP}.
\end{proof}

\begin{example}\label{pospre} For $i\in I, a\in\mathbb{C}^*$, we have the positive and negative prefundamental representations $L_{i,a}^\pm = L(\Psib_{i,a}^{\pm 1})$ 
in the category $\mathcal{O}_{\pm \omega_i^\vee}$. The representation $L(\Psib_{i,a})$ is one-dimensional, with the action of the $x_{j,m}^\pm$ equal to $0$ for $j\in I$, $m\in\mathbb{Z}$, and 
$$\phi_j^+(z) = 1 - a z\delta_{i,j}\text{ , }\phi_i^-(z) = z ( z^{-1} - a \delta_{i,j}).$$ 
The structure of $\mathcal{U}_q^{- \omega_i^\vee}(\hat{\Glie})$-module of $L_{i,a}^-$ extends the structure of $\mathcal{U}_q(\hat{\bo})$-module.
This generalizes the result in the $sl_2$-case obtained in terms of evaluation morphisms in Example \ref{exval} ($L_{\gamma,a}^-$ there is $L(\gamma (1- za)^{-1})$). 
In the case of shifted Yangians, these examples were first discussed in \cite{Z}.
It was first noted in \cite{HJ} that the action of $\mathcal{U}_q(\hat{\bo})$ can be extended to the asymptotic algebra, but not to the whole quantum affine algebra : we understand now that the correct framework
is given by shifted quantum affine algebras.
\end{example}

We now define the direct sum of the abelian categories
$$\mathcal{O}^{sh} = \bigoplus_{\mu \in \Lambda} \mathcal{O}_\mu.$$
By Theorem \ref{param}, the simple objects in $\mathcal{O}^{sh}$ are parametrized by $\mathfrak{r}$.

\subsection{Shift functors}\label{reps}

 A shift homomorphism is introduced in \cite[Section 10.(vii)]{FT} as an analog of the morphism defined 
for shifted Yangians in \cite{fkprw}. For $\mu\in \Lambda$, $\mu'\in -\Lambda^+$ and $a\in\mathbb{C}^*$ there is an algebra morphism
$$\iota_{\mu, \mu',a} : \mathcal{U}_q^\mu(\hat{\Glie}) \rightarrow \mathcal{U}_q^{\mu + \mu'}(\hat{\Glie})$$
defined by the following (for $i\in I$) :
\begin{equation}\label{iot}x_i^+(z)\mapsto x_i^+(z)\text{ , }x_i^-(z)\mapsto (1 - az)^{-\mu'(\alpha_i^\vee)} x_i^-(z)\text{ , }\phi_i^\pm(z)\mapsto (1 - az)^{-\mu'(\alpha_i^\vee)} \phi_i^\pm(z).\end{equation}
This is obtained from the shift homomorphism in \cite{FT} after conjugating by the change of variables $z\mapsto az$.

\begin{prop}\cite{FT}\label{iinj} The shift homomorphism $\iota_{\mu, \mu',a}$ is injective.
\end{prop}

\begin{rem} (i) This is proved explicitly in type $A$ in \cite[Proposition I.4]{FT}, but Tsymbaliuk explained to the author
the same argument, combined with embedding into shuffle algebras, gives the result for general types.

(ii) Consequently, for $\mu\in -\Lambda^+$, 
$\mathcal{U}_q^\mu (\hat{\Glie})$ contains a subalgebra isomorphic to $\mathcal{U}_q^0(\hat{\Glie})$ and so a subalgebra
isomorphic to $\mathcal{U}_q(\hat{\bo})$. It is not equal to the subalgebra constructed in the proof of Proposition \ref{suba} if $\mu\neq 0$ (this follows from Remark \ref{finc}).
\end{rem}

Let $\mu\in\Lambda$ and $\mu'\in -\Lambda^+$. Then the shift homomorphism $\iota_{\mu,\mu',a}$ defines a functor
$$\mathcal{R}_{\mu,\mu',a} : \mathcal{O}_{\mu+\mu'}\rightarrow \mathcal{O}_\mu.$$
Conversely, for a representation $V$ of $\mathcal{U}_q(\hat{\Glie})$, let us consider the $ \mathcal{U}_q^{\mu + \mu'}(\hat{\Glie})$-module : 
$$\mathcal{I}_{\mu,\mu',a}(V) = \mathcal{U}_q^{\mu + \mu'}(\hat{\Glie})\otimes_{\mathcal{U}_q^\mu(\hat{\Glie})} V.$$
This gives a functor 
$$\mathcal{I}_{\mu,\mu',a} : \text{Mod}_\mu \rightarrow \text{Mod}_{\mu + \mu'},$$
from the category $\text{Mod}_\mu$ of representations of $\mathcal{U}_q^\mu(\hat{\Glie})$ to the category $\text{Mod}_{\mu + \mu'}$ of representations of $\mathcal{U}_q^{\mu + \mu'}(\hat{\Glie})$.

From the defining formulas of $\iota_{\mu,\mu',a}$, we get the following.

\begin{prop}\label{repsp} Let $L(\Psib)$ be a simple module in $\mathcal{O}_{\mu + \mu'}$. Then $\mathcal{R}_{\mu,\mu',a}(L(\Psib))$ is a highest $\ell$-weight module of highest $\ell$-weight 
$$\Psib' = \Psib\prod_{i\in I}\Psib_{i,a}^{-\mu'(\alpha_i^\vee)}$$ 
and so admits $L(\Psib')$ as a subquotient. Conversely, $\mathcal{I}_{\mu,\mu',a}(L(\Psib'))$ is a representation of $\mathcal{U}_q^{\mu + \mu'}(\hat{\Glie})$
of highest $\ell$-weight $\Psib$ which admits $L(\Psib)$ as a simple quotient.
\end{prop}

\begin{example}
For $i\in I$ and $a\in\mathbb{C}^*$ we have the functor
$$\mathcal{R}_{\omega_i^\vee, -2\omega_i^\vee,a}:\mathcal{O}_{-\omega_i^\vee}\rightarrow \mathcal{O}_{\omega_i^\vee}.$$
Then, $L_{i,a}^+$ is a $1$-dimensional subquotient of $\mathcal{R}_{\omega_i^\vee,-2\omega_i^\vee,a}(L_{i,a}^-)$  which is not simple. 
Note also that the functors factors through $\mathcal{O}_0$ where we see already a $1$-dimensional subquotient.
\end{example}

\section{Fusion product and Grothendieck ring}\label{fusionr}

We construct the fusion product of representations of shifted quantum affine algebras in the category $\mathcal{O}^{sh}$ by using the deformed Drinfeld coproduct and 
the renomalization procedure in \cite{H2} (Theorem \ref{fp}). This leads to the definition of a ring structure on the Grothendieck group $K_0(\mathcal{O}^{sh})$. We establish a simple module in $\mathcal{O}^{sh}$ is a quotient of a fusion product of various prefundamental and constant representations (Corollary \ref{role}). Note that the Drinfeld coproduct is not an analog of 
the shifted Yangian coproduct of \cite{fkprw}  (see Remark \ref{opro}). Hence the methods and results in this section are not 
trigonometric versions of known results for Yangians.

Along the way we consider analogs of Frenkel-Reshetikhin $q$-characters of representations of shifted quantum affine algebras. We establish $q$-characters of simple representations
satisfy a triangularity property with respect to Nakajima partial ordering (Theorem \ref{partialo}).

\subsection{Characters}

Following \cite{Fre, HJ}, there is a linear $q$-character morphism
$$\chi_q : K_0(\mathcal{O}_\mu) \rightarrow \mathcal{E}_{\ell,\mu}$$
where $K_0(\mathcal{O}_\mu)$ is the Grothendieck group of the abelian category $\mathcal{O}_\mu$ and 
$\mathcal{E}_{\ell,\mu}\subset \mathbb{Z}^{\mathfrak{r}_\mu}$ is a group of formal series with coefficients in $\mathfrak{r}_\mu$ as in \cite{HJ}. 
It is defined by 
$$\chi_q(V) = \sum_{\Psib\in \mathfrak{r}_\mu}\text{dim}(V_{\Psib})[\Psib]$$
where $V_{\Psib}$ is the $\ell$-weight space of $\ell$-weight $\Psib$ as above and $[\Psib]$ is the map $\delta_{\Psib,.}$ 
(which assigns $1$ to $\Psib$ and $0$ to all other $\Psib'$).
We recover the standard character
$$\chi(V) = \varpi(\chi_q(V)) = \sum_{\omega\in \tb^*} \text{dim}(V_\omega) [\omega]$$
where we have set $\varpi([\Psib]) = \Psib(0)\in \tb^*$.

Due to Theorem \ref{param}, the $q$-characters morphism separates isomorphism classes of simple modules and 
the $q$-characters of simple modules are linearly independent for weight reasons. Hence, by standard arguments one
obtains the following.

\begin{cor}\label{injq} The $q$-character morphism $\chi_q$ is injective.\end{cor}

\begin{example}\label{exqchar} (i) For $i\in I$ and $a\in\mathbb{C}^*$, the prefundamental representation $L_{i,a}^+$ in $\mathcal{O}_{\omega_i^\vee}$ satisfies
$$\chi_q(L_{i,a}^+) = [\Psib_{i,a}].$$
It is different from the $q$-character of the corresponding simple $\mathcal{U}_q(\hat{\bo})$-module in $\mathcal{O}$ which is infinite-dimensional.

(ii) When restricted to the category $\hat{\mathcal{O}}\subset \mathcal{O}_0$, we recover the $q$-character of $\mathcal{U}_q(\hat{\Glie})$-modules.

(iii) For $\Glie = sl_2$, we have
$$\chi_q(L_{1,a}^-) =  \sum_{m\geq 0} [q^{-2m}\Psib_{1,aq^{-2m}}^{-1}\Psib_{1,aq^2}\Psib_{1,aq^{2(1 - m)}}^{-1}].$$

(iv) More generally, for $i\in I$, $a\in\mathbb{C}^*$, we get an analog of the $q$-character formula established in \cite{FH2} :
$$\tilde{\Psib}_{i,a} = \Psib_{i,a}^{-1}\left(\prod_{j,C_{i,j} = -1}\Psib_{j,aq_i} \right)\left(\prod_{j,C_{i,j} = -2}\Psib_{j,a}\Psib_{j,aq^2} \right)\left(\prod_{j,C_{i,j} = -3}\Psib_{j,aq^{-1}}\Psib_{j,aq}\Psib_{j,aq^3} \right),$$
$$\chi_q(L(\tilde{\Psib}_{i,a})) = \sum_{m\geq 0} [\overline{-m\alpha_i}][\tilde{\Psib}_{i,aq_i^{-2m}}\Psib_{i,aq_i^2}\Psib_{i,aq_i^{2(1-m)}}^{-1}].$$ 
Indeed, this representation can be realized with a basis $(v_m)_{m\geq 0}$ of $\ell$-weight vectors corresponding to the terms in this sum. For $r\in\mathbb{Z}$, the $x_{j,r}^\pm$ have a $0$ action if $j\neq i$,
$$x_{i,r}^+.v_m =   a^r q^{2r (1 - m)}\delta_{m>0}v_{m- 1}\text{ , }x_{i,r}^-.v_m = a^r q^{- (2r +1)m}\frac{[m+1]_q}{q - q^{-1}} v_{m+1}.$$
\end{example}

Let $\mathcal{E}_\ell = \bigoplus_{\mu \in \Lambda}\mathcal{E}_{\ell,\mu}$. We get an injective linear morphism.
$$\chi_q : K_0(\mathcal{O}^{sh}) = \bigoplus_{\mu\in \Lambda}K_0(\mathcal{O}_\mu) \rightarrow \mathcal{E}_\ell.$$
We have a natural bilinear product 
$$\mathcal{E}_{\ell,\mu_1}\otimes \mathcal{E}_{\ell,\mu_2}\rightarrow \mathcal{E}_{\ell,\mu_1 + \mu_2},$$ 
which induces a ring structure on $\mathcal{E}_\ell$. Hence, we can multiply $q$-characters. Let us explain the categorical meaning of this product.

\subsection{Deformed Drinfeld coproduct}

The Drinfeld coproduct, and its deformed version \cite[Section 3.1]{H2}, can be defined for 
shifted quantum affine algebras by using the same formula as for quantum affine algebras (see \cite[Section 10.1]{FT}). 

For $u$ a formal parameter and $\mu_1,\mu_2\in\Lambda$, 
$$(\mathcal{U}_q^{\mu_1}(\hat{\Glie})
\otimes\mathcal{U}_q^{\mu_2}(\hat{\Glie}))((u))$$ 
is the algebra of formal Laurent series
with coefficients in $\mathcal{U}_q^{\mu_1}(\hat{\Glie})
\otimes\mathcal{U}_q^{\mu_2}(\hat{\Glie})$. The deformed Drinfeld coproduct is the algebra morphism : 
$$\Delta_u : \mathcal{U}_q^{\mu_1 + \mu_2}(\hat{\Glie})\rightarrow (\mathcal{U}_q^{\mu_1}(\hat{\Glie})
\otimes \mathcal{U}_q^{\mu_2}(\hat{\Glie}))((u))$$
define by the formulas
$$\Delta_u (x_i^+(z)) = x_i^+(z)\otimes 1 + \phi_i^-(z)\otimes x_i^+(zu)\text{ , }\Delta_u(x_i^-(z)) = 1\otimes x_i^-(zu) + x_i^-(z)\otimes \phi_i^+(zu),$$
$$\Delta_u(\phi_i^\pm(z)) = \phi_i^\pm(z)\otimes \phi_i^\pm(zu).$$

\begin{rem}\label{opro} (i) The specialization at $u = 1$ of $\Delta_u$ is well-defined only in a completion of the tensor product $\mathcal{U}_q^{\mu_1}(\hat{\Glie})
\otimes \mathcal{U}_q^{\mu_2}(\hat{\Glie})$.

(ii) Another coproduct, analog to the Drinfeld-Jimbo coproduct of quantum affine algebras and to the coproduct for 
shifted Yangians in \cite{fkprw}, is defined in \cite{FT} for type $A$ shifted quantum affine algebras (it is conjectured
to exist for any types). 
\end{rem}

\subsection{Fusion product}

Consider $V_1$, $V_2$ respectively in $\mathcal{O}_{\mu_1}$, $\mathcal{O}_{\mu_2}$. We get a structure of
$\mathcal{U}_q^{\mu_1 + \mu_2}(\hat{\Glie})$-module on the space of Laurent formal power series with coefficients in $V_1\otimes V_2$ : 
$$(V_1\otimes V_2)((u)).$$
This representation is the sum of its weight spaces which are infinite-dimensional. But let us study how to get a representation in 
the category $\mathcal{O}_{\mu_1 + \mu_2}$ from this representation.

We use the fusion procedure introduced in \cite{H, H2} for quantum affine algebras (and quantum affinizations). 
In general, the formal parameter $u$ can not be specialized directly to a specific complex number. 
However, one can prove as in \cite[Lemma 3.10]{H} that for $V = V_1\otimes V_2$, the subspace of rational 
Laurent formal power series
$$V(u)\subset V((u))$$ 
is a stable submodule. 

Let $\mathcal{A}\subset \CC(u)$ be the subring of rational fractions regular at $1$. An $\mathcal{A}$-form 
$\tilde{V}\subset V(u)$ is a $\mathcal{A}\otimes \mathcal{U}_q^{\mu_1 + \mu_2}(\hat{\Glie})$-submodule generating $V(u)$ as a $\CC(u)$-vector space 
and so that its intersection with any weight space of $V(u)$ is a finitely generated $\mathcal{A}$-module.

Suppose that $V_1$, $V_2$ are of highest $\ell$-weights. Then it is proved as in \cite[Theorem 6.2]{H2} that $V(u)$ 
is cyclic for the action of $\mathcal{U}_q^{\mu_1 + \mu_2}(\hat{\Glie})\otimes \mathbb{C}(u)$ generated 
by a tensor product $v_1\otimes v_2$ of highest weight vectors $v_1$, $v_2$. Then we obtain as in \cite[Lemma 4.8]{H2} 
that the $\mathcal{A}\otimes \mathcal{U}_q^{\mu_1 + \mu_2}(\hat{\Glie})$-submodule generated by $v_1\otimes v_2$ is an 
$\mathcal{A}$-form that we denote $\tilde{V}$. Then 
$$V_1 * V_2 = \tilde{V}/(u - 1) \tilde{V}$$ 
is a $\mathcal{U}_q^{\mu_1 + \mu_2}(\hat{\Glie})$-module in the category $\mathcal{O}_{\mu_1 + \mu_2}$ called the
fusion product of $V_1$ and $V_2$.

\begin{thm}\label{fp} The fusion product $V_1 * V_2$ is a well-defined highest $\ell$-weight module in $\mathcal{O}_{\mu_1 + \mu_2}$ satisfying 
$$\chi_q(V_1 * V_2) = \chi_q(V_1)\chi_q(V_2).$$
\end{thm}

For $V_1,\cdots, V_r$ a family of highest $\ell$-weight representations $V_i$ in $\mathcal{O}_{\mu_i}$, the same procedure gives a fusion module
$$V_1 * V_2 * \cdots  * V_r$$
in $\mathcal{O}_{\mu_1 + \cdots + \mu_r}$ with
$$\chi_q(V_1 * V_2 * \cdots * V_r) = \chi_q(V_1)\cdots \chi_q(V_r).$$
The first example is the following fusion of positive (resp. negative) prefundamental representations.

\begin{thm}\label{fpm} A fusion product of positive (resp. negative) prefundamental representations is simple :
$$L_{i_1,a_1}^{\pm}*L_{i_2,a_2}^{\pm} *\cdots * L_{i_N,a_N}^{\pm }\simeq L((\Psib_{i_1,a_1}\Psib_{i_2,a_2} \Psib_{i_N,a_N})^{\pm 1})$$
for any $i_1,\cdots, i_N\in I$, $a_1,\cdots, a_N\in\mathbb{C}^*$.
\end{thm}

\begin{proof} For positive prefundamental representations, it is clear as these representations are one-dimensional. 
It is proved in \cite{FH} that a tensor product of negative prefundamental representations of $\mathcal{U}_q(\hat{\bo})$ 
is simple as a $\mathcal{U}_q(\hat{\bo})$-module. Hence, by Corollary \ref{still}, this tensor product has the same $q$-character 
as a simple module of the corresponding shifted quantum affine algebra. But, due to Theorem \ref{fp}, 
this is also the $q$-character of the fusion product of 
these negative prefundamental representations, hence this fusion product is simple.\end{proof}

The following confirms prefundamental representations play the role of fundamental representations in the category $\mathcal{O}^{sh}$.

\begin{cor}\label{role} A simple module in $\mathcal{O}^{sh}$ is a quotient of a fusion product of various prefundamental representations and a simple constant representation.
\end{cor}

\begin{proof} For $L(\Psib)$ such a simple representation, it suffices to write $\Psib = \Psib(0) \Psib^+\Psib^-$ where 
$\Psib^\pm$ is a product of various $\Psib_{i,a}^{\pm 1}$. Then $L(\Psib)$ is a subquotient of $L(\Psib(0))*L(\Psib^+)*L(\Psib^-)$.
\end{proof}

\begin{cor}\label{preff} A simple module in $\mathcal{O}^{sh}$ is a subquotient of a fusion product of $1$-dimensional module by a simple module of $\mathcal{U}_q(\hat{\bo})$.
\end{cor}

\begin{proof} Let $L(\Psib)$ be a simple representation in $\mathcal{O}_\mu$. Then we can factorize $\Psib = \lambda \Psib^+ \Psib^-$ where $\lambda$ is constant, $\Psib^\pm$ is a product of various $\Psib_{i,a}^{\pm 1}$. 
Then $L(\lambda \Psib^+)$ is $1$-dimensional and by Corollary \ref{still}, $L(\Psib^-)$ is simple when restricted to $\mathcal{U}_q(\hat{\bo})$. Then $L(\Psib)$ is a quotient of 
$$L(\lambda \Psib^+) * L(\Psib^-).$$
\end{proof}

\begin{rem}\label{caspe} (i) For $V_2$ a fusion product of positive prefundamental representations  and $V_1$ a representation in $\mathcal{O}_{\mu_1}$, 
$(V_1\otimes V_2)\otimes \mathcal{A}$ is an $\mathcal{A}$-lattice. Indeed $V_2$ is $1$-dimensional and 
$x_{i,r}^\pm$ act by $0$ and $\phi_{i,r}^+$ act by $0$ for $r$ large enough. Then the image of $\mathcal{U}_q^{\mu_1 + \mu_2}(\hat{\Glie})$ by $\Delta_u$, after composing by 
the representation morphisms, gives a Laurent polynomial in $\text{End}(V_1\otimes V_2)[u^{\pm 1}]$.

(ii) For $V_2 = L_{i,a}^+$, one gets a functor
$$*_{i,a} : \mathcal{O}_{\mu}\rightarrow \mathcal{O}_{\mu + \omega_i^\vee}$$
which preserves the dimension and the character so that $\chi_q(*_{i,a}(V)) = [\Psi_{i,a}]\chi_q(V)$. It coincides with 
the functor $\mathcal{R}_{\mu + \omega_i^\vee,-\omega_i^\vee,a}$ from Section \ref{reps}.
\end{rem}

\subsection{The Grothendieck ring $K_0(\mathcal{O}^{sh})$}

As the $q$-character morphism is injective by Corollary \ref{injq}, it follows from Theorem \ref{fp} that the 
image 
$$\chi_q(K_0(\mathcal{O}^{sh}))\subset \mathcal{E}_\ell $$ 
is a subring of $\mathcal{E}_\ell$. This induces a ring structure on $K_0(\mathcal{O}^{sh})$ 
with positive constant structures on the basis of simple modules. By construction
$$\chi_q : K_0(\mathcal{O}^{sh}) \rightarrow \mathcal{E}_{\ell}$$
is an injective ring morphism. Clearly, $K_0(\mathcal{O}_0)$ is a subring of $K_0(\mathcal{O}^{sh})$.

\begin{example} (i) For $\Glie = sl_2$, we have
$$[L_{1,a}^{-}][L_{1,a}^+] = 1 +  [\overline{-\alpha_1}][L_{1,aq^{-2}}^{-}][L_{1,aq^2}^+]).$$

(ii) More generally, recall the representations $L(\tilde{\Psib}_{i,a})$ from Example \ref{exqchar}. We have an analog of the $Q\tilde{Q}$-system 
established in \cite{FH2} in $K_0(\mathcal{O})$ :
$$[L(\tilde{\Psib}_{i,a})][L_{i,a}^+] = [\overline{-\alpha_i}][L(\tilde{\Psib}_{i,aq_i^{-2}})][L_{i,aq_i^2}^+]$$
$$+\left(\prod_{j,C_{i,j} = -1}[L_{j,aq_i}^+] \right)\left(\prod_{j,C_{i,j} = -2}[L_{j,a}^+][L_{j,aq^2}^+] \right)\left(\prod_{j,C_{i,j} = -3}[L_{j,aq^{-1}}^+][L_{j,aq}^+][L_{j,aq^3}^+] \right).$$
\end{example}

\subsection{Root monomials and Nakajima partial ordering}\label{npo}
Following \cite{Fre}, we introduce for $i\in I$, $a\in\mathbb{C}^*$ the following $\ell$-weight which is a monomial analog of a simple root : 
\begin{equation}\label{ayf}A_{i,a} = Y_{i,aq_i^{-1}}Y_{i,aq_i} \left(  \prod_{j\in I, C_{j,i} = -1}Y_{j,a}^{-1}\prod_{j\in I, C_{j,i} = -2}Y_{j,aq^{-1}}^{-1}Y_{j,aq}^{-1}\prod_{j\in I, C_{j,i} = -3}Y_{j,aq^{-2}}^{-1}Y_{j,a}^{-1}Y_{j,aq^2}^{-1}  \right),
\end{equation}
where 
$$Y_{i,a} =   \overline{\omega_i} \Psib_{i,aq_i}^{-1}\Psib_{i,aq_i^{-1}}.$$
Note this $\ell$-weight can also be written simply as
$$A_{i,a} = \overline{\alpha_i}  \prod_{j\in I}\Psib_{j,aq^{B_{i,j}}}^{-1}\Psib_{j,aq^{-B_{i,j}}}.$$

\begin{rem} Note that the Langlands dual Cartan matrix $(C_{j,i})_{i,j}$ occurs in the definition of $A_{i,a}$ in contrast to the definition of the 
$\ell$-weights $\tilde{\Psib}_{i,a}$ in Example \ref{exqchar}. However, we can rewrite the formula therein 
$$\chi_q(L(\tilde{\Psib}_{i,a})) = [\tilde{\Psib}_{i,a}]\sum_{k\geq 0} A_{i,a}^{-1}A_{i,aq_i^{-2}}^{-1} \cdots A_{i,aq_i^{-2(k - 1)}}^{-1}.$$
\end{rem}

We extend Nakajima partial ordering \cite{N1} to $\ell$-weights : we set $\Psib \preceq \Psib '$ if and only if $\Psib' \Psib^{-1}$ is a monomial in the $A_{i,a}$.

\begin{thm}\label{partialo} For $\Psib$ an $\ell$-weight, we have
$$\chi_q(L(\Psib))\in [\Psib] + \sum_{\Psib' \prec \Psib} \mathbb{N} [\Psib'].$$
\end{thm}

\begin{proof} By Corollary \ref{preff}, it suffices to prove the statement for prefundamental representations.
This is clear for positive prefundamental representations as they are $1$-dimensional. For negative
prefundamental representations $L_{i,a}^-$, it follows from Corollary \ref{still} that the $q$-character coincides with the $q$-character of the
negative prefundamental representation $L_{i,a}^{\bo,-}$ of $\mathcal{U}_q(\hat{\bo})$. In this 
case the result is proved in \cite{HJ}.
\end{proof}

\begin{rem} (i) The statement was proved for finite-dimensional representations of quantum affine algebras in \cite{Fre2}.

(ii) The analogous statement is not satisfied in general for the representations of the quantum affine Borel algebra (for example for
positive prefundamental representations $L_{i,a}^{\bo,+}$).
\end{rem}

\section{Finite-dimensional representations}\label{fdr}

In this section we classify the simple finite-dimensional representations of shifted quantum affine algebras (Theorem \ref{fdclas}). 

For $\Glie$ simply-laced, a classification of simple finite-dimensional representations of simply-laced shifted Yangians 
is given  in \cite[Theorem 1.4]{ktwwy} (see (iii) in Remark \ref{refy}) by a different method.

Our results and methods are uniform for any type of $\Glie$, simply-laced or not simply-laced.

\medskip

The standard Theorem of Chari-Pressley \cite{CP} classifying finite-dimensional representations of quantum
affine algebras in terms of Drinfeld polynomials can be formulated in the following form (see also \cite[Examples in Section 3.2]{HJ}).

\begin{thm}\label{case0} The simple finite dimensional representations of $\mathcal{U}^{\mu = 0}_q(\hat{\Glie})$ are the 
$L(\Psib)$ where $\Psib(z)(\Psib(0))^{-1}$ is a monomial in the 
$$\tilde{Y}_{i,a} = \overline{\omega_i}^{-1} Y_{i,a} =  \Psib_{i,aq_i^{-1}}(\Psib_{i,aq_i})^{-1} \text{ for $i\in I$, $a\in\mathbb{C}^*$.}$$
\end{thm}

\begin{proof} By \cite{CP}, the simple finite dimensional representations of the quantum affine algebra $\mathcal{U}_q(\hat{\Glie})$ are 
parametrized by the $\Psib$ which are monomial in the $Y_{i,a}$.

Recall that by (ii) in Remark \ref{com}, $\mathcal{U}_q^0(\hat{\Glie})$ is a central extension of $\mathcal{U}_q(\hat{\Glie})$.
Let $L(\Psib)$ be a simple finite-dimensional representation of $\mathcal{U}_q^0(\hat{\Glie})$. Then for $\lambda \in (\mathbb{C}^*)^n$ 
a square root of $(\Psib(0)\Psib(\infty))^{-1}$, $L(\lambda\Psib)$ is a simple finite-dimensional representation of $\mathcal{U}_q(\hat{\Glie})$
and so $\Psib$ has the correct form. Conversely, if $\Psib = \lambda^{-1}\Psib'$ where $L(\Psib')$ is a simple finite-dimensional representation of $\mathcal{U}_q(\hat{\Glie})$ and $L(\lambda^{-1})$ is a one-dimensional representation of $\mathcal{U}_q^0(\hat{\Glie})$, then 
$$L(\Psib) \simeq L(\lambda^{-1}) * L(\Psib')$$ 
is a finite-dimensional of $\mathcal{U}_q^0(\hat{\Glie})$ with the same dimension as $L(\Psib')$.
\end{proof}

To generalize this for all shifted quantum affine algebras, first the following follows easily from the previous results.

\begin{prop}\label{fsense} For $\mu\in\Lambda$, the algebra $\mathcal{U}_\mu(\hat{\Glie})$ admits 
non-zero finite-dimensional representations if and only if $\mu$ is dominant.

Let $\mu\in \Lambda^+$ be dominant and $\Psib\in \mathfrak{r}_\mu$ be such that $\Psib (\Psib(0))^{-1}$ is a product of various 
$$\tilde{Y}_{i,a}\text{ and }\Psib_{i,a}\text{ for $i\in I$, $a\in\mathbb{C}^*$}.$$
Then $L(\Psib)$ is a simple finite-dimensional representation of $\mathcal{U}_q^\mu(\hat{\Glie})$.
\end{prop}

\begin{proof} If there is $i\in I$ so that $\alpha_i(\mu) < 0$, then $\mathcal{U}_q^\mu(\hat{\Glie})$ contains a subalgebra
isomorphic to $\mathcal{U}_{q_i}^+(sl_2)$. So it follows from Proposition \ref{nonze} that zero is the only
finite-dimensional representation of  $\mathcal{U}_q^\mu(\hat{\Glie})$. This implies the "only if" part of the first point. 
For the "if" part, it suffices to establish the second point. So consider $\Psib(z)$ as in the statement. We can write $\Psib(z) = \Psib(0) M_1 M_2$
where $M_1$ is a product of various $\tilde{Y}_{i,a}$ and 
$$M_2 = \Psib_{i_1,a_1}\Psib_{i_2,a_2}\cdots \Psib_{i_N,a_N}$$
for various $i_1,\cdots, i_N\in I$, $a_1,\cdots, a_N\in\mathbb{C}^*$. Then $L(\Psib(0)M_1)$ is a simple finite-dimensional 
representation of $\mathcal{U}_q^0(\hat{\Glie})$ by Proposition \ref{case0}. Using inductively the functors 
$$\mathcal{R}_{\omega_{i_1}^\vee + \cdots + \omega_{i_{j-1}}^\vee,-\omega_{i_j}^\vee,a_j} : \mathcal{O}_{\omega_{i_1}^\vee + \cdots + \omega_{i_{j-1}}^\vee} \rightarrow \mathcal{O}_{\omega_{i_1}^\vee + \cdots + \omega_{i_{j}}^\vee} $$
from Section \ref{reps}, we get from Proposition \ref{repsp} that $L(\Psib)$ is finite-dimensional (with dimension no larger than the dimension of $L(M_1)$).
\end{proof}

\begin{rem}\label{refy} (i) The condition in Proposition \ref{fsense} appeared in \cite{HL} and in \cite{FJMM} for the $\ell$-weights of the simple modules of a category $\mathcal{O}^+$ of representations of the quantum affine Borel algebra (see Remark \ref{exom}). 

(ii) If $\Psib$ satisfies the condition of the statement and in addition is a Laurent monomial in the $\tilde{Y}_{i,a}$, then the powers of the $\tilde{Y}_{i,a}$ are 
positive. So, following \cite{Fre, FJMM}, let us call a general $\ell$-weight satisfying this 
condition a dominant $\ell$-weight.

(iii) Let $\Psib(z)$ be an $\ell$-weight whose poles and zeros are in $q^{\mathbb{Z}}$. It is a Laurent monomial in the $\Psib_{i,q^r}^{\pm 1}$.
There is a structure of crystal on the set of such Laurent monomials \cite{nacr, K2} (the variables $Y_{i,q^r}$ in \cite{nacr, K2} are the $\Psi_{i,q^r}$ here; 
this should not be confused with the $Y_{i,q^r}$ above). Such an $\ell$-weight 
$\Psib(z)$ is dominant if and only if it is highest weight for this crystal structure. For $\Glie$ simply-laced, it is the condition found in 
\cite[Theorem 1.4]{ktwwy} where a classification of simple finite-dimensional representations of simply-laced shifted Yangians was given 
(the proof therein is based on type $A$ results in \cite[Section 7.2]{brkl}).
\end{rem}

We will prove the converse statement which gives a complete classification of finite-dimensional representations of shifted quantum affine algebras.

\begin{thm}\label{fdclas} The simple finite-dimensional representations of shifted quantum affine algebras are the $L(\Psib)$ 
where $\Psib$ is dominant.
\end{thm}

\begin{proof}
From Proposition \ref{fsense}, it suffices to prove for $\Glie = sl_2$ that $L(\Psib)$ finite-dimensional implies $\Psib$ dominant.

Let $\mu \in \Lambda^+$ and suppose that $L(\Psib)$ is a simple finite-dimensional representation of 
$\mathcal{U}_q^\mu(\hat{sl}_2)$. As discussed above, $\Psib(z)$ is a rational fraction of non-negative degree. 
Without loss of generality, we may assume $\Psib(0) = 1$. There is a (non-unique) factorization $\Psib = \Psib^+ \Psib^0 $ where 
$$\Psib^0 = (\Psib_{a_1}\Psib_{b_1}^{-1})\cdots (\Psib_{a_N}\Psib_{b_N}^{-1})\text{ for certain $N\geq 0$, $a_1,\cdots a_N, b_1,\cdots, b_N\in\mathbb{C}^*$,}$$
$$\Psib^+ = \Psib_{c_1}\cdots \Psib_{c_M}\text{ where $M = \text{deg}(\Psib) \geq 0$ and $c_1,\cdots, c_M\in\mathbb{C}^*$.}$$ 
We will denote
$$\mathcal{F} = \{1\leq j\leq N|a_{j'}\in  b_{j'} q^{-2\mathbb{Z}}\}\text{ and }\mathcal{I} = \{1\leq j\leq N|a_{j'}\notin  b_{j'} q^{-2\mathbb{Z}}\}$$
so that for $j\in \mathcal{F}$ (resp. $j\in \mathcal{I}$) $L(\Psi_{a_j}\Psib_{b_j}^{-1})$ is finite-dimensional (resp. infinite-dimensional).
Moreover, we may assume that for any $1\leq j\leq M$, $1\leq j'\leq N$ : 
\begin{equation}\label{condcab}(c_j\notin b_{j'}q^{-2\mathbb{Z}}\text{ if }j'\in\mathcal{I})\text{ and }(c_j\notin \{b_{j'},b_{j'}q^{-2},\cdots, a_{j'}q^2\}\text{ if }j'\in\mathcal{F}).\end{equation}
Then we prove that
$$L(\Psib)\simeq L(\Psib^0) * L(\Psib^+).$$
As $\chi_q([L(\Psib^+)] = [\Psib^+]$, we want to prove that the multiplicities of $\ell$-weights are the same in 
$\chi_q(L(\Psib))$ and in $[\Psib^+]\chi_q(L(\Psib^0))$. First $L(\Psib)$ is a quotient of $L(\Psib^0) * L(\Psib^+)$, 
so the multiplicities in $\chi_q(L(\Psib))$ are lower or equal to the multiplicities in $[\Psib^+]\chi_q(L(\Psib^0))$. 
But $L(\Psib^0)$ is
a quotient of $L(\Psib) * L((\Psib^+)^{-1})$. We have precise informations on the $q$-character of $L(\Psib^0)$ and $L((\Psib^+)^{-1})$ :
$$\chi_q(L((\Psib^+)^{-1})\in [(\Psib^+)^{-1}] (1 + \sum_j A_{c_j}^{-1}\ZZ[A_{a}^{-1}]_{a\in\mathbb{C}^*}),$$
$$\chi_q(L(\Psib^0))\in [\Psib^0] \ZZ[A_{b_jq^{-2m}}^{-1}]_{j\in\mathcal{F},m\geq 0}\ZZ[A_{b_j}^{-1},A_{b_jq^{-2}}^{-1},\cdots, A_{a_jq^2}^{-1}]_{j\in\mathcal{I}}.$$
(see (iii) in Example \ref{exqchar} for the first one; as $L(\Psib_a\Psib_b^{-1})$ is a quotient of the fusion $L(\Psib_a)*L(\Psib_b^{-1})$, 
the second also follows from this Example). Now, from (\ref{condcab}), only the $\ell$-weight $[(\Psib^+)^{-1}]$ of $L((\Psib^+)^{-1})$ can contribute to 
an $\ell$-weight of $L(\Psib_0)$. So we get that the multiplicities of $\ell$-weights in $\chi_q(L(\Psib^0))$ are lower 
or equal than in $[(\Psib^+)^{-1}]\chi_q(L(\Psib))$. We have proved the isomorphism $L(\Psib)\simeq L(\Psib^0) * L(\Psib^+)$.
This implies that $L(\Psib^0)$ is finite-dimensional and we obtain the desired condition on $\Psib_0$ from Theorem \ref{case0}. Hence the result.
\end{proof}

\begin{rem} 

(i) The result implies that $L(\Psib)$ is finite-dimensional if and only if the simple module $L^\bo(\Psib)$ of $\mathcal{U}_q(\hat{\bo})$ 
is in the category $\mathcal{O}^+$ (see Remarks \ref{exom}, \ref{refy}).

(ii) The factorization of $\Psib = \Psib^+ \Psib^0$ appeared in \cite{FJMM} in the classification of "finite-type" simple representations of $\mathcal{U}_q(\hat{\bo})$. 
The proof that $L^\bo(\Psib)\simeq L^\bo(\Psib^+)\otimes L^\bo(\Psib^0)$ in this
context is given in \cite[Lemma 5.9]{FJMM} and is more complicated.

(iii) In type $A$, a classification of simple finite-dimensional representations of shifted quantum current algebras 
is obtained in \cite{kw} in terms of Drinfeld polynomials. These are subalgebras of a shifted quantum affine algebra of type $A$ generated by positive mode Drinfeld generators. Their motivation comes from representations of cyclotomic $q$-Schur algebras.
\end{rem}

\begin{example}\label{exhl1} In addition to finite-dimensional representations of quantum affine algebras, there are many new examples
of finite-dimensional representations. For example the positive prefundamental representations $L(\Psib_{i,a})$ have dimension $1$ and
for 
$$\Psib_{i,a}^* = \Psib_{i,a}^{-1}\prod_{j,C_{i,j} \neq 0}\Psib_{j,aq_j^{-C_{j,i}}}$$ 
the $\ell$-weight in \cite[Section 6.1.3]{HL}, $L(\Psib_{i,a}^*)$ has dimension $2$ with 
$$\chi_q(L(\Psib_{i,a}^*)) = [\Psib_{i,a}^*] + [\overline{-\alpha_i}] [\Psib_{i,a}^{-1}\prod_{j,C_{i,j} \neq 0}\Psib_{j,aq_j^{C_{j,i}}}] = [\Psib_{i,a}^*](1 + A_{i,a}^{-1}).$$
This representation is in $\mathcal{O}_\mu$ with $\mu = \sum_{j,C_{j,i} < 0}\omega_j^\vee$ and can be realized with a basis $v_0$, $v_1$  of $\ell$-weight vectors corresponding to the terms in this sum. 
For $r\in\mathbb{Z}$, the $x_{j,r}^\pm$ have a zero action if $j\neq i$ and
$$x_{i,r}^+.v_0 = x_{i,r}^-.v_1 = 0\text{ , }x_{i,r}^-.v_0 =   a^r v_{1}\text{ , }x_{i,r}^+.v_1 = a^r q^{-1} v_0.$$
We get an analog in $K_0(\mathcal{O}^{sh})$ of the $QQ^*$-system established in \cite{HL} in $K_0(\mathcal{O})$ :
$$[L(\Psib_{i,a}^*)][L_{i,a}^+] = [\overline{-\alpha_i}] \prod_{j,C_{i,j} \neq 0}[L_{j,aq_j^{C_{j,i}}}^+] 
+ \prod_{j,C_{i,j} \neq 0}[L_{j,aq_j^{-C_{j,i}}}^+].$$
\end{example}

\section{Induction and restriction functors}\label{irf}

We define and study induction and restriction functors relating the category $\mathcal{O}$ of representations
of the quantum affine Borel algebra $\mathcal{U}_q(\hat{\bo})$ and the categories $\mathcal{O}_\mu$ of representations of shifted quantum affine algebras $\mathcal{U}_q^\mu(\hat{\Glie})$. These functors will be useful tools for our study in the following. 
No analogs of these functors are known for shifted Yangians.

\subsection{Functors for antidominant weights}

\begin{prop} We have an equivalence of categories
$$\tilde{\mathcal{O}} \xrightarrow{\sim} \bigoplus_{\mu\in - \Lambda^+} \mathcal{O}_{\mu}.$$
\end{prop}

\begin{proof} For $\mu\in - \Lambda^+$, let $\tilde{\mathcal{O}}_{\mu}$ be the subcategory of representations in 
$\tilde{\mathcal{O}}$ on which for $i\in I$
$$\kappa_i = \phi_{i,-1}^- = \cdots = \phi_{i, \alpha_i(\mu) + 1}^- = 0$$
and $\phi_{i,\alpha_i(\mu)}$ is invertible. It is the subcategory of the representations $V$ in $\tilde{\mathcal{O}}$ so that
$V = V_{(\mu)}$. Such representations $V$ are representations of $\mathcal{U}_q^{\mu}(\hat{\Glie})$. 
Indeed, recall from Section \ref{asymprel} that $\mathcal{U}_{q}^{\mu}(\hat{\Glie})$ is obtained from a quotient of 
the asymptotic algebra $\tilde{\mathcal{U}}_q(\hat{\Glie})$. Then for a representation $V$ as above, we set $k_i.v = \omega(i) v$ for $v\in V^{(\omega)}$.
We obtain an equivalence of categories 
$$\tilde{\mathcal{O}}_{\mu} \xrightarrow{\sim} \mathcal{O}_{\mu}.$$
Now let $V$ be a representation in $\tilde{\mathcal{O}}$ and $\mu\in -\Lambda^+$. As the $\phi_{i,\alpha_i(\mu)}^-$ are invertible and 
diagonalizable on $V_{(\mu)}$ (from the definition of the category $\tilde{\mathcal{O}}$), it follows from relations (\ref{phix}) that $V_{(\mu)}$ 
is a submodule of $V$ which is in $\tilde{\mathcal{O}}_\mu$. Moreover, we obtain that there are no non-trivial extension between the 
modules $V_{(\mu)}$ and $V_{(\mu')}$ for $\mu\neq \mu'$.

We obtain that   $\tilde{\mathcal{O}} = \bigoplus_{\mu\in - \Lambda^+} \tilde{\mathcal{O}}_{\mu}$.
\end{proof}

For $\mu\in -\Lambda^+$, composing the equivalence $\mathcal{O}_\mu \xrightarrow{\sim} \tilde{\mathcal{O}}_{\mu}$ with $\tilde{f}$, we get a functor
$$f_\mu : \mathcal{O}_{\mu}\rightarrow \mathcal{O}.$$
Any simple module in $\tilde{\mathcal{O}}$ is in one of the categories $\tilde{\mathcal{O}}_\mu$.

As a consequence and from the results of the previous Sections, one gets the following.

\begin{thm}\label{simpasym}
The simple modules in the category $\tilde{\mathcal{O}}$ are parametrized by rational $\ell$-weights of non-positive degree and constant term $1$.
\end{thm}

\subsection{Induction functors}

To generalize the results of the previous section to $\mu\in\Lambda$, we have to proceed differently. 

Recall the algebra $\mathcal{U}_q^\mu(\hat{\bo})$ from Proposition \ref{mubo}. It is isomorphic to a subalgebra of $\mathcal{U}_q^\mu(\hat{\Glie})$ and of $\mathcal{U}_q(\hat{\bo})$. 

For $V$ a $\mathcal{U}_q^\mu(\hat{\Glie})$-module in the category $\mathcal{O}_\mu$, one can consider its restriction to $\mathcal{U}_q^\mu(\hat{\bo})$
and then its induction to $\mathcal{U}_q(\hat{\bo})$ : 
$$\mathcal{I}_{\mu} (V) = \mathcal{U}_q(\hat{\bo})\otimes_{\mathcal{U}_q^\mu(\hat{\bo})} V.$$
As $\mathcal{U}_q^+(\hat{\bo})$ and $\mathcal{U}_q^0(\hat{\bo})$ are both contained in $\mathcal{U}_q^\mu(\hat{\bo})$, 
the weight spaces of $\mathcal{I}_{\mu} (V)$ are finite-dimensional and we get a representation in the category $\mathcal{O}$. 
So this defines a functor
$$\mathcal{I}_{\mu} : \mathcal{O}_\mu \rightarrow \mathcal{O}.$$

\begin{rem} For $V = L(\Psib)$ in $\mathcal{O}_\mu$, the $\mathcal{U}_q(\hat{\bo})$-module 
$\mathcal{I}_{\mu}(V)$ is of highest $\ell$-weight generated by a highest weight vector of $V$. It admits $L^\bo(\Psib)$ as a quotient.
\end{rem} 

\begin{example} For $i\in I$, $a\in\mathbb{C}^*$, let $\mu = \omega_i^\vee$ and $V = L_{i,a}^+$ in $\mathcal{O}_{\omega_i^\vee}$. It has dimension $1$. Then 
 $\mathcal{I}_{\omega_i^\vee}(V)$ admits the simple infinite-dimensional $\mathcal{U}_q(\hat{\bo})$-module $L^{\bo,+}_{i,a}$ as a quotient.
In the $sl_2$-case, we have $\mathcal{U}_q(\hat{\bo}) = \mathcal{U}_q^{\omega_1^\vee}(\hat{\bo})\otimes \mathbb{C}[x_{1,1}^-]$ and so 
$\mathcal{I}_{\omega_i^\vee}(V) = \sum_{m\geq 0}(x_{1,1}^-)^m.V = L^{\bo,+}_{1,a}$.
\end{example}

\subsection{Restriction functors} Let $\mu\in \Lambda$. For $i\in I$, we set $\mu_i = \text{Max}(1,\alpha_i(\mu))$.

For $V$ a representation in $\mathcal{O}$, we consider its subspace $V_\mu$ (resp. $V_{<\mu}$) of vectors $v\in V$ 
so that for any $i\in I$, $\phi_i^+(z).v\in V(z)$ has degree lower or equal to $\alpha_i(\mu)$ (respectively strictly lower than $\alpha_i(\mu)$).

\begin{rem} $V_\mu$ is not a submodule of $V$ in general. Let $\mathfrak{g} = sl_2$ and $V = L((1 - z)^3)$ with highest weight vector $v$. Let 
$$x^{+,+}(z) = \sum_{m\geq 0}x_m^+z^m.$$ 
For $w = x_{1}^-.v$ one has $x^{+,+}(z).w=  \frac{z^2 - 3 z + 3}{q- q^{-1}}v$ and
$$- q^2 (1 - z)^3x_{0}^+ (\phi_1^+(z))^{-1}.w =  -  x_{0}^+.w + (1 - q^4)  x_i^{+,+}(z q^2).w = \frac{(1 - q^4)(q^4z^2 - 3q^2 z + 3) - 3}{q- q^{-1}}.v$$
In particular,  $(\phi_1^+(z))^{-1}.w$ has degree larger or equal to $-1$. But on the weight space of $V$ of weight $-\alpha_1$, $\phi_1^+(z) = q^{-2}(1 - z)^3 \text{Id} + N(z)$
with $N(z)^3 = 0$. Its inverse is 
$$q^2(1 - z)^{-3}\text{Id} - q^4(1 - z)^{-6}N(z) + q^6 (1 - z)^{-9}N^2(z).$$ 
Hence $N(z)$ has degree larger or equal to $4$, and so is the degree of $\phi_1^+(z)$.
\end{rem}

\begin{prop} For $v\in V_\mu$ (resp. $V_{<\mu}$), there is $M\geq 0$ so that  $\phi_{i,s}^+.v$, $x_{i,m}^+.v$, $x_{i,m}^-.v$ are in $V_\mu$ (resp. $V_{<\mu})$ for any 
$i\in I$, $s\geq 0$, $m\geq M$.
\end{prop}

\begin{proof} It suffices to show the statement for $V_{\mu}$ as 
$$V_{<\mu} = \sum_{i\in I} V_{\mu - \omega_i^\vee}.$$ 
Let $v\in V_\mu$  and $i,j\in I$. 
Let $m > 0$ so that 
$$\text{deg}(\phi_i^+(z)x_{j,m}^+.v) \geq \text{deg}(\phi_i^+(z)x_{j,m-1}^+.v).$$ 
Then the relation
$$\phi_i^+(z) (x_{j,m-1}^+ - q^{B_{i,j}}z x_{j,m}^+).v =   (q^{B_{i,j}}x_{j,m-1}^+ - z x_{j,m}^+)\phi_i^+(z).v$$
implies that $\text{deg}(\phi_i^+(z)x_{j,m}^+.v)\leq \alpha_i(\mu)$. 
So we are reduced to the case when for any $m\geq 0$, the maximum of the $\text{deg}(\phi_i^+(z)x_{j,m'}^+.v)$, $m' \geq m$
is realized at $m' = m$ only. This means that $\text{deg}(\phi_i^+(z)x_{j,m}^+.v)$ is strictly decreasing. Contradiction as 
weight spaces are finite-dimensional. 

This is analogous for the $x_{j,m}^-v$. And it is clear for the $\phi_{j,s}^+$ which commute with $\phi_i^+(z)$.
\end{proof}

We identify the elements of $\mathcal{U}_q(\hat{\bo})$ with the corresponding operators on $V$. For $v\in V_\mu$, we consider $M$ as in the previous Proposition.
Let $i\in I$. The rational function $\sum_{m\geq M}x_{i,m}^+.v z^m\in V_\mu(z)$ has a degree $d$. Let $M' > \text{Max}(M,d)$. Then 
$\sum_{m > M'}x_{i,m}^+z^m.v$ has degree $M'$. We expand it in $z^{-1}$ and we get a series 
$-\sum_{m \leq M'}v_{m}^{(M')}z^m$. We also set $v_{m}^{(M')} = x_{i,m}^+.v$ for $m > M'$.
We get a new family $(v_m^{(M')})_{m\in\mathbb{Z}}$. It does not depend on the choice of $M'$. Indeed, for $M'' > M'$, we have 
$$ -\sum_{m \leq M''}v_m^{(M'')} z^m = \sum_{m> M''} v_m^{(M'')} z^m 
= \sum_{m > M'}v_m^{(M')} z^m - \sum_{M''\geq m > M'}v_m^{(M')}z^m$$
$$= - \sum_{m \leq M'}v_m^{(M')}z^m - \sum_{M''\geq m > M'}v_m^{(M')}z^m = - \sum_{m \leq M''} v_m^{(M')} z^m.$$
Identifying the developments in $z^{-1}$, we get $v_m^{(M'')} = v_m^{(M')}$ for $m \leq M''$, and the developments in $z$
$v_m^{(M'')} = v_m^{(M')}$ for $m > M''$. So we can set $\tilde{x}_{i,m}^+.v = v_m^{(M')}$ which is well defined 
for $m\in\mathbb{Z}$. In the same way, one defines operators $\tilde{x}_{i,m}^-$ on $V_\mu$ for $i\in I$, $m\in\mathbb{Z}$.

These operators are well defined on $\tilde{V}_\mu = V_\mu/V_{<\mu}$. Also $\phi_i^+(z)$ is rational of degree $\alpha_i(\mu)$ on this quotient. Expanding in $z^{-1}$ we get an operator series 
$$\phi_i^-(z)\in z^{\alpha_i(\mu)}\text{End}(\tilde{V}_\mu)[[z^{-1}]].$$

\begin{prop}\label{resfunc} The operators $\tilde{x}_{i,m}^\pm$, $\phi_i^\pm(z)$ constructed above on $\tilde{V}_\mu$ define a structure 
of $\mathcal{U}_q^\mu(\hat{\Glie})$-module on $\tilde{V}_\mu$ which is in the category $\mathcal{O}_\mu$.
\end{prop}

We obtain a restriction functor 
$$\mathcal{R}_\mu : \mathcal{O} \rightarrow \mathcal{O}_\mu.$$

\begin{proof} We have to check the defining relations of $\mathcal{U}_q^\mu(\hat{\Glie})$ are satisfied. For example let us study the relation
$$[x_i^+(z),x_i^-(w)] = \frac{\delta(w z^{-1}) \phi_i^+(w) - \delta(zw^{-1})\phi_i^-(z)}{q_i - q_i^{-1}},$$
which is the most complicated to handle as it involves all operators. We work on a vector $v$ for which we can consider 
$M$ large enough so that $\tilde{x}_{i,m}^\pm.v =  x_{i,m}^\pm.v$ for $i\in I$ and $m\geq M$.
First, we prove exactly as in \cite[Lemma B.2 (c)]{FT} that for 
$$x_i^{\pm,M}(z) = \sum_{m\geq M}x_{i,m}^\pm z^m\text{ , }\phi_i^{\pm,M}(z) = \sum_{m\geq M} \phi_{i,\pm m}z^{\pm m},$$
we have in $\mathcal{U}_q(\hat{\bo})[[z,w]]$ : 
$$[x_i^{+,M}(z), x_i^{-,M}(w)] = \frac{ z^M w^{1 - M} \phi_i^{+,2M}(w) -  w^M z^{1 - M}\phi_i^{+,2M}(z)}{(w - z)(q_i - q_i^{-1})}.$$
Now we expand this in $z^{-1}$ and we get
$$[-\tilde{x}_i^{+,M}(z), x_i^{-,M}(w)] = \frac{z^Mw^{1 - M} \phi_i^{+,2M}(w) - w^M z^{1 - M}  \phi_i^-(z) +w^M z^{1 - M}  \sum_{0\leq m < 2M}\phi_{i,m}^+z^m }{(w - z)(q_i - q_i^{-1})},$$
where we denote
$$\tilde{x}_i^{\pm,M}(z) = \sum_{m < M}\tilde{x}_{i,m}^\pm z^m.$$
This implies
$$[\tilde{x}_i^+(z),x_i^{-,M}(w)] = z^{M-1}w^{1 - M} \frac{\phi_i^+(z) - \phi_i^-(z)}{(q_i - q_i^{-1})}\sum_{r\geq 0} (wz^{-1})^r.$$
In the same way, one gets
$$[\tilde{x}_i^+(z),\tilde{x}_i^{-,M}(w)] = z^M w^{-M}\frac{\phi_i^+(z) - \phi_i^-(z)}{(q_i - q_i^{-1})}\sum_{r \geq 0}(zw^{-1})^r.$$
The sum of the two relations give the correct relation.
\end{proof}

\begin{example} 
(i) Let $V$ be a simple finite-dimensional representation in $\mathcal{O}$. Then, up to a twist, it is a
representation of the quantum affine algebra and the action of the $\phi_i^+(z)$ are of degree $0$. 
This gives the restricted action and $\mathcal{R}_{\mu}(V)$ is simple of dimension $\text{dim}(V)$.

(ii) Consider the prefundamental representation $V = L^{\bo,-}_{1,1}$ in the $sl_2$-case. Then $V = V_{-\omega_1^\vee}$ 
and $\mathcal{R}_{-\omega_1^\vee}(V)$ is simple in $\mathcal{O}_{-\omega_1^\vee}$.

(iii) Consider the prefundamental representation $V = L^{\bo,+}_{1,1}$ in the $sl_2$-case. Then $V = V_{\omega_1^\vee}$.
By construction, $\tilde{x}_1^+(z) = \tilde{x}_1^-(z) = 0$ on $\mathcal{R}_{\omega_1^\vee}(V)$ 
which is semi-simple in $\mathcal{O}_{\omega_1^\vee}$ equal to an infinite direct sum of simple modules
of dimension $1$.
\end{example}

\begin{rem} We get functors 
$$\mathcal{I} : \mathcal{O}^{sh}\rightarrow \mathcal{O} \text{ and }\mathcal{R} : \mathcal{O} \rightarrow \mathcal{O}^{sh}.$$ 
We may wonder if these functors are biadjoint.
\end{rem}

\section{Characters and cluster algebra structures}\label{charclus}

In this section we establish a $q$-character formula for simple finite-dimensional 
representations of shifted quantum affine algebras in terms of the $q$-characters 
of certain simple representations of the quantum affine Borel algebra $\mathcal{U}_q(\hat{\bo})$ 
in the category $\mathcal{O}$ (Theorem \ref{charqf}). Then, we prove 
the results in \cite{HL} imply a description 
of simple finite-dimensional representations of $\mathcal{U}_q^\mu(\hat{sl}_2)$ (Theorem \ref{fact}), 
isomorphisms of Grothendieck rings between categories of representations
 of $\mathcal{U}_q^\mu(\hat{\Glie})$ associated to dominant and anti-dominant $\mu$ 
(Theorem \ref{exdu}), 
and a cluster algebra structure on the Grothendieck ring of finite-dimensional 
representations of shifted quantum affine algebras (Theorem \ref{clsh}). No analogs of these results are known for shifted Yangians.

Note that the structure of $K$-theoretic Coulomb branches of a 3d $N=4$ quiver gauge theory has been studied in the context of cluster theory in \cite{ss}. Here we consider cluster structures emerging from their representation theory.

\subsection{$q$-characters of finite-dimensional representations}

For $i\in I$, let $\chi_i$ be the character of $L^\bo(\Psib_{i,1})$. It is proved in \cite{HJ, FH} that for $a\in \mathbb{C}^*$ :
\begin{equation}\label{cform}\chi_q(L^{\bo,+}_{i,a}) =[\Psib_{i,a}]\chi_i. \end{equation}

\begin{thm}\label{charqf} Let $L(\Psib)$ be a simple finite-dimensional representation of $\mathcal{U}_q^\mu(\hat{\Glie})$. 
The $q$-character of the simple $\mathcal{U}_q(\hat{\bo})$-module  $L^\bo(\Psib)$ is :
$$\chi_q(L^\bo(\Psib)) = \chi_q(L(\Psib))\prod_{i\in I}\chi_i^{\alpha_i(\mu)}.$$
\end{thm}

\begin{rem} This generalizes the $q$-characters formulas (\ref{cform}) for $L^\bo(\Psib_{i,a})$ established in \cite{FH}) and the formula in \cite{HL} for $L^\bo(\Psib_{i,a}^*)$ (see Example \ref{exhl1}) :
$$\chi_q(L^\bo(\Psib_{i,a}^*)) = [\Psib_{i,a}^*]   (1 + A_{i,a}^{-1})   \prod_{j,C_{j,i} < 0} \chi_j =\chi_q(L(\Psib_{i,a}^*))   \prod_{j,C_{j,i} < 0} \chi_j .$$
\end{rem}

\begin{example} For $i\in I$, $r\geq 0$ and $\tilde{\Psib}_{i,1}$ is defined as in (iv) of Example \ref{exqchar}, we have
$$\chi_q(L^\bo(\tilde{\Psib}_{i,1}\Psib_{i,q_i^{-2r}})) = [\tilde{\Psib}_{i,1}\Psib_{i,q_i^{-2r}}] \sum_{0\leq m\leq r}(A_{i,1}A_{i,q_i^{-2}}\cdots A_{i,q_i^{-2m}})^{-1} \prod_{j\neq i}\chi_j^{-C_{i,j}}.$$
Indeed, it can be checked along the lines of (iv) of Example \ref{exqchar} that $L(\tilde{\Psib}_{i,1}\Psib_{i,q_i^{-2r}})$ is $(r+1)$-dimensional with the $q$-character corresponding to the formula above.
\end{example}

\begin{proof} From Theorem \ref{fdclas}, $L(\Psib)$ is a quotient of $L(\Psib_0)*L(\Psib^+)$ where $L(\Psib_0)$ is a finite-dimensional representation of $\mathcal{U}_q(\hat{\Glie})$
and $L(\Psib^+)$ is one-dimensional.

Consider the $\mathcal{U}_q(\hat{\bo})$-module $L^\bo(\Psib_0)\otimes L^\bo(\Psib^+)$ in the category $\mathcal{O}$. Let $v_0$ and $v_+$ be corresponding highest weight vectors. 
By \cite[Lemma 5.6]{FJMM}, the Drinfeld coproduct gives a well-defined action of $\mathcal{U}_q(\hat{\bo})$ on this tensor product. Let us denote by $V$ this representation. 
It is established in \cite[Lemma 5.7]{FJMM} that any non-zero submodule of $V$ is of the form $W\otimes L^\bo(\Psib^+)$ where $W\subset L^\bo(\Psib_0)$ 
is a subspace containing $v_0$.
In particular, the submodule $V'$ of $V$ generated by $v_0\otimes v_+$ is simple isomorphic to $L^\bo(\Psib)$. There is a corresponding 
$W\subset L^\bo(\Psib_0)$ : 
$$L^\bo(\Psib)\simeq W\otimes L^\bo(\Psib^+)= V' \subset V.$$
By construction, $W$ is stable by the action of the $x_{i,m}^+$, $\phi_{i,m}^\pm$ for $i\in I$, $m\geq 0$. 
Then by Formula (\ref{cform}), We have 
$$\chi_q(L^\bo(\Psib)) = \chi_q(W) \chi_q(L^\bo(\Psib^+)) = \chi_q(W)[\Psib^+]\prod_{i\in I}\chi_i^{\alpha_i(\mu)}.$$

Consider the restricted representation $\mathcal{R}_\mu(V)$ in the category $\mathcal{O}_\mu$. It admits $L(\Psib)$ as a subquotient. 

By construction, $W\otimes v_+$ is stable for the action of the $\tilde{x}_{i,m}^+, \phi_{i,m}^+$. Besides, the $\phi_i^+(z)$ have degree $0$ on $W$ (as a subspace of a $\mathcal{U}_q(\hat{\mathfrak{g}})$-module). This implies $(W\otimes v_+)_\mu = W\otimes v_+$ and $(W\otimes v_+)_{< \mu} = \{0\}$. This implies we have a subspace $W\otimes v_+\subset \mathcal{R}_\mu(V')$.

Now, for $m > \alpha_i(\mu)$, we have $x_{i,m}^-.v_+ = 0$ (this follows from $x_{j,r}^+x_{i,m}^- .v_+= \delta_{i,j}(q_i - q_i^{-1})^{-1} \phi_{i,m+r}^+.v_+$ for any 
$r\geq 0$, $j\in I$). Hence 
\begin{equation}\label{dec}x_{i,m}^-.(w\otimes v_+) = (\sum_{0\leq r\leq \alpha_i(\mu)} \Psi_{i,r}^+x_{i,m-r}^- )w\otimes v_+,\end{equation}
where $\Psi_{i,r}^+$ is the eigenvalue of $\phi_{i,r}^+$ on $v_+$.
This implies that $W\otimes v_+$ is stable for the action of $\mathcal{U}_q^\mu(\hat{\Glie})$. 
So, we have a submodule $W\otimes v_+\subset \mathcal{R}_\mu(V')$.

To conclude, it suffices to prove this module is simple isomorphic to $L(\Psib)$.

Let us prove that the $\mathcal{U}_q^\mu(\hat{\Glie})$-module $W\otimes v_+$ is generated by $v_0\otimes v_+$. 

Let $w\otimes v_+\in W$. We know there is $x$ in the negative subalgebra $\mathcal{U}_q^-(\hat{\bo})\subset \mathcal{U}_q(\hat{\bo})$ 
so that $x.(v_0\otimes v_+) = w\otimes v_+$. Although all Drinfeld generators are not in $\mathcal{U}_q(\hat{\bo})$, $x$ can be written in $\mathcal{U}_q(\hat{\mathfrak{g}})$ as an algebraic combination of various $x_{i,m}^-$, and from the Drinfeld coproduct formula. We have 
$$x. (v_0\otimes v_+) \in (x'.v_0\otimes v_+) + \text{ additional terms}$$
where $x'$ is obtained from $x$ by replacing each $x_{i,m}^-$ by $x_{i,m}^{-,'} = \sum_{0\leq r\leq \alpha_i(\mu)} \Psi_{i,r}^+x_{i,m-r}^-$ (see formula (\ref{dec})) and the additional terms have a right factor of weight strictly lower than the weight of $v_+$. Hence 
$$w\otimes v_+ = x. (v_0\otimes v_+) = (x'.v_0)\otimes v_+.$$
This implies that 
$$W\otimes v_+ \subset (<x_{i,m}^{-,'}>_{i\in I, m\in\mathbb{Z}}.v_0)\otimes v_+ = (<x_{i,m}^{-,'}>_{i\in I, m\geq \alpha_i(\mu)}.v_0)\otimes v_+ \subset \mathcal{U}_q^\mu(\hat{\Glie}).(v_0\otimes v_+).$$

We have proved that $W \otimes v_+ = \mathcal{U}_q^\mu(\hat{\Glie}) .(v_0\otimes v_+)$ is an highest $\ell$-weight module. To conclude it is simple, we just have to prove it has no primitive vector except the highest weight vectors. 
Recall that the there is a non-zero polynomial $P(z)$ so that $P(z).x_i^\pm(z) = 0$ on
the finite-dimensional representation $L(\Psib_0)$ (see Proposition \ref{zero} for instance).
If there is $w\otimes v_+\in W\otimes v_+$ so that $\tilde{x}_{i,m}^+ (w\otimes v_+) = 0$ for any $m\geq \alpha_i(\mu)$, $i\in I$, 
then $x_{i,m}^+.w = 0$ for any $m\in \mathbb{Z}$,$i\in I$, and $w$ is an highest weight vector. 
\end{proof}

\subsection{Description of simple finite-dimensional representations of $\mathcal{U}_q^\mu(\hat{sl}_2)$}

We get  a complete description of all simple finite dimensional representations of shifted quantum affine algebras $\mathcal{U}_q^\mu(\hat{sl}_2)$, 
that is simple objects in the category $\mathcal{C}^{sh}\subset \mathcal{O}^{sh}$ of finite-dimensional representations.

Suppose that $\Glie = sl_2$. A description of simple modules of the category $\mathcal{O}^+$ was given in \cite[Section 7.3]{HL}. 

For $k\geq 0$ and $a\in\mathbb{C}^*$, we have the Kirillov-Reshetikhin (KR) module 
$$W_{k,a} = L(\overline{k\omega}\Psib_{aq^{-1}}\Psib_{aq^{2k - 1}}^{-1}) = L(Y_aY_{aq^2}\cdots Y_{aq^{2(k-1)}}).$$ 
It is a representation of $\mathcal{U}_q(\hat{sl}_2)$ of dimension $k+1$ obtained by evaluation from a $\mathcal{U}_q(sl_2)$-module.

A $q$-set is a subset of $\CC^*$ of the form $\{aq^{2r}\mid R_1\leq r\leq R_2\}$
for some $a\in\CC^*$ and $R_1\leq R_2\in\ZZ\cup \{-\infty,+\infty\}$.
The modules $W_{k,a}$, $W_{k',b}$ are said to be in \emph{special position} if the union of $\{a,aq^2,\cdots, aq^{2(k-1)}\}$
and $\{b,bq^2, \cdots , bq^{2(k'-1)}\}$ is a $q$-set which contains both properly.
The module $W_{k,a}$ and the prefundamental representation $L_b^+$ are said to be in special
position if the union of $\{a,aq^2,aq^4,\cdots , aq^{2(k-1)}\}$ and $\{bq, bq^3, bq^5,\cdots \}$ is a $q$-set which contains both properly.
Two positive prefundamental representations are never in special position.
Two such representations are in \emph{general position} if they are not in special position.

The invertible elements in the category $\mathcal{C}^{sh}$ are the $1$-dimensional constant simple representations $[\omega]$.

From Theorem \ref{charqf}, we have now the following direct consequence of \cite[Theorem 7.9]{HL}.

\begin{thm}\label{fact} Suppose that $\Glie = sl_2$. The prime simple objects in the category $\mathcal{C}^{sh}$ are the 
positive prefundamental representations and the KR-modules. Any simple object in $\mathcal{C}^{sh}$ can be factorized in a
unique way as a fusion product of prefundamental representations and KR-modules (up to permutation of the factors and to invertibles).
Moreover, such a fusion product is simple if and only all its factors are pairwise in general position.
\end{thm}

\begin{rem}
{\rm
(i) This is a generalization of the factorization of simple representations in the category $\mathcal{C}$ of finite-dimensional representations 
of $\mathcal{U}_q(\hat{sl}_2)$ by Chari-Pressley \cite{CP}.

(ii) All simple finite-dimensional representations can be factorized in a unique way into a fusion product of a simple finite-dimensional representation of $\mathcal{U}_q^0(\hat{sl}_2)$ 
and a one-dimensional representation.

(iii) This result for $\Glie = sl_2$ implies that all simple finite-dimensional representations are real and that their factorization into prime representations is unique.
}
\end{rem}

\subsection{Grothendieck ring isomorphisms}

Let us consider completed tensor products $\hat{\otimes}_{\mathbb{Z}}$ as in \cite[Section 4.1]{HL}.
We have the following consequence of Theorem \ref{charqf}.

\begin{cor}\label{firstiso} There is a ring isomorphism
$$K_0(\mathcal{C}^{sh})\hat{\otimes}_{\mathbb{Z}} \mathcal{E} \simeq K_0(\mathcal{O}^+)$$
which preserves the classes of simple objects.
\end{cor}

\begin{proof} For $L$ a finite-dimensional representation of $\mathcal{U}_q^\mu(\hat{\Glie})$, we assign to the class of $L$ 
in $K_0(\mathcal{C}^{sh})$ the class in $K_0(\mathcal{O}^+)$ of $q$-character 
$$\chi_q(L)\prod_{i\in I}\chi_i^{\alpha_i(\mu)}.$$
This defines an injective ring morphism from $K_0(\mathcal{C}^{sh})\hat{\otimes}_{\mathbb{Z}} \mathcal{E}$ to $K_0(\mathcal{O}^+)$ which sends a simple class to a simple class. As $K_0(\mathcal{O}^+)$ is topologically 
generated by the $[L^\bo(\Psib)]$ where $\Psib$ is a dominant $\ell$-weight, the morphism is surjective.
\end{proof}

It is proved in \cite{HL} that there is an isomorphism of Grothendieck rings
$$D : K_0(\mathcal{O}^+) \rightarrow K_0(\mathcal{O}^-)$$
which preserves dimensions, characters and so that $D([L^\bo(\Psib)]) = [L^\bo(\Psib^{-1})]$. 
Note however that it is not compatible with $q$-characters.

Let $\mathcal{O}^{sh,+}$ (resp. $\mathcal{O}^{sh,-}$) be the subcategory of representations in $\mathcal{O}^{sh}$ 
whose simple constituents have a highest $\ell$-weight $\Psib$ so that $\Psib$ is dominant (resp. $\Psib^{-1}$ is dominant).
This is motivated by analogous categories of $\mathcal{U}_q(\hat{\bo})$-modules (see Remarks \ref{exom}, \ref{refy}).

 Note that all simple modules in $\mathcal{O}^{sh,+}$ are finite-dimensional and that $\mathcal{C}^{sh}\subset \mathcal{O}^{sh,+}$.

\begin{thm}\label{exdu} The categories $\mathcal{O}^{sh,+}$, $\mathcal{O}^{sh,-}$ are stable by fusion product and we have a ring isomorphism which preserves simple classes
$$\bigoplus_{\mu\in \Lambda^+} K_0(\mathcal{O}_\mu) \supset K_0(\mathcal{O}^{sh,+})\simeq K_0(\mathcal{O}^{sh,-})\subset \bigoplus_{\mu\in -\Lambda^+} K_0(\mathcal{O}_\mu).$$
\end{thm}

\begin{proof} The stability of $\mathcal{O}^{sh,+}$ by fusion product follows from the stability of the 
category $\mathcal{C}^{sh}$ of finite-dimensional representations as both categories have the same simple
objects. We have
$$K_0(\mathcal{O}^{sh,+}) = K_0(\mathcal{C}^{sh})\hat{\otimes}_{\mathbb{Z}} \mathcal{E}.$$
The stability of $\mathcal{O}^{sh,-}$ by fusion product is clear as the simple objects in $\mathcal{O}^-$ and 
$\mathcal{O}^{sh,-}$ have the same $q$-character (Corollary \ref{still}). Hence we have an isomorphism
$$K_0(\mathcal{O}^{sh,-})\simeq K_0(\mathcal{O}^-).$$
Then we can use the isomorphism in Corollary \ref{firstiso}, combined with the isomorphism $D$ and the previous line :
$$K_0(\mathcal{O}^{sh,+}) = K_0(\mathcal{C}^{sh})\hat{\otimes}_{\mathbb{Z}} \mathcal{E} \simeq K_0(\mathcal{O}^+).$$
\end{proof}

\subsection{Cluster algebra structure}

A certain monoidal subcategory 
$$\mathcal{O}_{2\mathbb{Z}}^+\subset \mathcal{O}^+\subset\mathcal{O}$$ 
of representations of $\mathcal{U}_q(\hat{\bo})$ 
is introduced in \cite{HL}. It is defined as the subcategory of representations in $\mathcal{O}^+$ whose simple constituents have a highest $\ell$-weight 
$\Psib$ such that the roots and the poles of $\Psi_i(z)$ are of the form $q^r$ where $(i,r)$ belong to certain remarkable $V\subset I\times \mathbb{Z}$.

Its Grothendieck ring $K_0(\mathcal{O}_{2\mathbb{Z}}^+)$ captures the combinatorics of $K_0(\mathcal{O}^+)$. Moreover the main Theorem of \cite{HL} 
is a ring isomorphism
$$K_0(\mathcal{O}_{2\mathbb{Z}}^+)\simeq \mathcal{A}\hat{\otimes}_{\mathbb{Z}}  \mathcal{E},$$
where $\mathcal{A}$ is a cluster algebra and the classes of prefundamental representations $[L_{i,q^r}^{\bo, +}]_{(i,r)\in V}$ in $\mathcal{O}_{2\mathbb{Z}}^+$ form an initial seed.

Now consider the subcategory 
$$\mathcal{C}^{sh}_{2\mathbb{Z}} \subset \mathcal{C}^{sh}\subset \mathcal{O}^{sh}$$
of finite-dimensional representations whose simple constituents have a highest $\ell$-weight 
$\Psib$ such that the roots and the poles of $\Psi_i(z)$ are of the form $q^r$ where $(i,r)\in V$.

Similarly, we have also corresponding categories $\mathcal{O}^{sh,\pm}_{2\mathbb{Z}}\subset \mathcal{O}^{sh,\pm}$.

\begin{thm}\label{clusth} We have ring isomorphisms
$$K_0(\mathcal{O}^{sh,+}_{2\mathbb{Z}}) \simeq \mathcal{A}\hat{\otimes}_{\mathbb{Z}}  \mathcal{E} \simeq K_0(\mathcal{O}^{sh,-}_{2\mathbb{Z}}),$$
with classes of prefundamental representations corresponding to an initial seed.
\end{thm}

Let us recall that a simple object is said to be real if its fusion square is simple.

\begin{conj}\label{clconj} The classes of real simple objects in $K_0(\mathcal{O}^{sh,+}_{2\mathbb{Z}})$ (resp. $K_0(\mathcal{O}^{sh,-}_{2\mathbb{Z}})$) get identified with cluster monomials.
\end{conj}

By the results in the present paper, in particular Corollary \ref{firstiso}, this Conjecture \ref{clconj} is equivalent to \cite[Conjecture 7.12]{HL}. 
Moreover, by \cite[Theorem 7.12]{HL}, \cite[Conjecture 7.12]{HL} is equivalent to \cite[Conjecture 5.2]{HLJEMS}. Then a part of this 
Conjecture \ref{clconj} was established in \cite{Q} for $ADE$ types, for general types recently in \cite{kkop, kkop2}. Combining these results, one gets the following.

\begin{thm}\label{clsh} (i) We have an algebra isomorphism
$$\mathcal{A}\simeq K_0(\mathcal{C}^{sh}_{2\mathbb{Z}}).$$
(ii) The cluster monomials in $\mathcal{A}$ are real simple objects in $K_0(\mathcal{C}^{sh}_{2\mathbb{Z}})$.
\end{thm}

\begin{proof} By the discussion above, (ii) is known. The arguments in \cite[Proposition 6.1]{HL} imply that 
$$K_0(\mathcal{C}^{sh}_{2\mathbb{Z}})\subset\mathcal{A}.$$
As cluster monomials generate a cluster algebra, now (i) follows from (ii).
\end{proof}

\begin{rem} In the $sl_2$-case, Conjecture \ref{clconj} is proved in \cite[Theorem 7.11]{HL}. \end{rem}

We will study again these and other cluster algebra structures related to the representation theory of shifted quantum affine
algebras in another work.

\section{Cartan-Drinfeld series and Baxter polynomiality}\label{cds}

Adjoint versions of shifted quantum affine algebras are defined as the usual adjoint versions of quantum affine algebras
by adding Cartan generators corresponding to fundamental weights. We discuss series of Cartan-Drinfeld elements $Y_i^\pm(z)$ and 
$T_i^\pm(z)$ ($i\in I$) introduced respectively in \cite{Fre} in the study of transfer-matrices of finite-dimensional representations of 
quantum affine algebras and in \cite{HJ} as limits of transfer-matrices of prefundamental representations of quantum affine Borel algebras.

As the main result of this section (Theorem \ref{newpol}), we establish the rationality of $Y_i^\pm(z)$ (resp. the polynomiality 
of $(T_i^\pm(z))^{\mp 1}$) on a simple representation in the category $\mathcal{O}_\mu$ (up to the highest eigenvalue). 
The proof is partly based on the Cartan-Drinfeld polynomiality established in \cite{FH} as a limit of Baxter polynomiality of quantum integrable models.
We also obtain the equality up to a scalar multiple of the rational operators associated respectively to $Y_i^+(z)$ and $Y_i^-(z)$. 

These methods and results are also new in the case of ordinary quantum affine algebras
or shifted Yangians. 

\subsection{Adjoint versions} 
Adjoint versions for quantum affine algebras are used in the literature 
(see \cite{Fre} for instance where additional elements are denoted by $\tilde{k}_i$).

Fix $\mu\in\Lambda$. The adjoint version $\mathcal{U}_q^{\mu, ad}(\hat{\Glie})$ of 
the shifted quantum affine algebra in \cite{FT} is a slight extension of $\mathcal{U}_q^{\mu}(\hat{\Glie})$. New generators $(\overline{\phi}_i^+)^{\pm 1}$, $(\overline{\phi}_i^-)^{\pm 1}$ are added satisfying 
$$\prod_{j\in I}(\overline{\phi}_j^+)^{C_{j,i}} = \phi_{i,0}^+\text{ and }\prod_{j\in I}(\overline{\phi}_j^-)^{C_{j,i}} = \phi_{i,\alpha_i(\mu)}^-,$$
and satisfying the analogs of the quasi-commutations relations (\ref{un}), (\ref{deux}), that is, for $i,j\in I$, $r\in\mathbb{Z}$ :
$$[\overline{\phi}_i^\pm,\overline{\phi}_j^\pm] 
= [\overline{\phi}_i^\pm,\overline{\phi}_j^\mp] 
= [\phi_{i,r}^\pm,\overline{\phi}_{j}^\pm] 
= [\phi_{i,r}^\pm,\overline{\phi}_{j}^\mp] = 0,$$
$$\overline{\phi}_i^+ x_{j,r}^{\pm} = q_i^{\pm \delta_{i,j} }x_{j,r}^{\pm} \overline{\phi}_i^+
\text{ and }
\overline{\phi}_i^- x_{j,r}^{\pm} = q_i^{\mp \delta_{i,j}}x_{j,r}^{\pm} \overline{\phi}_i^-.$$

\begin{rem}\label{newc} (i) For $i\in I$, $\overline{\phi}_{i}^+\overline{\phi}_i^-$ is central in $\mathcal{U}_q^{\mu, ad}(\hat{\Glie})$.

(ii) There is a group of automorphisms of $\mathcal{U}_q^{\mu,ad}(\hat{\Glie})$ isomorphic to $(\mathbb{Z}/2\mathbb{Z})^n$ : for $(\epsilon_1,\cdots, \epsilon_n)\in\{\pm 1\}^n$, there is a unique automorphism so that for $i\in I$ and $m\in\mathbb{Z}$ : 
$$\overline{\phi}_i^\pm\mapsto \epsilon_i \overline{\phi}_i^\pm\text{ , }\phi_{i,m}^\pm \mapsto \eta_i \phi_{i,m}^\pm\text{ , }x_{i,m}^- \mapsto \eta_i x_{i,m}^-\text{ , }x_{i,m}^+\mapsto x_{i,m}^+\text{ where }\eta_i = \prod_{j\in I}  \epsilon_j^{C_{j,i}}\in\{\pm 1\}.$$ 
This induces a group of automorphisms of $\mathcal{U}_q^\mu(\hat{\Glie})$ that we call sign-twist (the order is $2^n$, except $2^{n-1}$ in type $B_n$ and $1$ in type $A_1$).
\end{rem}

The representation theory of $\mathcal{U}_q^{\mu, ad}(\hat{\Glie})$ is a slight modification of the representation theory of $\mathcal{U}_q^{\mu}(\hat{\Glie})$.
Indeed, a representation in $\mathcal{O}_\mu$ has a structure of $\mathcal{U}_q^{\mu, ad}(\hat{\Glie})$-module, but it is not unique. 
A representation of $\mathcal{U}_q^{\mu, ad}(\hat{\Glie})$ is said to be in the category $\mathcal{O}_{\mu, ad}$ if it is in the
category $\mathcal{O}_\mu$ as a $\mathcal{U}_q^{\mu}(\hat{\Glie})$-module. A module in $\mathcal{O}_{\mu, ad}$ is simple if and only if it is simple
as a $\mathcal{U}_q^{\mu}(\hat{\Glie})$-module. Hence the simple module in $\mathcal{O}_{\mu, ad}$ are parametrized by triples 
$$(\Psib, \overline{\omega}^+, \overline{\omega}^-) \in \mathfrak{r}_\mu \times \tb^*\times \tb^*$$ 
satisfying : 
$$\prod_{j\in I}(\overline{\omega}^+(j))^{C_{j,i}} = \Psi_i(0)\text{ and }\prod_{j\in I}(\overline{\omega}^-(j))^{C_{j,i}} = (z^{-\alpha_i(\mu)}\Psi_i(z)) (\infty)\text{ for $i\in I$.}$$
Such $\overline{\omega}^\pm$ are said to be compatible with $\Psib$. The corresponding simple representation is denoted by $L(\Psib,\overline{\omega}^+, \overline{\omega}^-)$. 
The dimensions of its weight spaces are the same as those of $L(\Psib)$ and its structure can be completely described from the structure of $L(\Psib)$, and $\overline{\omega}^\pm$.

\begin{rem}\label{groupk} $\overline{\omega}^+$ (resp. $\overline{\omega}^-$) is uniquely determined by $\Psib$ up to an element $(c_1,\cdots,c_n)$ in the group $K\subset (\mathbb{C}^*)^n$ of solutions of the equations $\prod_{j\in I}c_j^{C_{j,i}} = 1$.
\end{rem}

\begin{example}\label{follow} The algebra $\mathcal{U}_q^{-\omega_1^\vee, ad}(\hat{sl}_2)$ has additional generators $\overline{\phi}^\pm$ with $(\overline{\phi}^+)^2 = \phi_0^+$ and 
$(\overline{\phi}^-)^2 = \phi_{-1}^-$. For $a, b \in\mathbb{C}^*$, the representation $L(b(1 - az)^{-1})$ is in the category $\mathcal{O}_{-\omega_1^\vee}$. 
For $\alpha$ (resp. $\beta$) a square root of $a$ (resp. $b$), we have compatible 
$\overline{\omega}^+ = \beta$ and $\overline{\omega}^- = i\beta\alpha^{-1}$ which give 
a structure of $\mathcal{U}_q^{-\omega_1^\vee, ad}(\hat{sl}_2)$-module on $L(b(1 - az)^{-1})$.
\end{example}

\subsection{Fundamental Cartan-Drinfeld series}
We consider a collection of Cartan-Drinfeld series which appear naturally from 
$R$-matrices and transfer-matrices in \cite{Fre}. For $i\in I$, set
$$Y_i^\pm (z) = \overline{\phi}_i^\pm \text{exp}\left( \pm (q- q^{-1}) \sum_{m > 0} \tilde{h}_{i, \pm m} z^{ \pm m}  \right),$$
$$\tilde{h}_{i,m} = \sum_{j\in I} [r_j]_q \tilde{C}_{j,i}(q^m) h_{j, m}\text{ for $m \neq 0$,}$$
where $\tilde{C}(z)$ is the inverse of the quantum Cartan matrix $C(q)$ (invertible for a generic $q$).

By \cite[Formula (4.9)]{Fre} (see formula (\ref{ayf}) above), we have
\begin{equation}\label{ya}z^{-\alpha_i(\mu)\delta{\pm, -}}\phi_i^\pm (z) = H_i(Y_1^\pm(z),\cdots, Y_n^\pm(z))\end{equation}
where $H_i(Y_1^\pm(z),\cdots, Y_n^\pm(z))$ is set to be equal to 
$$\frac{Y_i^\pm(zq_i^{-1})Y_i^\pm(zq_i)}
{\prod_{j\in I, C_{j,i} = - 1} Y_j^\pm (z) \prod_{j\in I, C_{j,i} = -2} Y_j^\pm (zq^{-1})Y_j^\pm (zq) 
\prod_{j\in I, C_{j,i} = -3} Y_j^\pm(q^{-2}z) Y_j^\pm( z) Y_j^\pm(q^2 z)}.$$

\begin{rem}\label{unisol} Note that for Laurent formal power series $d_i(z)\in A((z))$ with coefficients in a commutative algebra $A$, if the system of $n$-functional equations 
$$d_i(z) = H_i(s_1(z),\cdots, s_n(z))\text{ , $i\in I$,}$$ 
has a solution as a formal Laurent power series, it is unique up to constant factors.
\end{rem}

The following Cartan-Drinfeld series are introduced in \cite{FH} as limits of 
transfer-matrices associated to prefundamental representations : 
$$T_i^\pm (z)  =  \text{exp}\left(\mp \sum_{m > 0}z^{\mp m} \frac{\tilde{h}_{i,\pm m}}{[r_i]_q[ m]_{q_i}}  \right).$$
We have
\begin{equation}\label{yt}Y_i^\pm(z) = \overline{\phi}_i^\pm \frac{T_i^\pm (z^{-1}q_i^{\pm 1})}{T_i^\pm (z^{-1}q_i^{\mp 1})}.\end{equation}
For $i\in I$ and $a\in\mathbb{C}^*$, recall the $\ell$-weight $\tilde{\Psib}_{i,a}$ in Example \ref{exqchar}. Motivated by the next result, we set
$$\Lambda_{i,a} = \tilde{\Psib}_{i,aq_i^{-1}}^{-1}\Psib_{i,aq_i}$$
$$= \Psib_{i,aq_i^{-1}}\Psib_{i,aq_i}\left(\prod_{j,C_{i,j} = -1}\Psib_{j,a}^{-1} \right)\left(\prod_{j,C_{i,j} = -2}\Psib_{j,aq^{-1}}^{-1}\Psib_{j,aq}^{-1} \right)\left(\prod_{j,C_{i,j} = -3}\Psib_{j,aq^{-2}}^{-1}\Psib_{j,a}^{-1}\Psib_{j,aq^2}^{-1} \right).$$
In particular we have  
\begin{equation}\label{alambda}A_{i,a} = \overline{\alpha_i}\Lambda_{i,aq_i^{-1}}/\Lambda_{i,aq_i}.\end{equation}

\begin{rem}\label{reml} The degrees of the coordinates of $\Lambda_{i,a}$ form the 
simple roots $\alpha_i^\vee$ of the Langlands dual Lie algebra ${}^L\mathfrak{g}$, in opposition to the powers of the monomial
$A_{i,a}$ in terms of the $Y_{j,b}$ which give the simple roots of $\mathfrak{g}$. 
This is an indication of the important role played by the Langlands dual Lie algebra ${}^L\mathfrak{g}$ in the
following (see Section \ref{maincsec}).
\end{rem}

\begin{lem}\label{egw} Consider a rational $\ell$-weight $\Psib = \Psib(0)\prod_{i\in I, a\in\CC^*} \Psib_{i,a}^{\nu_{i,a}}$. For $i\in I$, the corresponding eigenvalue of $(\overline{\phi}_i^\pm)^{-1}Y_i^\pm(z)$ is equal to   
$$\overline{Y}_{i,\Psib}^\pm(z) 
 = \text{exp}\left( \sum_{j\in I, m > 0,a\in\mathbb{C}^*}\tilde{C}_{j,i}(q^m) \nu_{j,a} \frac{a^{\pm m}}{-m}  z^{\pm m}\right).$$
The following are equivalent : 

(i) for any $i\in I$, $\overline{Y}_{i,\Psib}^+(z)$ is rational.

(ii) for any $i\in I$, $\overline{Y}_{i,\Psib}^-(z)$ is rational.

(iii) $\Psib (\Psib(0))^{-1}$ is a Laurent monomial in the $\Lambda_{i,a}$, $i\in I, a\in\mathbb{C}^*$.

\noindent Then $\overline{Y}_{i,\Psib}^+(z)$ and $z^{\omega_i(\mu)}\overline{Y}_{i,\Psib}^-(z)$ coincide as rational fractions up to a constant.
\end{lem}

\begin{proof} The formula for $\overline{Y}_{i,\Psib}^\pm(z)$ is clear as the eigenvalue of $h_{i,m}$ associated to $\Psib$ is 
$$- \sum_{a\in \mathbb{C}^*} \frac{\nu_{i,a} a^m }{ m(q_i - q_i^{-1})}\text{ for $m\in\mathbb{Z}\setminus\{0\}$.}$$
Now suppose that (i) is satisfied. There are $v_{i,b}\in\mathbb{Z}$ so that for any $i\in I$, $m\in\mathbb{Z}\setminus\{0\}$ : 
$$ \sum_{j\in I,a\in\mathbb{C}^*}\tilde{C}_{j,i}(q^m) \nu_{j,a} \frac{a^{m}}{m}   = \sum_{b\in\mathbb{C}^*} \frac{v_{i,b} b^m}{m} .$$
Note that finitely many $v_{i,b}$ are non-zero. We obtain for any $k\in I$ : 
$$\sum_{a\in\mathbb{C}^*} \nu_{k,a} a^m = \sum_{i\in I,b\in\mathbb{C}^*} C_{i,k}(q^m) v_{i,b} b^m,$$ 
$$\Psib  (\Psib(0))^{-1} = \prod_{k\in I, b\in\mathbb{C}^*} \Lambda_{k,b}^{v_{k,b}}.$$
Hence we get (iii). The same computation gives that (ii) implies (iii), and that (iii) implies (i) or (ii). 

\noindent To conclude, let us suppose that the conditions are satisfied. From (\ref{ya}), we have 
$$\text{deg}(\overline{Y}_{i,\Psib}^+(z)) = \omega_i(\mu)\text{ and }\text{deg}(\overline{Y}_{i,\Psib}^-(z)) = 0.$$
The $v_{i,b}$ are well-defined from (iii) as the powers of the $\Lambda_{i,b}$ 
in the factorization of $\Psib (\Psib(0))^{-1}$. Then from the computations above : 
$$\overline{Y}_{i,\Psib}^+(z) = \prod_{b\in\mathbb{C}^*} (1 - zb)^{v_{i,b}} 
= (-z)^{\omega_i(\mu)}\prod_{b\in\mathbb{C}^*}b^{v_{i,b}} (1 - z^{-1}b^{-1})^{v_{i,b}} 
=((-z)^{\omega_i(\mu)}\prod_{b\in\mathbb{C}^*}b^{v_{i,b}}) \overline{Y}_{i,\Psib}^-(z).$$
\end{proof}

\begin{rem} This statement can be seen as a generalization of \cite[Lemma 5]{Fre} where the case when $\Psib$ is a Laurent monomial in the $Y_{i,a}$ is considered.
\end{rem}

With the same notations as in Lemma \ref{egw}, the eigenvalue\footnote{This is consistent with the eigenvalue computed in \cite[(5.20)]{FH} when $\Psib$ is a Laurent monomial in the $Y_{j,b}$, except that there is a misprint in that paper : $\tilde{C}_{i,j}(q^m)$ there should be $\tilde{C}_{j,i}(q^m)$).} of $T_i^\pm(z)$ associated to $\Psib$ is 
$$T_{i,\Psib}^\pm(z) = \text{exp}\left(  \sum_{j\in I, m > 0, a\in\mathbb{C}^*} \frac{z^{\mp m} \tilde{C}_{j,i}(q^m)}{(q_i^m - q_i^{-m})m} a^{\pm m} \nu_{j,a} \right).$$

\subsection{Rationality and polynomiality}

Consider $W = L(\Psib)$ simple in the category $\mathcal{O}_\mu$. 
Let $\omega = \Psib(0)$ be its 
highest weight and $w$ be a highest weight vector of $W$. 

Let $\Psib'$ be an $\ell$-weight space of $W$. We have proved in Theorem \ref{partialo} that 
there are $i_1,\cdots, i_R\in I$, $a_1,\cdots, a_R\in\mathbb{C}^*$ so that
$$\Psib' = \Psib A_{i_1,a_1}^{-1}\cdots A_{i_R,a_R}^{-1}.$$
The same computation as for \cite[Proposition 5.8]{FH} gives the following.

\begin{prop}\label{egw2} The eigenvalue of $T_i^\pm(z)$ on $W_{\Psib'}$ is 
$$T_{i,\Psib'}^\pm(z) = T_{i,\Psib}^\pm(z) \times \prod_{1\leq k\leq R,i_k = i} (1 - (z a_k^{-1})^{\mp 1})^{\mp 1}.$$
\end{prop}

\begin{rem}\label{invsy} The formula in \cite[Proposition 1.6]{C} to define an involution of $\mathcal{U}_q(\hat{\Glie})$ also defines an involution
$$\sigma : \mathcal{U}_q^{\mu,ad}(\hat{\Glie}) \rightarrow \mathcal{U}_q^{-\mu,0,ad}(\hat{\Glie})\simeq \mathcal{U}_q^{\mu,ad}(\hat{\Glie})$$
so that for $i\in I$, $m\in\mathbb{Z}$, $r\in\mathbb{Z}\setminus\{0\}$ : 
$$\sigma(x_{i,r}^\pm) = x_{i,-r - \delta_{\pm, - }\alpha_i(\mu)}^\mp\text{ , }\sigma(h_{i,m}) = - h_{i,-m}\text{ , }\sigma(\overline{\phi}_i^\pm) =\overline{\phi}_i^\mp.$$
For $W$ a representation of $\mathcal{U}_q^{\mu,ad}(\hat{\Glie})$, we denote by $W^\sigma$ its twist by $\sigma$. Note that we have :
$$\sigma(\phi_i^+(z)) = z^{\alpha_i(\mu)}\phi_i^-(z^{-1})\text{ , }\sigma(Y_i^+(z)) = Y_i^-(z^{-1})\text{ , }\sigma(T_i^+(z)) = T_i^-(z^{-1}).$$ 
\end{rem}

\begin{example}\label{longex} The representation  $W = L(Y_1^2)$ of $\mathcal{U}_q(\hat{sl}_2)$ was studied in the Example of \cite[Section 5.8]{FH}. 
It is a simple representation of $\mathcal{U}_q^{0,ad}(\hat{sl}_2)$  with parameter $\left(q^2\frac{(1 - zq^{-1})^2}{(1 - zq)^2},q, q^{-1}\right)$. 
It has a weight space of weight $0$ of dimension $2$. In a slight modification of the basis in \cite{FH}, the matrix of $T^-(z)/T_{\Psib}^- (z)$ and of $Y^-(z)/Y_{\Psib}^- (z)$  are respectively 
$$\begin{pmatrix}1 & 0 & 0 & 0\\ 0 &  1 - zq^{-1} & z & 0 \\  0 & 0 & 1 - zq^{-1} & 0 \\ 0 & 0 & 0 & (1 - zq^{-1})^2 \end{pmatrix}
\text{ , }\begin{pmatrix}1 & 0 & 0 & 0\\ 0 &  \frac{q - z^{-1}q^{-1}}{1 - z^{-1}} & \frac{z^{-1}(1 - q^2)}{(1 - z^{-1})^2}\ & 0 
\\  0 & 0 &  \frac{q - z^{-1}q^{-1}}{1 - z^{-1}} & 0 \\ 0 & 0 & 0 & \frac{(q - z^{-1}q^{-1})^2}{(1 - z^{-1})^2} \end{pmatrix},$$ 
$$\text{ for }\frac{T_{\Psib}^-(z^{-1}q^{-2})}{T_{\Psib}^-(z^{-1}q^2)} = \left(\frac{1 - z^{-1} q}{1 - z^{-1}q^{-1}}\right)^2\text{ , }Y_{\Psib}^-(z) = \frac{q^{-1} T_{\Psib}^-(z^{-1}q^{-1})}{T_{\Psib}^-(z^{-1}q)}.$$
As $W^\sigma\simeq L(Y_{q^{-2}}^2)$, the matrix of $T^+(z)/T_{\Psib}^+(z)$ and $Y^+(z)/Y_{\Psib}^+(z)$ are respectively 
$$\begin{pmatrix}1 & 0 & 0 & 0\\ 0 &  \frac{1}{1 - z^{-1}q} & \frac{z^{-1}q^2}{(1 - z^{-1}q)^2} & 0 \\  0 & 0 & \frac{1}{1 - z^{-1}q} & 0 \\ 0 & 0 & 0 & \frac{1}{(1 - z^{-1}q)^{2}} \end{pmatrix}
\text{ , }\begin{pmatrix}1 & 0 & 0 & 0\\ 0 & \frac{q^{-1} - zq}{1 - z} & \frac{z  (1 - q^2)}{(1 - z)^2}\ & 0 
\\  0 & 0 &  \frac{q^{-1} - zq}{1 - z} & 0 \\ 0 & 0 & 0 & \frac{(q^{-1} - zq)^2}{(1 - z)^2} \end{pmatrix},$$ 
$$\text{ for }T_{\Psib}^+(z) = T_{\Psib}^-(z^{-1}q^2)(1 - z^{-1}q)^2\text{ , }Y_{\Psib}^+(z) = q \frac{T_{\Psib}^+(z^{-1}q)}{T_{\Psib}^+(z^{-1}q^{-1})} = q\frac{T_{\Psib}^-(zq)(1 - z)^2}{T_{\Psib}^-(zq^3)(1 - zq^2)^2}.$$
These operators are rational. The action of $Y^+(z)/Y_{\Psib}^+(z)$ and $Y^-(z)/Y_{\Psib}^-(z)$ coincide. On a weight space 
of weight $\overline{\omega}^2\overline{\alpha}^{-h}$, the action of the following does not depend on $z$ :  
$$z^{-h} \frac{T^+(z)T^-(z)}{T_{\Psib}^+(z)T_{\Psib}^-(z)} = \begin{pmatrix}1 & 0 & 0 & 0 
\\0 & - q^{-1} & 1& 0
\\0 & 0 & - q^{-1} &0
\\0 & 0 & 0 & q^{-2}
\end{pmatrix}.$$
We note that 
$(T^\pm(z)/T_{\Psib}^\mp(z^{-1}q^2))^{\pm 1}$ and $(T_{\Psib}^\pm(z)/T^\pm(z))^{\pm 1}$ are polynomials in $z^{\mp 1}$.
\end{example}

\begin{example} Consider the prefundamental representation $W = L(\Psib_{1}^{-1})$ of $\mathcal{U}_q^{-\omega_1^\vee}(\hat{sl}_2)$ as in Example \ref{exval}.
Then we have for $j\geq 0$ : 
$$\frac{Y^\pm(z)}{Y^\pm_{\Psib}(z)} .v_j = \frac{q^{-j}(1 - zq)}{1 - q^{1 - 2j}z}.v_j \text{ , }\frac{T^\pm(z)}{T^\pm_{\Psib}(z)}.v_j = (1 - z^{\mp 1})^{\mp 1}(1 - z^{\mp 1}q^{\mp 2})^{\mp 1}\cdots (1 - z^{\mp 1}q^{\mp 2(j-1)})^{\mp 1}.v_j,$$
$$\text{for }Y^\pm_{\Psib} = (i)^{\delta_{\pm,-}}\text{exp}\left(- \sum_{m > 0} \frac{z^{\pm m}  }{m(q^m + q^{-m})}\right)\text{ , }T^\pm_{\Psib}(z) = \text{exp}\left(  \sum_{m> 0} \frac{z^{\mp m}}{m(q^{2m} - q^{-2m})}  \right).$$
We have $(T^+(z)T^-(z)/(T^+_{\Psib}(z)T^-_{\Psib}(z))).v_j = (-z)^jq^{j(j-1)}v_j$.
\end{example}

The following was partly established in \cite[Theorem 5.17]{FH} for simple finite-dimensional representations of $\mathcal{U}_q(\hat{\Glie})$. Let $\omega'$ be a weight of $W$. For $i\in I$, we denote by $ht_i(\omega (\omega')^{-1})$  
the multiplicity of $\overline{\alpha_i}$ 
in the factorization of $\omega (\omega')^{-1}$ as a product of simple roots.

\begin{thm}\label{newpol}  (i) The operators 
$$ \frac{Y_i^-(z)}{Y_{i,\Psib}^-(z)} \text{ and }\frac{Y_i^+(z)}{Y_{i,\Psib}^+(z)}  \in (\text{End}(W))(z) $$
are rational of degree $0$ on $W$ and coincide.

(ii) On $W_{\omega'}$ the operators 
$$\frac{T_i^-(z)}{T_{i,\Psib}^-(z)} \text{ and }z^{ht_i(\omega(\omega')^{-1})}\frac{T_{i,\Psib}^+(z)}{T_i^+(z)}  \in (\text{End}(W_{\omega'}))[z]$$
are polynomial in $z$ of degree $ht_i(\omega(\omega')^{-1})$ and coincide up to a constant operator factor.
\end{thm}

\begin{rem} This constant operator is not necessarily diagonalizable, see Example \ref{longex}.\end{rem}

\begin{proof} For $m\in\mathbb{Z}\setminus\{0\}$, $r\in\mathbb{Z}$, $\epsilon = 1$ or $\epsilon = -1$,  we have 
$$[\tilde{h}_{i,m}, x_{j,r}^\pm] = \delta_{i,j} \frac{[mr_i]_q}{m} x_{j,m + r}^\pm\text{ and }\overline{\phi}_i^\pm x_{j,r}^\epsilon = q_i^{\pm\delta_{i,j}\epsilon }x_{j,r}^\epsilon \overline{\phi}_i^\pm.$$ 
In particular $[x^\epsilon_j(w),Y_i^\pm(z)] = 0$ for $j\neq i$ and $\epsilon = 1$ or $- 1$.
For $i = j$, the relation (\ref{hd}) is 
$$[(q_i - q_i^{-1})h_{i,m},x_{i,r}^{\pm}] = \pm \frac{q_i^{2 m} - q_i^{-2m}}{m}  x_{i,m+r}^{\pm}.$$ 
It is the same relation as  
$$[(q - q^{-1})\tilde{h}_{i,m}, x_{i,r}^\pm] = \frac{q_i^m - q_i^{-m}}{m} x_{i,m + r}^\pm,$$ 
except that in the 
right side we have $q_i$ replaced by $q_i^2$, the term $(q_i - q_i^{-1})h_{i,m}$ being replaced by $(q - q^{-1})\tilde{h}_{i,m}$ as 
in the definition of $Y_i^\pm(z)$ in comparison to the definition of $\phi_i^\pm(z)$ (we have the same substitution for the analog 
of the relations (\ref{deux})). Hence we get as for the relation (\ref{phix}) : 
$$Y_i^{\epsilon}(z) x_j^\pm(w) =   \left(\frac{q_i^{\pm 1}w - z}{w - q_i^{\pm 1}z}\right)^{\delta_{i,j}}  x_j^\pm(w) Y_i^\epsilon(z) \text{ for $\epsilon = +$ or $-$},$$
First assume that $T_i^\pm(z)/T_{i,\Psib}^\pm(z)$ has a rational action on $W$. Then by (\ref{yt}) $Y_i^\pm(z) /Y_{i,\Psib}^\pm(z)$ 
has a rational action on $W$ of degree $0$. By the remarks above, the rational operator
$$\frac{Y_i^+(z)Y_{i,\Psib}^-(z)}{Y_i^-(z)Y_{i,\Psib}^+(z)}$$
commutes with all operators in the image of $\mathcal{U}_q(\hat{\mathfrak{g}})$ in the endomorphism ring of $W$. Hence, by Schur Lemma, 
$Y_i^+(z)/Y_{i,\Psib}^+(z)$ and $Y_i^-(z)/Y_{i,\Psib}^-(z)$, as these operators coincide on an highest weight vector, they 
coincide as rational operators on $W$. Hence (i) is proved. 
Besides, consider the rational operator on $W_{\omega'}$ : 
$$U(z) = z^{-ht_i(\omega(\omega')^{-1})}\frac{T_i^-(z)T_i^+(z)}{T_{i,\Psib}^-(z)T_{i,\Psib}^+(z)}$$
Then by (\ref{yt}) and (i), we get that $U(z) = U(zq_i^2)$  and so $U(z)$ does not depend on $z$, that is $\frac{T_i^-(z)}{T_{i,\Psib}^-(z)}$, $z^{ht_i(\omega(\omega')^{-1})}\frac{T_{i,\Psib}^+(z)}{T_i^+(z)}$ 
coincide up to a constant operator factor.

Now we establish (ii) in the Theorem (it does not follow directly from Proposition \ref{egw2} as the operator do not
have necessarily a diagonal action). 
The polynomiality for $T_i^-(z)/T_{i,\Psib}^-(z)$ is known for finite-dimensional simple $\mathcal{U}_q(\hat{\Glie})$-modules by \cite[Theorem 5.17]{FH}. But, using the involution $\sigma$ of $\mathcal{U}_q^{ad}(\hat{\Glie})$ as in Remark \ref{invsy}, 
we get that $T_i^+(z)$ is rational on a finite-dimensional $W$ up to a scalar map, and so that $T_i^+(z)/T_{i,\Psib}^+(z)$ is rational. 
By the discussion in the first part of this proof, this implies the polynomiality result for $T_i^+(z)$ in this case.

Now let $W$ be a tensor product of various negative prefundamental representations. By Corollary \ref{still}, it is simple as a $\mathcal{U}_q(\hat{\bo})$-module 
isomorphic to $L^\bo(\Psib)$. This representation can be constructed in \cite{HJ, HL} as an inductive limit of a linear inductive system 
of simple tensor products of Kirillov-Reshetikhin modules which are simple finite-dimensional representations of $\mathcal{U}_q(\hat{\Glie})$. 
In this inductive system, the highest weight vectors are preserved and the action of $\phi_i^+(z)$ is stationary up to 
a scalar function factor. Hence the polynomiality result follows for $T_i^+(z)$ for $W$ from the result for the finite-dimensional 
$\mathcal{U}_q(\hat{\bo})$-modules. But not only the inductive construction gives the action of $\mathcal{U}_q(\hat{\bo})$, 
but also of the whole asymptotic algebra $\tilde{\mathcal{U}}_q(\hat{\Glie})$ from which the action of $T_i^-(z)$ on $W$ is 
obtained. As above, it is stationary up to a scalar function factor. The polynomiality of $T_i^-(z)/T_{i,\Psib}^-(z)$ on $W$, and the result, follow.

The result is also clear for a tensor product of various positive prefundamental representations as they are one-dimensional. Now, as 
$$\Delta_u(T_i^\pm(z)) = T_i^\pm(z) \otimes T_i^\pm(zu^{-1}),$$
it follows from Corollary \ref{role} and (i) in Remark \ref{caspe} that the result holds true for a tensor product of negative prefundamental representations
by a tensor product of positive prefundamental representations. The result follows.
\end{proof}

\section{Truncated shifted quantum affine algebras}\label{tsqaa}

Truncations of shifted quantum affine algebras are defined in \cite[Section 8.(iii)]{FT} in the study of 
quantized $K$-theoretic Coulomb branches of $3d$ $N = 4$ SUSY quiver gauge theories (see the Introduction).

We recall the definition of truncated shifted quantum affine algebras in terms of series $A_i^{\mathcal{Z},\pm}(z)$ 
of Cartan-Drinfeld generators. We explain how these series appear naturally in terms of the Cartan-Drinfeld series 
derived from transfer-matrices in the previous section.

We establish (Proposition \ref{rata}) a necessary and sufficient condition for the defining series $A_i^{\mathcal{Z},\pm}(z)$ 
to have a rational action on a simple representation.

\subsection{Truncation series}

We consider a variation of the series $Y_i^\pm(z)$. We fix 
$$\lambda = \sum_{i\in I} N_i \omega_i^\vee\in \Lambda^+\text{ and }\lambda\succeq \mu = \lambda - \sum_{i\in I} a_i \alpha_i^\vee \in \Lambda,$$ 
with the $N_i, a_i\in\mathbb{Z}$ non-negative. We consider a family of polynomials $Z_i$ of degree $N_i$ :
$$Z_i(z) = (1 - q_i z z_{i,1})(1 - q_i z z_{i,2}) \cdots (1 - q_i z z_{i,N_i}),$$
and we get an $\ell$-weight ${\bf Z}=(Z_i(z))_{i\in I}$. We also fix additional parameters $z_i'\in\mathbb{C}^*$ so that 
$$\prod_{j\in I}(z_j')^{C_{j,i}} = (- q_i)^{N_i}z_{i,1}\cdots z_{i,N_i}.$$
These are unique up to the group $K$ of Remark \ref{groupk}. The collection of these data is denoted
$$\mathcal{Z} = ({\bf Z}, z_1',\cdots, z_n').$$
Then we define 
\begin{equation}\label{atil}A_i^{\mathcal{Z},\pm}(z) =  \sum_{r\geq 0}  A^{\mathcal{Z},\pm}_{i, \pm r} z^{\pm r} 
=  (z_i')^{\delta_{\pm,-}} \frac{\overline{Y}_{i,{\bf Z}}^\pm(zq_i^{-1})}
{Y_i^\pm(zq_i^{-1})}\in \mathcal{U}_{q}^{\mu,ad}(\hat{\Glie})[[z^{\pm 1}]].\end{equation}
In particular, $\overline{Y}_{i,{\bf Z}\Psib^{-1}}(z q_i^{-1})$ is the eigenvalue of $A_i^{\mathcal{Z},+}(z)$ on a 
highest weight vector of a simple representation $L(\Psib)$ in $\mathcal{O}_\mu$.

Note that by definition we have
\begin{equation}\label{conta}A_{i,0}^{\mathcal{Z},+} = (\overline{\phi}_i^+)^{-1}\text{ , }A_{i,0}^{\mathcal{Z},-} = z_i' (\overline{\phi}_i^-)^{-1}.\end{equation}
From (\ref{ya}), we recover the defining formula in \cite{FT}, that for $i\in I$ : 
$$z^{\alpha_i(\lambda - \mu)\delta{\pm, -}}\phi_i^\pm (z)(Z_i(z))^{-1} = (H_i(A_1^{\mathcal{Z},\pm}(zq_1), \cdots, A_n^{\mathcal{Z},\pm}(zq_n)))^{-1}.$$

\begin{rem}
(i) The series $A_i^{\mathcal{Z},\pm}(z)$ are uniquely characterized by this property and by (\ref{conta}), see Remark \ref{unisol}. 

(ii)  The notations could be misleading as the series 
$A_i^{-,\mathcal{Z}}(z)$ in \cite{FT} are variations of the $Y_i(z)$ in \cite{Fre}, not of the $A_i(z)$ therein.

(iii) The subalgebra generated by the Yangian counterpart of the $A_{i,\pm r}^{\mathcal{Z},\pm}$ is called the Gelfand-Tsetlin subalgebra in \cite{brkl}. 
It equals the Cartan-Drinfeld subalgebra generated by the $\phi_{i,\pm m}^\pm$ and the $A_{i,0}^{\mathcal{Z},\pm 1}$.
\end{rem}

\begin{example}\label{b2} Assume that $\Glie$ is of type $B_2$ with $r_1 = 2$ and $r_2 = 1$. The formula give
$$\phi_1^+ (z)(Z_1(z))^{-1} = \frac{A_2^{\mathcal{Z},+}(z)A_2^{\mathcal{Z},+}(zq^2)}{A_1^{\mathcal{Z},+}(z)A_1^{\mathcal{Z},+}(zq^4)}\text{ and } \phi_2^+ (z)(Z_2(z))^{-1} = \frac{A_1^{\mathcal{Z},+}(zq^2)}{A_2^{\mathcal{Z},+}(z)A_2^{\mathcal{Z},+}(zq^2)}.$$
\end{example}

\subsection{Definition}

\begin{defi} The truncated shifted quantum affine algebra $\mathcal{U}_{q,\lambda}^{\mu,\mathcal{Z}}(\hat{\Glie})$ is the quotient of $\mathcal{U}_{q}^{\mu,ad}(\hat{\Glie})$ by the 
relations that for $i\in I$, $A_i^{\mathcal{Z},\pm}(z)$ is a polynomial of degree $a_i$ in $z^{\pm 1}$ and : 
$$A_{i,0}^{\mathcal{Z},+}A_{i,0}^{\mathcal{Z},-} = (-q_i)^{a_i} 
\text{ , } A_i^{\mathcal{Z},+}(z) = (zq_i^{-1})^{a_i} A_i^{\mathcal{Z},-}(z)\text{ for $i\in I$.}$$
\end{defi}

\begin{rem}\label{todeal} (i) The relations imply $\phi_{i,0}^+ \phi_{i,\alpha_i(\mu)}^- = \phi_{i,\mathcal{Z}}$ for 
$$\phi_{i,\mathcal{Z}} = (-1)^{N_i + \sum_j C_{j,i}a_j} q_i^{\alpha_i(\mu)}  z_{i,1}z_{i,2}\cdots z_{i,N_i}.$$

(ii) We do not write the relations \cite[(8.11)]{FT} which are redundant in our notations.

(iii) The relations are not preserved by a twist of the spectral parameter $z\mapsto az$.
\end{rem}

\begin{example} $\mathcal{U}_{q,\omega_1^\vee}^{- \omega_1^\vee,\mathcal{Z}}(\hat{sl}_2)$ is the quotient of $\mathcal{U}_{q}^{- \omega_1^\vee,ad}(\hat{sl}_2)$ by the relations :
$$A_{\pm s}^{\mathcal{Z},\pm}  = 0 \text{ for $s > 1$, }A_0^{\mathcal{Z},+}A_1^{\mathcal{Z},+} = -1\text{ , }A_0^{\mathcal{Z},-} = A_1^{\mathcal{Z},+} q\text{ , }A_{-1}^{\mathcal{Z},-} = A_0^{\mathcal{Z},+} q.$$
\end{example}

In the following, when $\mathcal{Z}$ is fixed without ambiguity, we will simply denote $A^{\pm}(z)$ and $\mathcal{U}_{q,\lambda}^\mu(\hat{\Glie})$.
 The defining relations of $\mathcal{U}_{q,\lambda}^{\mu}(\hat{\Glie})$ can be interpreted in the following way.

\begin{prop}\label{debprop} For each $i\in I$, the images of $\phi_i^\pm (z)$ in $\mathcal{U}_{q,\lambda}^{\mu}(\hat{\Glie})[[z^{\pm 1}]]$ 
are rational of degree $\alpha_i(\mu)$ and coincide in $\mathcal{U}_{q,\lambda}^{\mu}(\hat{\Glie})(z)$. They satisfy $\phi_i^+(0) (\phi_i^+(z)z^{-\alpha_i(\mu)})(\infty) =  \phi_{i,\mathcal{Z}}$.
\end{prop}

\begin{proof} The rationality is clear as the $A_j^{\pm}(z)$ are polynomials. Then we get in $\mathcal{U}_{q,\lambda}^{\mu}(\hat{\Glie})(z)$ :
$$\frac{\phi_i^+(z)}{Z_i(z)} =  \frac{\prod_{j,C_{j,i} = -1}z^{a_j}\prod_{j,C_{j,i} = -2}(zq^{-1}zq)^{a_j}\prod_{j,C_{j,i} = -1}(zq^{-2}zzq^2)^{a_j} }{(zq_i^{-1}zq_i)^{a_i}} \frac{z^{\alpha_i(\lambda - \mu)} \phi_i^-(z)}{Z_i(z)}$$
$$= \prod_{j\in I}z^{-a_jC_{j,i}} \frac{z^{\alpha_i(\lambda - \mu)} \phi_i^-(z)}{Z_i(z)} = z^{\sum_{j\in I}C_{j,i}\omega_j(\mu - \lambda)} 
\frac{z^{\alpha_i(\lambda - \mu)} \phi_i^-(z)}{Z_i(z)} = \frac{\phi_i^-(z)}{Z_i(z)}.$$
This implies the equality as rational fractions. The degree is
$$N_i - \sum_{j\in I} C_{j,i} a_j = \alpha_i(\lambda) + \sum_{j\in I}C_{j,i}\omega_j (\mu - \lambda)  = \alpha_i(\mu).$$
\end{proof}

\subsection{Rationality of truncation series}

Let $W = L(\Psib)$ be a simple module in $\mathcal{O}_\mu$.

\begin{prop}\label{rata} (1) The following are equivalent : 

(i) for any $i\in I$, $A_i^{\mathcal{Z},+}(z)$ is rational on $W$.

(ii) for any $i\in I$, $A_i^{\mathcal{Z},-}(z)$ is rational on $W$.

(iii) $\Psib(0){\bf Z}\Psib^{-1}$
 is a Laurent monomial in the $\Lambda_{i,a}$, $i\in I$, $a\in\mathbb{C}^*$.

(2) When these conditions are satisfied, $A_i^{\mathcal{Z},+}(z)$ and $(zq_i^{-1})^{a_i} A_i^{\mathcal{Z},-}(z)$ coincide on
$W$ as rational fractions (up to an element of the group $K$ of Remark \ref{groupk}) and have degree $a_i$. 
\end{prop}

\begin{proof} (1) As above, it does not follow directly from Proposition \ref{egw2} as the operator do not
have necessarily a diagonal action. However, from Theorem \ref{newpol}, $A_i^{\mathcal{Z},\pm}(z)$ is rational on $W$ if and only 
the eigenvalue on a highest weight vector is rational. From Lemma \ref{egw}, this is equivalent to (i) 
or (ii) or (iii).

(2) When the conditions are satisfied, it follows from Theorem \ref{newpol} that $Y_i^\pm(z)/Y_{i,\Psib}^\pm(z)$ 
coincide as rational fractions, so it suffices to prove that $(z_i' z^{a_i})^{\delta_{\pm,-}} \overline{Y}_{i,{\bf Z}}^\pm(z) /Y_{i,\Psib}^\pm(z)$ coincide as rational fractions. From Lemma \ref{egw}, they coincide up to a constant $c_i$ : 
$$\overline{Y}_{i,{\bf Z}}^+(z) /Y_{i,\Psib}^+(z) 
= c_i z_i' z^{a_i} \overline{Y}_{i,{\bf Z}}^-(z) /Y_{i,\Psib}^-(z).$$
As 
$$H_i(\overline{Y}_{1,{\bf Z}}^-(z),\cdots, \overline{Y}_{n,{\bf Z}}^-(z)) = 
(-q_iz)^{-N_i}  (z_{i,1}\cdots z_{i,N_i})^{-1}Z_i(z),$$
we get
$$Z_i(z) \Psi_i^{-1}(z) 
= H_i(c_1,\cdots, c_n) z^{\alpha_i(- \mu)} Z_i(z) 
/ (z^{-\alpha_i(\mu)}\Psi_i(z)),$$
and so $(c_1,\cdots, c_n)\in K$.
\end{proof}

\begin{example} We continue Example \ref{longex}. Let us consider the polynomial operators : 
$$A^-(zq) = \frac{T^-(z^{-1}q)}{T_{\Psib}^-(z^{-1}q)}\frac{T_*^-(z^{-1}q^{-1})}{T^-(z^{-1}q^{-1})} = \frac{q^3(1 - q^{-2}z^{-1})^2Y_{\Psib}^-(z)}{Y^-(z)}$$
$$=\begin{pmatrix}q^3(1 - q^{-2}z^{-1})^2 & 0 & 0 & 0\\ 0 &  (q^2 - z^{-1})(1 - z^{-1}) & 
 z^{-1} (q^3 - q) & 0 
\\  0 & 0 & (q^2 - z^{-1})(1 - z^{-1})  & 0 \\ 0 & 0 & 0 & q(1 - z^{-1})^2 
\end{pmatrix}.$$ 
$$A^+(zq) = \frac{T^+(z^{-1}q^{-1})}{T_*^+(z^{-1}q^{-1})}\frac{T_{\Psib}^+(z^{-1}q)}{T^+(z^{-1}q)} = \frac{q^{-1}(1 - z q^2)^2Y_{\Psib}^+(z)}{Y^+(z)} $$
$$= \begin{pmatrix}q^{-1}(1 - z q^2)^2 & 0 & 0 & 0\\ 0 &  (1 - zq^2)(1 - z) & z  (q^3 - q) & 0 
\\  0 & 0 & (1 - zq^2)(1 - z) & 0 \\ 0 & 0 & 0 & q (1 - z)^2 \end{pmatrix}$$ 
We note that $A^\pm(z) = A^{\mathcal{Z},\pm}(z)$ with $\mathcal{Z} = ((1 - zq^{-1})^2(1 - zq^3)^2, q^2)$ as we have 
$$A^+(z) = (zq^{-1})^2 A^-(z)\text{ , }A_+(0)A_-(\infty) = q^2\text{Id},$$
$$(A^+(z)A^+(zq^2))^{-1} = \frac{\phi^+(z)}{(1 - zq^{-1})^2(1 - zq^3)^2}\text{ and }(A^-(z)A^-(zq^2))^{-1} = \frac{z^2 \phi^-(z)}{(1 - zq^{-1})^2(1 - zq^3)^2}.$$
\end{example}

\begin{rem}\label{rcront} We see from the proof of Lemma \ref{egw} that the contribution of the factor $\Lambda_{i,a}$ to the eigenvalue of $A_i^+(z)$ is $(1 - zaq_i^{-1})^{-1}$.
\end{rem}

\section{Descent to the truncation}\label{dttt}

We study which simple representations descend to truncated shifted quantum affine algebras using methods which are also new in the case of ordinary quantum affine algebras
or shifted Yangians. See the introduction for a
discussion on earlier results \cite{brkl, ktwwy, ktwwy2, N4, nw}.

We establish a necessary condition on a simple representation to be a representation of 
a truncated shifted quantum affine algebra (Proposition \ref{maint}). As a consequence we establish that a 
truncated shifted quantum affine algebra has only a finite number of isomorphism classes of simple representations (Theorem \ref{finsimp}).
Then we introduce a partial ordering $\preceq_{\mathcal{Z}}$ on $\ell$-weights (up to sign) and we prove
that the simple representations $L(\Psib)$ of a truncated shifted quantum affine algebra of parameter ${\bf Z}$ must satisfy $\Psib\preceq_{\mathcal{Z}}{\bf Z}$ (Theorem \ref{thpartial}).
We note this partial ordering is also related to the Langlands dual Lie algebra ${}^L\Glie$, a point which will be crucial in 
the next Section. In the $sl_2$-case we establish a complete characterization 
of simple representations of a truncated shifted quantum affine algebra (Theorem \ref{carsl22}).

\subsection{Descent} 

\begin{defi} A representation in $\mathcal{O}_\mu$ descends to the truncation 
$\mathcal{U}_{q,\lambda}^{\mu,\mathcal{Z}}$ if it has a structure of $\mathcal{U}_q^{\mu,ad}(\hat{\Glie})$-module 
compatible with the defining relations of the quotient $\mathcal{U}_{q,\lambda}^{\mu,\mathcal{Z}}(\hat{\Glie})$. 
\end{defi}

This defines an abelian subcategory $\mathcal{O}_{\mu,\mathcal{Z}}^{\lambda}$ of $\mathcal{O}_\mu$.

\begin{rem} The category $\mathcal{O}_{\mu,\mathcal{Z}}^\lambda$ is stable by sign-twist.
\end{rem}

We investigate which simple modules $L(\Psib)$ are in $\mathcal{O}_{\mu,\mathcal{Z}}^\lambda$. 
It means that there is a structure $L(\Psib,\overline{\omega}^+,\overline{\omega}^-)$ of $\mathcal{U}_q^{\mu,ad}(\hat{\Glie})$-module 
on $L(\Psib)$ which is a $\mathcal{U}_{q,\lambda}^{\mu,\mathcal{Z}}(\hat{\Glie})$-module. 
If $\overline{\omega}^\pm$ exist, then 
\begin{equation}\label{normc}(\overline{\omega}^+\overline{\omega}^-)(i) = z_i' (-q_i)^{-a_i}\text{ for any $i\in I$.}\end{equation}
By Proposition \ref{debprop}, such $\overline{\omega}^\pm$ compatible with $\Psib\in\mathfrak{r}_\mu$ exist if and only if for any $i\in I$, $\Psi_i(0) (\Psi_i(z)z^{-\alpha_i(\mu)})(\infty) =  \phi_{i,\mathcal{Z}}$. We will denote by $\mathfrak{r}_{\mu,\mathcal{Z}}$ the set of such $\Psib$. We will work with such $\ell$-weights (if necessary, although not written explicitly, we 
will renormalize $\ell$-weights by constants to work in this set). So let us consider $\Psib\in\mathfrak{r}_{\mu,\mathcal{Z}}$ 
and $\overline{\omega}^\pm\in\tb^*$ compatible with $\Psib$ satisfying the relations (\ref{normc}). 
Then the central element $\overline{\phi}_{i}^+\overline{\phi}_i^-$ (resp. $A_{i,0}^{\mathcal{Z},+}A_{i,0}^{\mathcal{Z},-}$) acts as the scalar $z_i' (-q_i)^{-a_i}$ 
(resp. $(-q_i)^{a_i}$) on $L(\Psib)$.

\begin{example}\label{fex} Suppose $\lambda = \mu$. Then for $V$ in $\mathcal{O}_{\mu,\mathcal{Z}}^\mu$, 
the operator $A_{i,0}^{\mathcal{Z},+} = A_i^{\mathcal{Z},+}(z) = A_i^{\mathcal{Z},-}(z) = A_{i,0}^{\mathcal{Z},-}$ is constant 
and satisfies $\text{Id} = A_{i,0}^{\mathcal{Z},+}A_{i,0}^{\mathcal{Z},-} = (A_{i,0}^{\mathcal{Z},+})^2$ and $\phi_i^+(z)(Z_i(z))^{-1} = \prod_{j\in I} (A_{j,0}^{\mathcal{Z},+})^{-C_{j,i}}$. 
Hence $\mathcal{O}_{\mu,\mathcal{Z}}^\mu$ is semi-simple, its simple objects are $1$-dimensional of 
highest $\ell$-weight $\Psi(z) = (\eta_i Z_i(z))_{i\in I}$ with $\eta_i = \prod_{j}\epsilon_j^{C_{j,i}}$
for a choice of $\epsilon_j = \pm 1$ (see Example \ref{pospre}). Up to a sign-twist, 
there is a unique simple representation in $\mathcal{O}_{\mu,\mathcal{Z}}^\mu$.
\end{example}

\begin{example}\label{sex}
Let $\Glie = sl_2$, $\lambda = 2 \omega_1^\vee$, $\mu = 0$. Set $\mathcal{Z} = ((1 - aq^3z)(1 - aq^{-1}z), aq)$. 
For $\Psib(z) = \frac{aq(1 - aq^{-1}z)}{(1 - aqz)}$, $L(\Psib)$ is a $2$-dimensional 
representation in $\mathcal{O}_{0,\mathcal{Z}}^{2\omega_1^\vee}$. Indeed its other $\ell$-weight is
$\Psib'(z) = \frac{aq^{-1}(1 - aq^3 z)}{(1 - aqz)}$. Choose $\alpha$ so that $\alpha^2aq = 1$.
For $A(z) = (\alpha - \alpha^{-1} z)$ and $A'(z) = (q\alpha - \alpha^{-1}q^{-1}z)$, one has
$$\Psi(z) = \frac{Z(z)}{A(z)A(zq^2)}\text{ , }\Psi'(z) = \frac{Z(z)}{A'(z)A'(zq^2)}.$$
\end{example}

\begin{rem}\label{flem}  Let $i\in I$ and $V$ be a representation in $\mathcal{O}_\mu$ so that the central element 
$\overline{\phi}_{i}^+\overline{\phi}_i^-$ acts by $z_i' (-q_i)^{-a_i}$. 
As a direct consequence of Proposition \ref{rata}, if $A_i^{\mathcal{Z},+}(z)$ has
a rational action on $V$, then $A_i^+(z) = (z q_i^{-1})^{a_i} A_i^-(z)$ if and only if we have
$$A_i^+(z)\sim_{\infty} (-z)^{a_i}(A_i^+(0))^{-1}$$ 
on $V$. Besides it suffices that this condition is satisfied on a highest weight vector.\end{rem}

\subsection{Partial ordering}\label{paror}

We have the following refinement of Proposition \ref{rata}.

\begin{lem}\label{porder} If for any $i\in I$, $A_i^{\mathcal{Z},+}(z)$ is polynomial on $L(\Psib)$, then $\Psib(0){\bf Z}\Psib^{-1}$ is a monomial in the $\Lambda_{i,a}$, $i\in I$, $a\in\mathbb{C}^*$.
\end{lem}

In fact, our proof implies it suffices it is polynomial on a highest weight vector of $L(\Psib)$.

\begin{proof} By Proposition \ref{rata}, $\Psib(0){\bf Z}\Psib^{-1}$ is a Laurent monomial in the $\Lambda_{i,a}$, $i\in I$, $a\in\mathbb{C}^*$. Following the proof of Lemma \ref{egw}, we see that the powers of the $\Lambda_{i,a}$ 
have to be non-negative so that the eigenvalue of $A_i^{\mathcal{Z},+}(z)$ on a highest weight vector of $L(\Psib)$ is a polynomial.
\end{proof}

This suggests the following definition for the set of $\ell$-weights :
$$\mathfrak{r}_{\mathcal{Z}} = \bigsqcup_{\mu\in\Lambda}\mathfrak{r}_{\mu,\mathcal{Z}}.$$
For $\Psib,\Psib'\in\mathfrak{r}_{\mathcal{Z}}$, we set $\Psib'\preceq_{\mathcal{Z}} \Psib$ 
if $(\Psib(0))^{-1}(\Psib'(0)) \Psib (\Psib')^{-1}$ is a monomial in the $\Lambda_{i,a}$, $i\in I$, $a\in\mathbb{C}^*$.

\begin{prop} If $\Psib'\preceq_{\mathcal{Z}} \Psib$, then $\Psib$ is determined by $\Psib'$ and $(\Psib(0))^{-1}(\Psib'(0)) \Psib (\Psib')^{-1}$ up to a sign $i\in I$.

In particular, $\preceq_{\mathcal{Z}}$ defines a partial ordering on $\mathfrak{r}_{\mathcal{Z}}$ (up to signs).
\end{prop}

\begin{proof} Let $\Psib\in\mathfrak{r}_{\mu,\mathcal{Z}}$, $\Psib'\in\mathfrak{r}_{\mu',\mathcal{Z}}$. 
It suffices to prove that each $\Psib_i(0)$ is determined up to a sign. 
The conditions imply that for any $i\in I$, $(\Psib_i(0))^{-1}(\Psib_i(z)z^{-\alpha_i(\mu)})(\infty)$ is 
determined. But $\Psib_i(0)(\Psib_i(z)z^{-\alpha_i(\mu)})(\infty)$ is fixed, so $(\Psib_i(0))^2$ is determined.

For the second point, it suffices to consider $\Psib$, $\Psib'$ so that $\Psib\preceq_{\mathcal{Z}} \Psib'$ and $\Psib'\preceq_{\mathcal{Z}} \Psib$. 
Then, $(\Psib(0))^{-1}(\Psib'(0)) \Psib (\Psib')^{-1} = 1$ and from the first point we get that $\Psib$ and $\Psib'$
are equal (up to a sign for each $i\in I$).
\end{proof}

\begin{rem} (i) If we add the data of a $\overline{\omega}^\pm\in\tb^*$ compatible with each $\Psib$ as above, then 
as for Lemma \ref{egw} we can replace "up to a sign" by "up to a sign twist".

(ii) This partial ordering is different from the extension of Nakajima partial ordering $\preceq$ in Section \ref{npo}. 
\end{rem}

Now Lemma \ref{porder} can be reformulated in terms of this partial ordering.

\begin{thm}\label{thpartial} For $L(\Psib)$ a simple representation in $\mathcal{O}^\lambda_{\mu,\mathcal{Z}}$ we have
$$\Psib  \preceq_{\mathcal{Z}} {\bf Z}.$$
\end{thm}

\begin{rem}\label{mustaslo} For $L(\Psib)$ a representation in $\mathcal{O}^\lambda_{\mu,\mathcal{Z}}$, all its $\ell$-weights $\Psib'$ 
must also satisfy $\Psib'\preceq_{\mathcal{Z}} {\bf Z}$ (by the proof of Lemma \ref{egw}). 
As an example, more concretely, consider an $\ell$-weight of the form $\Psib' = \Psib A_{i,a}^{-1}$. Then Formula (\ref{alambda}) gives that a factor $\Lambda_{i,aq_i}$ in 
the factorization of ${\bf Z}\Psib^{-1}$ is replaced by $\Lambda_{i,aq_i^{-1}}$ 
in the factorization of ${\bf Z}(\Psib')^{-1}$. 
\end{rem}

\subsection{A necessary condition on highest $\ell$-weight and finiteness}

\begin{prop}\label{maint} Suppose that $L(\Psib)$ is in $\mathcal{O}_{\mu,\mathcal{Z}}^\lambda$. 
Then for $i\in I$, $\overline{Y}^+_{i,{\bf Z}\Psib^{-1}}(z q_i^{-1})$ and $\overline{Y}^+_{i,{\bf Z}\Psib^{-1}}(z q_i^{-1}) \Psi_i(z)$ are polynomials with 
$$\overline{Y}^+_{i,{\bf Z}\Psib^{-1}}(z q_i^{-1})\sim_{\infty} (-z)^{a_i}(\overline{Y}^+_{i,{\bf Z}\Psib^{-1}}(0))^{-1}.$$
\end{prop}

\begin{rem} (i) The property does not depend on the choice of $\overline{\omega}^\pm$ but only on $\Psib$.

(ii) The degree of the two polynomials are $a_i$ and $a_i + \alpha_i(\mu)$ respectively.
\end{rem}

We will use the currents defined for $i\in I$ by
$$x_i^{+,+}(z) =  \sum_{r\geq 0} x_{i, r}^+ z^{ r}\text{ , }x_i^{-,+}(z) = - \sum_{r > 0} x_{i, r}^- z^{ r}\text{ , }
C_i (z) = (q_i - q_i^{-1}) x_i^{-,+}(z) A_i^+(z) = \sum_{r > 0} C_{i,r} z^r.$$
The following relations were proved in \cite{FT} when $\mu\in- \Lambda^+$. 
But since the commutators $[x_{i,r}^\pm, \phi_{j,s}^+]$ are the same for a general $\mu\in \Lambda$, it holds in general : 
 \begin{equation}\label{ac}  (w - z)[A_i(z), C_i(w)]_{q_i} = (q_i - q_i^{-1})  (z C_i(w) A_i(z)  - w C_i(z) A_i(w) ). \end{equation}
Here we use that standard notation $[a,b]_{q_i^{\pm 1}} = ab - q_i^{\pm 1}ba$. 

\begin{proof} Let $i\in I$ and $v$ be a highest weight vector of $L(\Psib)$. The polynomiality of $\overline{Y}^+_{i,{\bf Z}\Psib^{-1}}(z)$, 
which is the eigenvalue of $A_i^{\mathcal{Z},+}(z)$ on $v$, is clear. 
It is already observed in \cite{FT} that the polynomiality of $C_i(z)$ can be deduced in the following way : 
the coefficient of $w$ in relation (\ref{ac}) gives 
$$-z [A_i^+(z), C_{i,1}]_{q_i^{-1}} = -(q_i - q_i^{-1})  C_i(z) A_{i,0}^+ .$$ 
As $A_{i,0}^+$ is invertible, $C_i(z)$ is a polynomial on $L(\Psib)$. Now we have also : 
$$x_{i,0}^+.C_i(z).v = (q_i - q_i^{-1}) [x_{i,0}^+,x_i^{-,+}(z)]A_i^+(z).v = -
\overline{Y}^+_{i,{\bf Z}\Psib^{-1}}(z q_i^{-1})(\sum_{r > 0}(\phi_{i,r}^+ - \phi_{i,r}^-)z^r) .v$$
$$=  - \overline{Y}^+_{i,{\bf Z}\Psib^{-1}}(z q_i^{-1})
(\Psi_i(z) - \Psi_{i,0}^+ - \sum_{0 < r \leq \alpha_i(\mu)}\Psi_{i,r}^-z^r) .v,$$
As $\overline{Y}^+_{i,{\bf Z}\Psib^{-1}}(z q_i^{-1})$ and $x_{i,0}^+.C_i(z).v$ are polynomial in $z$, this implies that 
$\overline{Y}^+_{i,{\bf Z}\Psib^{-1}}(z q_i^{-1}) \Psi_i(z)$ is a polynomial. For the second point, the eigenvalue of $A_{i,a_i}^+$ on $v$ is $(-1)^{a_i}(
\overline{Y}^+_{i,{\bf Z}\Psib^{-1}}(0))^{-1}$.
\end{proof}

This condition is not sufficient in general. 

\begin{example} Let $\Glie = sl_3$, $\lambda = \omega_1^\vee$, $\mu = \lambda - 2 \alpha_1^\vee - \alpha_2^\vee = -2\omega_1^\vee$,
$Z_1(z) = 1 - z$, $Z_2(z) = 1$,
$$\Psib = \left(\frac{q^{-4}}{(1 - zq^{-4})(1 - zq^{-2})} , 1\right).$$
We fix $\mathcal{Z}$ with compatible $z_1', z_2'$. We have 
$$ \overline{Y}^+_{1,{\bf Z}\Psib^{-1}}(z q^{-1})
= \alpha^2(1 - z q^{-4})(1 - z q^{-2})\text{ and }
\overline{Y}^+_{2,{\bf Z}\Psib^{-1}}(z q^{-1}) = \alpha(1 - z q^{-3})$$ 
with $\alpha^3 = q^4$. 
The condition of Proposition \ref{maint} are satisfied. However, $L(\Psib)$ is not in $\mathcal{O}^\lambda_{\mu,\mathcal{Z}}$. Indeed, by Theorem \ref{fpm}, this representation is the fusion product of two negative prefundamental representations (with a $1$-dimensional constant representation). The $q$-character of each of these factors is known as it is the same as for the corresponding negative prefundamental representation of $\mathcal{U}_q(\hat{\bo})$ (in the $sl_3$-case these representations are explicitly described in \cite{HJ}). The following is an $\ell$-weight of $L(\Psib)$ : 
$$\Psib' = \Psib A_{1,q^{-2}}^{-1}A_{2,q^{-1}}^{-1}= \left(\frac{q^{-5}}{(1 - zq^{-4})^2},\frac{q^{-1}(1 - zq)}{1 - zq^{-1}}\right).$$
and is not of the correct form
$$ \left(\frac{(1-z)A_2(zq)}{A_1(z)A_1(zq^2)},\frac{A_1(zq)}{A_2(z)A_2(zq^2)}\right)$$
for $A_1(z)$, $A_2(z)$ polynomials of respective degrees $2$, $1$. Indeed, identifying the second coordinates, we would have 
$A_1(z)$ of the form $(1 - z)(1-za)$ up to a constant and for a certain $a\in\mathbb{C}^*$. 
This contradicts the relation for the first coordinates, as $q^4$ could not be a pole of order $1$.
\end{example}

However in some cases the conditions are enough to determine the simple representations in $\mathcal{O}^\lambda_{\mu,\mathcal{Z}}$. This is the case for the $\Glie = sl_2$ in 
the next section. We have also the following.

\begin{example}\label{almost} Let $\lambda = \omega_i^\vee$ and $\mu = \omega_i^\vee - \alpha_i^\vee$. Set $Z_j(z) = 1 - zaq_i^2\delta_{i,j}$ and associate $z_j'$ accordingly. 
Up to sign twist $\mathcal{O}^\lambda_{\mu,\mathcal{Z}}$ contains a unique simple representation whose highest $\ell$-weight is $\tilde{\Psib}_{i,a}$ (up to a constant, see Example \ref{exqchar}). The uniqueness follows from Proposition \ref{maint}. From Example \ref{exqchar}, the $\ell$-weight spaces are of dimension $1$ parametrized by $m\geq 0$ with the corresponding eigenvalue of 
$A_j^{\mathcal{Z},+}(z)$ equal to $1$ if $j\neq i$ and if $j = i$ equal to $v_i^{-1}q_i^m - v_iq_i^{- m}z$ where $v_i^2 = a_i$ are fixed.
\end{example}

We have the following consequence of Proposition \ref{maint}.

\begin{thm}\label{finsimp} There is a finite number of simple representations $L(\Psib)$ in $\mathcal{O}^\lambda_{\mu,\mathcal{Z}}$.
\end{thm}

\begin{proof} Let $L(\Psib)$ be in $\mathcal{O}^\lambda_{\mu,\mathcal{Z}}$. As $\Psib\preceq_{\mathcal{Z}}{\bf Z}$ by Theorem \ref{thpartial}, 
there are $v_{i,a}\geq 0$ so that 
$$\Psib (\Psib(0))^{-1} = {\bf Z}\prod_{i\in I, a\in\mathbb{C}^*}\Lambda_{i,a}^{-v_{i,a}}.$$
Moreover, $\Psib$ is determined by the $v_{i,a}$ up to sign twist. So it suffices to show that there is a finite
number of possibilities for the $v_{i,a}$. Besides, $\sum_{a\in\mathbb{C}^*}v_{i,a} = a_i$. So it suffices
to prove that there is a finite number of possible $a$ so that $v_{i,a}\neq 0$.

For $i\in I$, $a\in\mathbb{C}^*$, let $z_{i,a}$ be the multiplicity of $a^{-1}$ as a root of $Z_i(z)$.

For each $i\in I$, let 
$$\tilde{Z}_i(z) = \prod_{a\in\mathbb{C}^*}(1 - za)^{v_{i,a}}$$ 
which is equal to $\overline{Y}^+_{i,{\bf Z}\Psib^{-1}}(zq_i^{-1})$ up to a constant. Then from Proposition \ref{maint}, $\tilde{Z}_i(zq_i^2)$ divides 
$$Z_i(z) \left(\prod_{C_{j,i} = -1}\tilde{Z}_j(zq_i)\right)
\left(\prod_{C_{j,i} = -2}\tilde{Z}_j(z)\tilde{Z}_j(zq^2)\right)
\left(\prod_{C_{j,i} = -3}\tilde{Z}_j(z)\tilde{Z}_j(zq^2)\tilde{Z}_j(zq^4)\right).$$

Suppose that $v_{i,a}\neq 0$. Then $z_{i,aq_i^{2}}\neq 0$ or 
there is $j\neq i$ so that $v_{j,b}\neq 0$ with $b = a q^r$ for $r = r_i$ if $C_{j,i} =  -1$, 
$r = 4$ or $2$ if $C_{j,i} = -2$, $r = 6$ or $4$ or $2$ if $C_{j,i} = -3$.

Consequently : $v_{i,a}\neq 0$ implies that there is $R\geq 0$ and a finite sequence 
$$(i_0,a_0) = (i,a), (i_1,a_1) = (i_1,aq^{r_1}), \cdots , (i_R,a_R) = (i_R,aq^{r_R})$$ 
so that for any $k$, $r_k < r_{k+1} \leq r_k + 2r_{i_k} $ and $v_{i_k,a_k} > 0$, and 
$$ z_{i_R,aq^{r_R}q_{i_R}} > 0.$$ 
This implies that $0\leq r \leq 6 \sum_{l\in I} a_l$. So there is a finite number of possible $a$ and the result follows.
\end{proof}

\begin{rem}\label{finol} (i) It follows from the proof that $\Psib$ is the product of ${\bf Z}$ by various 
$\Lambda_{i,a}^{-1}$ so that there are $j\in I$, $r > 0$ with $z_{j,aq^r}\neq 0$.

(ii) The same proof as above, taking into account not only the fact that the
$v_{j,b}\neq 0$, but also there actual value, implies that the category 
$$\mathcal{O}^\lambda_{\mathcal{Z}} = \bigoplus_{\mu\in\Lambda} \mathcal{O}^\lambda_{\mu,\mathcal{Z}}$$
has a finite number of simple objects. Indeed, there is a finite subset $\mathcal{A}\subset \mathbb{C}^*$ of $a\in\mathbb{C}^*$ so that
one $v_{i,a}$ might be non zero. One obtains by induction on $|\{r\geq 0,aq^r\in\mathcal{A}\}|\geq 0$ that the 
possible values of the $v_{i,a}$ are bounded.
\end{rem}

\subsection{Descent for $\Glie = sl_2$}

We suppose in this section $\Glie = sl_2$. 

The condition of Proposition \ref{maint} is 
\begin{equation}\label{sl2ec}({\bf Z}(zq^{-2}))^{-1}  \preceq_{\mathcal{Z}} \Psib (z) \preceq_{\mathcal{Z}} {\bf Z}(z).\end{equation}
Indeed $\overline{Y}^+_{{\bf Z}\Psib^{-1}}(zq^{-1})$ polynomial means $\Psib (z) \preceq_{\mathcal{Z}} {\bf Z}(z)$.
Let $\tilde{Z}(z)$ be this polynomial. Then $\Psi(z) \tilde{Z}(z)= Z(z)/\tilde{Z}(zq^{2}) = Q(z)$ polynomial means
$\tilde{Z}(zq^{2})$ divides $Z(z)$. But 
$$\Psi(z) Z(zq^{-2}) = \frac{Z(z)Z(zq^{-2})}{\tilde{Z}(zq^2)\tilde{Z}(z)} = Q(z)Q(zq^{-2})$$
and so $\Psib(z) {\bf Z}(zq^{-2})\succeq_{\mathcal{Z}} 1$.

We shall prove now that the condition (\ref{sl2ec}) is sufficient. Let us denote 
$$\Phi(z) = \sum_{r\geq 0}\Phi_r z^r = \sum_{r\geq 0}(\phi_r^+ - \phi_r^-)z^r  =  \phi^+(z) - \sum_{0\leq r\leq \alpha_1(\mu)}   \phi_r^- z^r.$$
Note that if $\mu \in -\Lambda^+\setminus\{0\}$ is strictly antidominant, $\Phi(z) = \phi^+(z)$. For $m,m' \geq 0$, we have 
$(q - q^{-1})[x^+_m,x^-_{m'}] = \Phi_{m+m'}$, and so
\begin{equation}\label{plusmoins}(z - w)[x^{+,+}(z), x^{-,+}(w)] = w \frac{\Phi(w) - \Phi(z) }{ q - q^{-1}}.\end{equation}
This generalizes formulas established in \cite{FT} for $\mu\in -\Lambda^+$.
The following relation is established in \cite{FT} (for $\mu\in - \Lambda^+$ but the same
proof gives the result in general) : 
\begin{equation}\label{cci}[C(z),C(w)] = 0.\end{equation}

We prove the converse of the statement in Proposition \ref{maint} is true in the $sl_2$-case.

\begin{thm}\label{carsl22} $L(\Psib)$ in $\mathcal{O}_\mu$ descends to 
the truncation $\mathcal{U}_{q,\lambda}^{\mu,\mathcal{Z}}(\hat{sl_2})$ if and only if $\overline{Y}^+_{{\bf Z}\Psib^{-1}}(z q^{-1})$ and 
$\overline{Y}^+_{{\bf Z}\Psib^{-1}}(z q^{-1}) \Psi(z)$ are polynomials with 
$\overline{Y}^+_{{\bf Z}\Psib^{-1}}(z q^{-1})\sim_{\infty} (-z)^{a}(\overline{Y}^+_{{\bf Z}\Psib^{-1}}(0))^{-1}$.
\end{thm}

\begin{proof} One implication is proved in Proposition \ref{maint}. Let us suppose that the conditions are satisfied.
We prove that the action of $A^+(z)$, and also of $C(z)$, are polynomial 
on $V = L(\Psib)$. Let $\omega_0$ be the highest weight of $V$. For $N\geq 0$, let $V_N$ be the sum of weight spaces $V_\omega$ with $\omega - \omega_0 = N \alpha$. 

First we prove that $C(z)$ is polynomial on $V_0$. Let $v\in V_0$. We prove that for $m\geq a + |\alpha_1(\mu)|$, the coefficient of $w^{m + 1}$ in $C(w).v$ is zero. 
As $V$ is simple, it suffices to prove this for $x^{+}(z)C(w).v$.
Due to Proposition \ref{zero}, it suffices to prove this is true for 
$$x^{+,+}(z)C(w).v = (q - q^{-1})[x^{+,+}(z),x^{-,+}(w)] A^+(w).v.$$
Relation (\ref{plusmoins}) gives
$$(z - w)x^{+,+}(z)C(w).v = w (\Phi(w) - \Phi(z)) A^+(w).v.$$
For $l\geq 0$, considering the coefficients of $w^{l+1}$, we get 
$$-x^{+,+}(z)C_{l}.v + z x^{+,+}(z)C_{l+1}.v = (\Phi A)_l.v - \Phi(z)A_l.v,$$
where $(\Phi A)_l$ is the coefficient of $z^l$ in $\Phi(z)A^+(z)$. Let us multiply this relation by $z^l$ and take the sum of the relations for $l = 0, \cdots , m$. As $C_{0} = 0$, $A^+(z).v$ is a polynomial of degree 
$a\leq m$ and $\Phi(z)A^+(z).v$ a polynomial of degree lower than $m$ :
$$ x^{+,+}(z)z^{m+1}C_{m+1}.v = (\sum_{0\leq l\leq m}(\Phi A)_lz^l - \Phi(z) A_{l}z^l).v = (\Phi(z)A^+(z) - \Phi(z)A^+(z)).v = 0.$$
This proves the claim, that is $C_{m+1}.v = 0$.

Now by the relation (\ref{cci}), we get by induction on $N$ that $C(z)$ is a polynomial on $V_N$ of degree lower than $a + |\alpha_1(\mu)|$.
Then, using relation (\ref{ac}), we obtain that the action of $A^+(z)$ on $V$ is polynomial. We conclude by Remark \ref{flem}.
\end{proof}

\begin{rem} As a by product one gets that $B (z) = (q - q^{-1}) A^+(z) x^{+,+}(z)$ and 
$D (z) = A^+ (z) \phi^+(z) + (q - q^{-1})^2 x^{-,+} (z) A(z) x^{+,+}(z)$ are polynomial on this representation.
\end{rem}

\begin{example} Let $\lambda = \omega_1^\vee$, $\mu = - \omega_1^\vee = \lambda - \alpha_1^\vee$. 
We fix $z_1, z_1'\in\mathbb{C}^*$. From Example \ref{almost}, up to sign twist there is a unique simple representation $L(\Psib)$ 
in $\mathcal{O}^{\omega_1^\vee}_{-\omega_1^\vee,\mathcal{Z}}$ with
$$\Psi(z) =   \frac{q^{-1} z_1 }{1 - z q^{-1}z_1}.$$
This is Example \ref{follow} with $a = b = z_1 q^{-1}$. Indeed for $L(b(1 - az)^{-1})$ we have (here $\alpha^2 = a$) : 
$$A^+(z)^.v_j = (\alpha^{-1} q^{j} - \alpha q^{- j} z)v_j \text{ , }\phi_+(z).v_j = \frac{a(1 - q^2 az)q^{-2j}}{(1 - q^{2 - 2j}az)(1 - q^{-2j} az)}v_j,$$
$$B(z).v_j = \alpha^{-1} (q - q^{-1})  q^{ j -1} v_{j-1}\text{ , }C(z).v_j = [j+1]_q  q^{-2 j} \alpha^3   z  v_{j+1}\text{ , }D(z).v_j 
=
\alpha q^{-j}v_j.$$
The representation $L(b(1 - az)^{-1})$ descends to a truncation if and only if $a = b$. 
\end{example}

\begin{example} Let $\lambda = 2 \omega_1^\vee$, $\mu = - 2 \omega_1^\vee = \lambda  - 2\alpha_1^\vee$. 
We fix $z_1, z_2 \in\mathbb{C}^*$. Let $\Psib$ be of degree $\alpha(\mu) = -2$ and $\Psi(0)(\Psi(z)z^2)(\infty) = \phi_{\Psib} = q^{-2}z_1z_2$. We have $\overline{Y}^+_{{\bf Z}\Psib^{-1}}(z)$ of degree $2$ and
\begin{equation}\label{form2}\Psi(z) = \frac{(1  - qzz_1)(1 - qz z_2)}{\overline{Y}^+_{{\bf Z}\Psib^{-1}}(z q^{-1})\overline{Y}^+_{{\bf Z}\Psib^{-1}}(z q)}.\end{equation}
The conditions of Proposition \ref{maint} 
give that the roots of $\overline{Y}^+_{{\bf Z}\Psib^{-1}}(z )$ are $z_1^{-1}$, $z_2^{-1}$ and
$$\Psi(z) =   \frac{\Psi(0)}{(1 - z q^{-1}z_1)(1 - zq^{-1}z_2)} =  \frac{q^{-2} z_1z_2 }{(1 - z q^{-1}z_1)(1 - z q^{-1}z_2)}.$$
There is one simple representation in $\mathcal{O}_{-2\omega_1^\vee,\mathcal{Z}}^{2\omega_1^\vee}$. It is a fusion product of two negative prefundamental representations.
\end{example}

\begin{example} Let $\lambda = 2 \omega_1^\vee$, $\mu = 0 = \lambda  - \alpha_1^\vee$. 
We fix $z_1, z_2 \in\mathbb{C}^*$. Let $\Psib$ be an $\ell$-weight of degree $\alpha(\mu) = 0$ and 
$\Psi(0)(\Psi(z)z^2)(\infty) = z_1z_2$. We have (\ref{form2}) but with $\overline{Y}^+_{{\bf Z}\Psib^{-1}}(z)$ of degree $1$. The conditions of the Proposition \ref{maint} give that 
$\overline{Y}^+_{{\bf Z}\Psib^{-1}}(z q^{-1}) = \overline{Y}^+_{{\bf Z}\Psib^{-1}}(0) -  (\overline{Y}^+_{{\bf Z}\Psib^{-1}}(0))^{-1}z$. 
Hence the root of $\overline{Y}^+_{{\bf Z}\Psib^{-1}}(z)$ is $z_1^{-1}$ or $z_2^{-1}$, say it is $z_i^{-1}$, and let $z_j^{-1}$ be the other root. Then
$$\Psi(z) = \frac{\Psi(0)(1 - qz_jz)}{(1 - q^{-1}z_iz)} = \frac{q^{-1}z_i(1 - qz_jz)}{(1 - q^{-1}z_iz)}.$$
Hence if $z_1\neq z_2$, there are two simple representations 
in $\mathcal{O}_{0,\mathcal{Z}}^{2\omega_1^\vee}$. If $z_1 = z_2$, there is one simple representation in $\mathcal{O}_{0,\mathcal{Z}}^{2\omega_1^\vee}$, its $\ell$-weight is $q^{-1}z_1$ and it is of dimension $1$.
\end{example}

\begin{cor}\label{carsl2} If $\mathcal{O}^\lambda_{\mu,\mathcal{Z}}$ is non empty then $-\lambda\preceq \mu\preceq \lambda$. Its simple representations are in bijection with the divisors $\tilde{Z}(z)$ of $Z(zq^2)$ of degree $\omega_1(\lambda - \mu)$ satisfying $\tilde{Z}(0) = 1$. The corresponding $\ell$-weight satisfies
$$\Psi(z) = \Psi(0)\frac{Z(z)}{\tilde{Z}(zq^{-2})\tilde{Z}(z)}.
$$
\end{cor}

\begin{proof} The conditions of the Theorem \ref{carsl22} imply that $\overline{Y}^+_{{\bf Z}\Psib^{-1}}(z q)$ divides $Z(z)$. This determines $\Psi(z)$ up to the constant $\Psi(0)$.
Conversely, for each divisor $\tilde{Z}(z)$ of $Z(zq^2)$ normalized with $\tilde{Z}(0) = 1$, we can fix $\overline{Y}^+_{{\bf Z}\Psib^{-1}}(z q^{-1}) = \alpha \tilde{Z}(z)$ for a certain $\alpha\in\mathbb{C}^*$ determined up to a sign by the limit condition.
\end{proof}

\subsection{An example for $\Glie = sl_3$}\label{exsl3}

Let $\Glie = sl_3$, $\lambda = \omega_1^\vee$ with $Z_1(z) = 1 - zq^3$ and $Z_2(z) = 1$. 

$\mu = \lambda = \omega_1^\vee$. By Example \ref{fex}, we get $L(1 - zq^3,1)$ is the unique simple in $\mathcal{O}^\lambda_{\mu,\mathcal{Z}}$.

$\mu = \lambda - \alpha_1^\vee = \omega_2^\vee - \omega_1^\vee$. By Example \ref{almost}, we get $L(\frac{q}{1 - zq},v^{-1}(1 - zq^2))$ 
is the unique simple in $\mathcal{O}_{\mu,\Psib}^\lambda$ up to sign twist (here $v$ is a square root of $q$).

$\mu = -\omega_2^\vee = \lambda - \alpha_1^\vee - \alpha_2^\vee$. By Proposition \ref{maint}, there is at most one simple
module $L(q,\frac{v^{-1}}{1 - z} )$ in $\mathcal{O}_{\mu,\mathcal{Z}}^\lambda$ up to sign twist. 
The action of $\mathcal{U}_q(\hat{\bo})$ is described in \cite[Section 4.1]{HJ}. The $\ell$-weight spaces are of dimension $1$ with $\ell$-weights 
parametrized by $0\leq n'\leq n$ :
$$\left( q^{1 + n - 2n'} \frac{(1 - q^3z)(1 - q^{1-2n}z)}{(1 - q^{1 - 2n'}z) (1 - q^{3 - 2n'}z)} , v^{-1}q^{n' - 2n}\frac{1 - q^{2 - 2n'} z }{(1 - q^{-2n}z)(1 - q^{-2n + 2}z)}   \right),$$
with $v^{-1}q^{n'} - vq^{- n'}z$, $q^n - q^{- n}z$ respective eigenvalues of $A_1^+(z), A_2^+(z)$.

\section{A conjecture : truncation and Langlands dual standard modules}\label{maincsec}

We state a conjecture (Conjecture \ref{mainc}) on the parametrization of simple modules of non-simply laced truncated shifted quantum affine algebras. The statement of the conjecture involves the Langlands dual Lie algebra ${}^L\Glie$ : it is given
in terms of the structure of a standard module of the twisted quantum affine algebra $\mathcal{U}_q({}^L\hat{\Glie})$, more precisely 
in terms of its Langlands dual $q$-character that we introduce.

For simply-laced types, simple representations of truncated shifted Yangians have been parametrized in terms of Nakajima monomial crystals \cite{ktwwy2}. Combining with \cite{N4}, this implies an analogous statement for simply-laced shifted quantum affine algebras. 
This is a fundamental motivation for our conjecture in non simply-laced types. See the introduction and Remark \ref{remconj} for a discussion on earlier results. 

We have several strong evidences for our conjecture. We establish in type $B_2$ that our parametrization gives representations of the 
truncated shifted quantum affine algebra (Proposition \ref{tdeux}). In general, we establish that a simple finite-dimensional
representation of a shifted quantum affine algebra descends to a truncation as in Conjecture \ref{mainc} (Theorem \ref{truncfd}). 
The proof of this last result is based on Baxter polynomiality of quantum integrable models.

For non simply-laced types, the conjecture, these results and these methods are also new in the case of ordinary quantum affine algebras
or shifted Yangians (see the discussion in the Introduction and Remark \ref{remconj}). 

As in the previous section, we have fixed $\lambda \in\Lambda^+$ and corresponding ${\bf Z}$, $\mathcal{Z}$.

\subsection{Reminder - interpolating $(q,t)$-characters}

Interpolating $(q,t)$-characters were introduced in \cite{FH0} as an 
incarnation of Frenkel-Reshetikhin deformed $\mathcal{W}$-algebras \cite{Wd} to 
interpolate between $q$-characters of a non simply-laced quantum affine algebra and its Langlands dual. 
 These interpolating $(q,t)$-character are tools in \cite{FH0} to study Langlands duality between finite-dimensional representations 
of quantum affine algebras.

Let $r = \text{Max}_{i\in I}(r_i)$ be the lacing number of $\mathfrak{g}$. We set $\epsilon = e^{i\pi/r}$.

For $i\in I$, $a\in\mathbb{C}^*$, we set
\begin{align}\label{zy} Z_{i,a} = \begin{cases} Y_{i,a}&\text{ if $r_i = r$,}\\ Y_{i,aq^{-1}}Y_{i,aq}&\text{ if $r_i =r - 1$,} \\Y_{i,aq^{-2}}Y_{i,a}Y_{i,aq^2}&\text{ if $r_i = r - 2$.} \end{cases}\end{align}
Then $W = L(Z_{i,a})$ is a Kirillov-Reshetikhin module of the untwisted quantum affine algebra $\mathcal{U}_q(\hat{\Glie})$ 
(it is a fundamental representation when $r_i = r$). 

Let us recall that the interpolating $(q,t)$-character $\chi_{q,t}(W)$ of $W$ (more precisely its refined version constructed in \cite{FHR}) is an element of a (completion of a) quotient of the ring 
$$ \mathbb{Z}[Y_{j,b}^{\pm 1}, \alpha]_{j\in I, b\in a q^{\mathbb{Z}}t^{\mathbb{Z}}\epsilon^{\mathbb{Z}}}$$ 
where $t$ is an additional formal variable and $\alpha$ is an indeterminate (in type $G_2$, this indeterminate is denoted by $\beta$ in \cite{FH0}). 

%The quotient is defined in the following way.  It is obtained from the ideal generated for $j,j'\in I$, $b,b'\in az^{\mathbb{Z}}t^{\mathbb{Z}}$ by the :
%$$\alpha(\alpha - 1)\text{ , }\alpha(Y_{j,b} - Y_{j,bt})\text{ , }(\alpha - 1)(Z_{j,bq} - Z_{j,b\epsilon})\text{ , }
%(Z_{j,b} - Z_{j,bt})(Z_{j',b'q} - Z_{j',b'\epsilon}).$$

The interpolating $(q,t)$-characters have various interesting limits, for instance the following limits discussed in \cite{FH0, FHR} : 

When $t \rightarrow 1$, $\alpha$ specializes to $1$ and $\chi_{q,t}(L(Z_{i,a}))$ specializes to the $q$-character of the $\mathcal{U}_q(\hat{\Glie})$-module $L(Z_{i,a})$. 

When $q \rightarrow \epsilon$, $\alpha$ specializes to $0$ and $\chi_{q,t}(L(Z_{i,a}))$ specializes to the $t$-character of the fundamental representation $V_i^L(a)$ of highest monomial $Z_{i,a}$ of the Langlands dual twisted quantum affine algebra $\mathcal{U}_t({}^L\hat{\Glie})$ (in the sense of \cite{H3}).

\subsection{Langlands dual $q$-characters}

For $i\in I$, $a\in\CC^*$, consider the interpolating $q,t$-character $\chi_{q,t}(L(Z_{i,a}))$. 

To state our conjecture, let us consider another specialization : we set $t = 1$ but we discard the monomials 
with coefficient $\alpha$, that is we set $\alpha = 0$. By \cite{FHR}, this limit is well-defined, 
and only variables $Z_{j,aq^m}^{\pm 1}$ with $j\in I$, $m\in\mathbb{Z}$, occur. We get a well-defined element $\chi_q^L(V_i^L(a))$ 
in (a completion of) the ring 
$$\mathbb{Z}[Z_{j,aq^m}^{\pm 1}]_{j\in I, m\in \mathbb{Z}},$$
that we call the Langlands dual $q$-character of $V_i^L(a)$. 

The Langlands dual $q$-character in defined from a combination of the limit $t = 1$, which gives monomials occurring 
in the $q$-character of the $\mathcal{U}_q(\hat{\Glie})$-module $L(Z_{i,a})$ (which is not a fundamental representation in general, but a Kirillov-Reshetikhin module), 
and of the limit $\alpha = 0$, which give the coefficients of the $t$-character of a fundamental $\mathcal{U}_t({}^L\hat{\Glie})$-module. 

Now let $V^L$ be the standard module of $\mathcal{U}_q({}^L\hat{\Glie})$ of highest monomial  
$$M_0 = \prod_{i\in I}Z_{i,q_i^{-1}z_{i,1}^{-1}}\cdots Z_{i,q_i^{-1}z_{i,N_i}^{-1}}.$$
It is defined as a tensor product (for the ordinary coproduct) of the fundamental representations $V_i^L(q_i^{-1}z_{i,j})$ 
(as for standard modules considered in simply-laced cases in \cite{N, VV}). Its isomorphism class depends on 
the ordering of the tensor product, but not its class in the Grothendieck ring. 

We introduce its Langlands dual $q$-character        
\begin{equation}\label{ldqc}\chi_q^L(V^L) = \prod_{i\in I, 1\leq s \leq N_i} \chi_q^L(V_i^L(q_i^{-1}z_{i,s}^{-1}))
%\in\mathbb{Z}[Z_{j,b}^{\pm 1}]_{j\in I, b\in\CC^*}
.\end{equation}
It should not be confused with its $q$-character $\chi_q(V^L)$ as a representation of $\mathcal{U}_q({}^L\hat{\Glie})$.

We also the use a representation $V$ of $\mathcal{U}_q(\hat{\Glie})$ which is Langlands dual to $V$ in the sense of 
\cite{FH0}, that is an (ordered) tensor product of the simple $\mathcal{U}_q(\hat{\Glie})$ representations of highest
monomial $Z_{i,q_i^{-1}z_{i,j}^{-1}}$, expressed in terms of the $Y_{j,b}$ variables as in (\ref{zy}) (these are not fundamental representations in general).

\subsection{Statement}

For a Laurent monomial $M = \prod_{i\in I,a\in\mathbb{C}^*}Z_{i,a}^{u_{i,a}}$, we define the corresponding $\ell$-weight 
$\Psib_M = (\Psi_i(z))_{i\in I}$ by 
\begin{equation}\label{corrl}\Psi_i(z) = \Psi_i(0)\prod_{a\in\mathbb{C}^*}(1 - za^{-1})^{u_{i,a}},\end{equation}
with 
$$(\Psi_i(0))^2 = (\prod_{a\in\CC^*}(- a)^{u_{i,a}})\phi_{i,\mathcal{Z}}$$ 
(this $\ell$-weight is defined up to sign-twist). The corresponding weight is $\mu_M = \sum_{i\in I, a\in\CC^*}u_{i,a}\omega_i^\vee \in \Lambda$.
As an example, we have $\Psib_{M_0} = {\bf Z}$ up to a constant.

Recall the partial ordering $\preceq_{\mathcal{Z}}$ in Section \ref{paror}.

\begin{prop}\label{compord}  For a monomial $M$ occurring in the Langlands dual $q$-character of $V^L$, we have
$$\Psib_M \preceq_{\mathcal{Z}} {\bf Z}.$$
\end{prop}

\begin{proof} By definition, $M$ is the limit at $t = 1$ of a monomial $\tilde{M}$ in the interpolating $(q,t)$-character
of $V$, without a factor $\alpha$. By construction \cite[Section 4]{FH0}, $M_0$ is the product of $\tilde{M}$ by a product of various 

$\tilde{A}_{j,bt^s}$ with $j\in I$ so that $r_j = r$, 

$\tilde{A}_{j,bq^{-1}t^s}\tilde{A}_{j,bqt^s}$ with  $j\in I$ so that $r_j = 1$ and $r = 2$,

$\tilde{A}_{j,bq^{-2}t^s}\tilde{A}_{j,bt^s}\tilde{A}_{j,bq^2t^s}$ with $j\in I$ so that $r_j = 1$ and $r = 3$,

\noindent for certain $b\in q^{\mathbb{Z}}\{q_i^{-1}z_{i,p}^{-1}\}_{i\in I, 1\leq p\leq N_i}$, $s\in\mathbb{Z}$. The limits at $t = 1$ are respectively

$A_{j,b} = Z_{j,bq_j^{-1}}Z_{j,bq_j}\prod_{k,C_{j,k} = -1}Z_{k,b}^{-1},$

$A_{j,bq^{-1}}A_{j,bq} = Z_{j,bq^{-1}}Z_{j,bq}
\prod_{k,C_{j,k} = -1}Z_{k,b}^{-1}\prod_{k,C_{j,k} = -2}Z_{j,bq^{-1}}^{-1}Z_{j,bq}^{-1},$ 

$A_{j,bq^{-2}}A_{j,b}A_{j,bq^2} = Z_{j,bq^{-1}}Z_{j,bq}
\prod_{k,C_{j,k} = -1}Z_{k,b}^{-1}\prod_{k,C_{j,k} = -3}Z_{j,bq^{-2}}^{-1}Z_{j,b}Z_{j,bq^2}^{-1}.$ 

\noindent The corresponding $\ell$-weight, as defined by (\ref{corrl}), is $\Lambda_{j,b^{-1}}$, hence the result.
\end{proof}

Recall that by Theorem \ref{thpartial}, $L(\Psib)$ in $\mathcal{O}^\lambda_{\mu,\mathcal{Z}}$ must satisfy 
$\Psib  \preceq_{\mathcal{Z}} {\bf Z}$, the same condition as in Proposition \ref{compord}. 

We state our main Conjecture for $\Glie$ of non simply-laced type.

\begin{conj}\label{mainc} 

(A) For a monomial $M$ occurring in the Langlands dual $q$-character of $V^L$, 
the simple module $L(\Psib_M)$ is in the category $\mathcal{O}^\lambda_{\mu_M,\mathcal{Z}}$.

(B) For a simple module $L(\Psib)$ in a category $\mathcal{O}^\lambda_{\mu,\mathcal{Z}}$, there is a 
monomial $M$ occurring in $\chi_q(V)$ so that $\Psib = \Psib_M$ and $\mu = \mu_M$.

(C) This defines a bijection between the classes of simple representations in $\mathcal{O}^\lambda_{\mu,\mathcal{Z}}$ up to sign-twist and the monomials of weight $\mu$ occurring in the Langlands dual $q$-character of $V^L$. 
\end{conj}

\begin{rem}\label{remconj} (i) As explained above, results already obtained for simply-laced types are 
fundamental motivations for this conjecture. 
Simple representations of simply-laced truncated shifted Yangians have been parametrized in terms of Nakajima monomial crystals \cite{ktwwy2}. Combining with \cite{N4}, this implies an analogous statement for simply-laced shifted quantum affine algebras. 
In simply-laced type, $V^L$ is a standard module of $\mathcal{U}_q({}^L\hat{\Glie}) = \mathcal{U}_q(\hat{\Glie})$. 
Note that the set of monomials occurring in $\chi_q(V^L)$ is the product of the 
set of monomials occurring in the monomial crystals $\mathcal{M}(Y_{i, q^{-1}z_{i,s}^{-1}})$, see \cite[Section 7]{ktwwy}.

(ii) Conjecture \ref{mainc} 
does not involve the monomial crystal for non simply-laced types in \cite{K2} (see also \cite{hn}) or 
the $q$-character of the standard module of $\mathcal{U}_q(\hat{\Glie})$ (see also Remark \ref{contexx}).

(iii) According to our conjecture, the $\ell$-weights of simple representations in $\mathcal{O}_{\mu, \mathcal{Z}}^{\lambda}$ can be read 
from the monomials in the Langlands dual $q$-character. By interpolation, these monomials correspond to monomials in the $q$-character of the standard module of the twisted 
quantum affine algebra, but the $\ell$-weights can not be read directly from the latter in general. 
Also, by folding in \cite{H3}, these monomials correspond to monomials 
in the $q$-character of a standard module of a corresponding untwisted simply-laced quantum affine algebra.
Hence we get a relation between non simply-laced shifted quantum affine algebras and simply-laced quantum affine algebras
 (see also (iii) in Remark \ref{contexx}).

(iv) As discussed in the introduction, Nakajima-Weekes \cite{nw}, combining with \cite{ktwwy2}, gave 
an explicit parametrization of simple representations in category $\mathcal{O}$ of truncated non 
simply-laced shifted Yangians and quantum affine algebras. Using the previous point (iii), one can compare 
and consider a possible relation between the two parametrizations. 
In small examples this different method seems to give the same parametrization as our result.

(v) Conjecture \ref{mainc} implies that an arbitrary simple representation in $\mathcal{O}_\mu$ is in one of the categories 
$\mathcal{O}^\lambda_{\mu,\mathcal{Z}}$. Indeed, for any $i\in I$, $a\in\mathbb{C}^*$, the monomial $Z_{i,a}^{-1}$ occurs as
the lowest weight monomial of the $q$-character of the $\mathcal{U}_q(\hat{\Glie})$-module $L(Z_{\overline{i},aq^{-rh^{\vee}}})$, 
where $Z_{\overline{i},aq^{-rh^{\vee}}}$ is expressed in terms of the $Y_{j,b}$ variables as in (\ref{zy}) 
(this follows from \cite{Fre2}). Here $h^{\vee}$ is the dual Coxeter number of $\Glie$ and 
$\overline{i}\in I$ is set so that $w_0(\alpha_i) = -\alpha_{\overline{i}}$ for $w_0$ the longest element of the Weyl 
group of $\mathfrak{g}$. Let us assume Conjecture \ref{mainc} is true. Then the Langlands dual $q$-character of the representation $V_{\overline{i}}^L(aq^{-rh^{\vee}})$ has the lowest monomial $Z_{i,a}^{-1}$, and an arbitrary Laurent monomial 
in the $Z_{i,a}^{\pm 1}$ occurs in the Langlands dual $q$-character of a standard module of $\mathcal{U}_t({}^L\hat{\Glie})$. 
The statement follows.

(vi) Recall the category $\mathcal{O}_{\mathcal{Z}}^\lambda$ has a finite number of simple objects by Remark \ref{finol}. We may expect it is a categorification of (a natural subspace of) the module $V$, associated to the monomials of $\chi_q^L(V^L)$, in the spirit of \cite{ktwwy2}.

(vii) Should we extend the construction of this paper to twisted shifted quantum affine algebras, we expect the 
parametrization of simple representations would involve interpolating $(q,t)$-characters
of finite-dimensional representations of twisted quantum affine algebras as in \cite[Section 6]{FH0}, as well as Langlands dual
$q$-character of standard modules of untwisted non-simply laced quantum affine algebras.
\end{rem}

\subsection{Examples in simply-laced types} 

See the Introduction and Remark \ref{remconj} for general earlier results in simply-laced types 
(see also Corollary \ref{carsl2} for $\Glie = sl_2$).

For $\Glie = sl_3$, in the examples in section \ref{exsl3}, the $3$ simple representations in $\mathcal{O}_{\mathcal{Z}}^\lambda$ correspond to the $3$ monomials occurring in  
$$\chi_q(L(Y_{1,q^{-3}})) = Y_{1,q^{-3}} + Y_{2,q^{-2}}Y_{1,q^{-1}}^{-1} +  Y_{2,1}^{-1}.$$

\subsection{Examples in type $B_2$} 

We work in type $B_2$ with $r_1 = 2$ and $r_2 = 1$.

First let us set $\lambda = \omega_2^\vee$, $Z_1(z) = 1$, $Z_2(z) = 1 - z$. 
The interpolating $(q,t)$-character of the $\mathcal{U}_q(B_2^{(1)})$-representation $L(Z_{2,1})$ 
has $11$ terms and was computed explicitly in \cite[Section 3.5]{FH0} :
$$Y_{2,q^{-1}}Y_{2,q} +\alpha Y_{2,q^{-1}}Y_{2,q^3t^2}^{-1}Y_{1,q^2t} + Y_{2,qt^2}^{-1}Y_{2,q^3t^2}^{-1}Y_{1,t}Y_{1,q^2 t}+ \alpha  Y_{2,q^{-1}}Y_{2,q^5t^2}Y_{1,q^6t^3}^{-1}
+ Y_{1,q^2t}Y_{1,q^4t^3}^{-1}$$ 
$$+ Y_{2,qt^2}^{-1}Y_{2,q^5t^2}Y_{1,q^6t^3}^{-1}Y_{1,t}+\alpha Y_{2,q^{-1}}Y_{2,q^7t^4}^{-1}+ Y_{1,q^4t^3}^{-1}Y_{1,q^6t^3}^{-1}Y_{2,q^3t^2}Y_{2,q^5t^2}
+ \alpha Y_{2,qt^2}^{-1}Y_{2,q^7t^4}^{-1}Y_{1,t}$$
$$+\alpha Y_{1,q^4t^3}^{-1}Y_{2,q^3t^2}Y_{2,q^7t^4}^{-1}+Y_{2,q^5t^4}^{-1}Y_{2,q^7t^4}^{-1}.$$
We put $\alpha = 0$, $t = 1$, $Z_{2,q^r} = Y_{2,q^{-r-1}}Y_{2,q^{r+1}}$, $Z_{1,q^r} = Y_{1,q^{r}}$ and we get 
$$\chi_q^L(V_2^L(1)) = Z_{2,1} + Z_{2,q^2}^{-1}Z_{1,1}Z_{1,q^2} + Z_{1,1}Z_{1,q^6}^{-1} Z_{2,q^2}^{-1}Z_{2,q^4} 
+ Z_{1,q^2}Z_{1,q^4}^{-1} +  Z_{1,q^6}^{-1}Z_{1,q^4}^{-1}Z_{2,q^4} + Z_{2,q^6}^{-1}.$$
We have
$$\phi_1^+ (z)= \frac{A_2^+(z)A_2^+(zq^2)}{A_1^+(z)A_1^+(zq^4)}\text{ and } \phi_2^+ (z) = \frac{(1 - z)A_1^+(zq^2)}{A_2^+(z)A_2^+(zq^2)}.$$
For the following values of $\mu$, one gets a unique simple object in $\mathcal{O}^\lambda_{\mu,\mathcal{Z}}$ up to sign-twist :

$\mu = \omega_2^\vee$ : $\Psib = (1 , 1 - z)$.

$\mu = \lambda - \alpha_2^\vee = 2\omega_1^\vee - \omega_2^\vee$ : $\Psib = \left( q^2(1 - zq^{-2})(1 - z)   ,\frac{q^{-2}}{1 - zq^{-2}}\right)$.

$\mu = \lambda - \alpha_2^\vee - 2\alpha_1^\vee = \omega_2^\vee - 2\omega_1^\vee$ : $\Psib = \left(  \frac{q^{-8}}{(1 - zq^{-6})(1 - zq^{-4})}, q^3(1 - zq^{-4} )\right)$.  

$\mu = \lambda - 2\alpha_2^\vee - 2\alpha_1^\vee = -\omega_2^\vee$ : $\Psib = \left(q^{-2} ,\frac{q^{-3}}{1 - zq^{-6}} \right)$.

\noindent The first two cases follow respectively from respective Example \ref{fex} and Example \ref{almost}. 

\noindent The representation associated to $\mu = \omega_2^\vee - 2\omega_1^\vee$ is a subquotient of 
$$L\left(\frac{q^{-8}}{1 - zq^{-4}}, q^3(1 - zq^{-4} )\right)\otimes L_{1,q^{-6}}^-\text{ and }L\left(\frac{q^{-8}}{1 - zq^{-6}}, q^3(1 - zq^{-4} )\right)\otimes L_{1,q^{-4}}^-.$$ 
This implies $[\Psib^{-1}]\chi_q(L(\Psib))$ equals
$$\sum_{0\leq\alpha, \beta,0\leq \gamma\leq\text{Min}(2\alpha,1 + 2\beta) }
(A_{1,q^{-4}}^{-1}A_{1,q^{-8}}^{-1}\cdots A_{1,q^{-4\alpha}}^{-1})
(A_{1,q^{-6}}^{-1}A_{1,q^{-10}}^{-1}\cdots A_{1,q^{-2-4\beta}}^{-1})
(A_{2,q^{-2}}^{-1}\cdots A_{2,q^{-2\gamma}}^{-1}).$$
It is a thin representation, that is its $\ell$-weight spaces have dimension $1$. 
So to check this representation descends to truncation reduces to a direct combinatorial check on 
the explicit $q$-character formula, which can be done directly.

\noindent By \cite{H3tb}, Kirillov-Reshetikhin modules are thin in type $B$ 
and we have an explicit formula for its $q$-character. So this is also true for negative prefundamental representations which
are limits of Kirillov-Reshetikhin modules by \cite{HJ}. Hence we can conclude as for the previous representation.

\noindent Let us focus for the details on the most subtle weight $\mu = \lambda - \alpha_2^\vee - \alpha_1^\vee = 0$. 

\noindent Let $L(\Psib)$ be a simple representation in $\mathcal{O}^\lambda_{\mu,\mathcal{Z}}$. We have 
$$\overline{Y}^+_{1,{\bf Z}\Psib^{-1}}(z q^{-2}) = (\delta^{-1} - z\delta)
\text{ , }\overline{Y}^+_{2,{\bf Z}\Psib^{-1}}(z q^{-1}) = (\gamma^{-1} - z\gamma)$$ 
for some $\beta, \gamma\in\mathbb{C}^*$. As $\Psi_2(z)\overline{Y}^+_{2,{\bf Z}\Psib^{-1}}(z q^{-1})$ and $\Psi_1(z)\overline{Y}^+_{1,{\bf Z}\Psib^{-1}}(z q^{-2})$ are polynomials, 
we have $\overline{Y}^+_{2,{\bf Z}\Psib^{-1}}(z q^{-1}) = \pm(q - zq^{-1})$ and $
\overline{Y}^+_{1,{\bf Z}\Psib^{-1}}(z q^{-2}) = \pm (q^3 - z q^{-3})$ or $\pm (q^2 - zq^{-2})$. 
Hence, up to sign-twist, we get two possibilities for $\Psib(z)$ : 
$$\Psib_1 = \left(q^{-2}\frac{1 - z}{1 - zq^{-6}} , q \frac{1 - zq^{-4}}{1 - zq^{-2}}\right)\text{ or }\Psib_2 = \left(q^{-2}\frac{1 - zq^{-2}}{1 - zq^{-4}}, 1\right).$$
We can check directly these two simple representations descends indeed to the truncation. First
$$\chi_q(L(\Psib_1)) = \Psib_1 \sum_{\alpha\beta\geq 0,\beta\leq 2\alpha + 1} (A_{1,q^{-6}}A_{1,q^{-10}}\cdots A_{1,q^{-2-4\alpha}})^{-1}
(A_{2,q^{-2}}A_{2,q^{-4}}\cdots A_{2,q^{-2\beta}})^{-1},$$
with the action on the $1$-dimensional $\ell$-weight space corresponding to each term given by
$$A_1^+(z) = q^{3 + 2\alpha} - z q^{-3 - 2\alpha}\text{ , }A_2^+(z) = q^{\beta + 1} - z q^{-\beta - 1}.$$ 
The formula is compatible with Remark \ref{mustaslo} as ${\bf Z}\Psib_1^{-1} = \Lambda_{1,q^{-4}}\Lambda_{2,q^{-1}}$. Then
$$\chi_q(L(\Psib_2)) = \Psib_2 \sum_{\alpha\beta\geq 0,\beta\leq 2\alpha} (A_{1,q^{-4}}A_{1,q^{-8}}\cdots A_{1,q^{-4\alpha}})^{-1}
(A_{2,q^{-2}}A_{2,q^{-4}}\cdots A_{2,q^{-2\beta}})^{-1}$$
with the action on the $1$-dimensional $\ell$-weight space corresponding to each term given by
$$A_1^+(z) = q^{2 + 2\alpha} - z q^{-2 - 2\alpha}\text{ , }A_2^+(z) = q^{\beta + 1} - z q^{-\beta - 1}.$$ 
We get two non twist-equivalent simple representations in the category $\mathcal{O}^\lambda_{\mu,\mathcal{Z}}$ as predicted by 
Conjecture \ref{mainc}. This proves Conjecture \ref{mainc} for these weights in type $B_2$.

\begin{rem}\label{contexx} (i) There are $6$ terms in $\chi_q^L(V_2^L(1))$ which is the dimension of the fundamental representation $V_2^L(1)$ of $\mathcal{U}_q({}^L(B_2^{(1)})) = \mathcal{U}_q(A_3^{(2)})$, but the number of vertices in the corresponding crystal of finite type $B_2$ in only $5$. This example also explains the importance of involving the Langlands dual Lie algebra ${}^L\Glie$ in Conjecture \ref{mainc}.
Indeed, the study of these simple representations in $\mathcal{O}^\lambda_{\mu,\mathcal{Z}}$ would not be compatible with usual $q$-character 
of type $C_2^{(1)}$ : 
$$Y_{2,1} + Y_{2,q^4}^{-1}Y_{1,q}Y_{1,q^3} + Y_{1,q^5}^{-1}Y_{1,q} + Y_{1,q^3}^{-1}Y_{1,q^5}^{-1}Y_{2,q^2} + Y_{2,q^4}^{-1},$$
or with the monomial crystal :
$$\mathcal{M}(Y_{2,1}) = \{Y_{2,1}, Y_{2,q^2}^{-1}Y_{1,q}^2, Y_{1,q}Y_{1,q^3}^{-1}, Y_{1,q^3}^{-2}Y_{2,q^2}, Y_{2,q^4}^{-1}\}.$$

(ii) The $q$-character of $V_2^L(1)$ would not be directly relevant either : 
$$Z_{2,1} + Z_{2,q^4}^{-1}Z_{1,-q}Z_{1,q} + Z_{1,q}Z_{-q^3}^{-1} + Z_{1,-q}Z_{1,q^3}^{-1} + Z_{2,q^4}Z_{1,q^3}^{-1}Z_{1,-q^3}^{-1}    + Z_{2,q^8}^{-1}.$$

(iii) For $\mu = 0$, we have seen the two $\ell$-weights can be read in terms of the Langlands dual $q$-character, 
involving the monomials of the $q$-character a Kirillov-Reshetikhin modules of the untwisted quantum affine algebra 
$\mathcal{U}_q(B_2^{(1)})$ :
$$Z_{1,q^2}Z_{1,q^4}^{-1} = Y_{1,q^2}Y_{1,q^4}^{-1} = (Y_{2,q^{-1}}Y_{2,q}) A_{2,1}^{-1}A_{2,q^2}^{-1}A_{1,q^2}^{-1},$$ 
$$Z_{1,1}Z_{1,q^6}^{-1}Z_{2,q^2}^{-1}Z_{2,q^4} = (Y_{2,q^{-1}}Y_{2,q}) A_{2,1}^{-1}A_{2,q^2}^{-1}A_{1,q^4}^{-1}.$$
By interpolation, they correspond to the following monomials occurring in the $q$-characters of the fundamental module $L(Z_{2,1})$ 
of the twisted quantum affine algebra $\mathcal{U}_q(A_3^{(2)})$ :
$$Z_{1,q}Z_{1,-q^3}^{-1} = Z_{2,1} A_{2,q^2}^{-1}A_{1,-q^2}^{-1}\text{ , }Z_{1,-q}Z_{1,q^3}^{-1} = Z_{2,1} A_{2,q^2}^{-1}A_{1,q^2}^{-1}.$$
By folding \cite{H3}, they correspond to the following monomials occurring in the $q$-character of the fundamental module $L(Y_{2,1})$ 
of the simply-laced quantum affine algebra $\mathcal{U}_q(A_3^{(1)})$ :
$$Y_{1,q}Y_{3,q^3}^{-1} = Y_{2,1} A_{2,q}^{-1}A_{3,q^2}^{-1}\text{ , }Y_{3,q}Y_{1,q^3}^{-1} = Y_{2,1} A_{2,q}^{-1}A_{1,q^2}^{-1}.$$
The corresponding eigenvalues of $A_1^+(z)$ and $A_2^+(z)$ are respectively 
$$(q^2(1-zq^{-4}), q(1-zq^{-2}))\text{ and }(q^3(1-zq^{-6}), q(1-zq^{-2})),$$
which correspond to the contribution of $\Lambda_{1,q^{-2}}^{-1}\Lambda_{2,q^{-1}}^{-1}$ and $\Lambda_{2,q^{-1}}^{-1}\Lambda_{1,q^{-4}}^{-1}$ (see Remark \ref{rcront}).
\end{rem}

To complete the picture of fundamental representations in type $B_2$, let us now set $\lambda = \omega_1^\vee$, $Z_1(z) = 1 - z$, $Z_2(z) = 1 $. An analogous computation gives 
\begin{equation}\label{fundunc}\chi_q^L(V_1(1)) = Z_{1,1} + Z_{1,q^4}^{-1}Z_{2,q^2} + Z_{2,q^4}^{-1}Z_{1,q^2} + Z_{1,q^6}^{-1}.\end{equation}
In the same way the following representations descend to the corresponding truncation : 

$\mu = \omega_1^\vee$ : $\Psib = (1 - z, iq^{-1})$,

$\mu = \lambda - \alpha_1^\vee = \omega_2^\vee - \omega_1^\vee$ : $\Psib = \left( \frac{q}{1 - zq^{-4}},iq^{-1}(1 - zq^{-2}))\right)$,

$\mu = \lambda - \alpha_1^\vee - \alpha_2^\vee = \omega_1^\vee - \omega_2^\vee$ : $\Psib = \left( q^{-1}(1 - zq^{-2}) ,\frac{i}{1 - zq^{-4}} \right)$,

$\mu = \lambda - 2\alpha_1^\vee - \alpha_2^\vee = -\omega_1^\vee$ : $\Psib = \left(\frac{q^2}{1 - zq^{-6}}, i q^{-1} \right)$.

\subsection{Reduction to fundamental representations}

Let us first study the compatibility between fusion products and truncated shifted quantum affine algebras.

Let $\mu_1,\mu_2\in\Lambda$ and $\lambda_1,\lambda_2\in\Lambda^+$ so that $\mu_1\preceq \lambda_1$ and $\mu_2\preceq \lambda_2$. 
We consider corresponding set of parameters $\mathcal{Z}_1$, $\mathcal{Z}_2$. The product $\mathcal{Z}_1\mathcal{Z}_2$ is defined component-wise.

\begin{prop}\label{fustunc} If $V_1$ is in $\mathcal{O}_{\mu_1,\mathcal{Z}_1}^{\lambda_1}$ and $V_2$ is in $\mathcal{O}_{\mu_2,\mathcal{Z}_2}^{\lambda_2}$ then $V_1 * V_2$ is in $\mathcal{O}_{\mu_1 + \mu_2,\mathcal{Z}_1\mathcal{Z}_2}^{\lambda_1 + \lambda_2}$.
\end{prop}

\begin{proof} This follows from 
$$\Delta_u(\phi_i^\pm(z)) = \phi_i^\pm(z)\otimes \phi_i^\pm(zu).$$ 
Let us explain it for $\phi_i^+(z)$. 
Consider elements $A_{i,u}^+(z)$ associated to $\mathcal{Z}_1(z)\mathcal{Z}_2(zu)$. 
The $\mathcal{A}$-form which defines $V_1 * V_2$ is stable by the coefficients of $A_{i,u}^+(z)$ and
$$\Delta_u(A_{i,u}^+(z)) = A_i^{\mathcal{Z}_1,+}(z)\otimes A_i^{\mathcal{Z}_2,+}(zu).$$ 
Hence $A_{i,u}^+(z)$ is a polynomial in $z$ on the $\mathcal{A}$-form and 
$A_i^+(z)$ is polynomial on $V_1 * V_2$.\end{proof}

Let $\mu\in\Lambda$ and recall the functor $*_{i,a} : \mathcal{O}_{\mu} \rightarrow \mathcal{O}_{\mu + \omega_i^\vee}$. Consider $V$ a representation in $\mathcal{O}_\mu$. 
Let $\lambda\in \Lambda^+$ so that $\mu\preceq \lambda$. 
 
\begin{prop} The representation $*_{i,a}(V)$ is in $\mathcal{O}_{\mu+\omega_i^\vee,\mathcal{Z}'}^{\lambda + \omega_i^\vee}$ if and only if $V$ is in $\mathcal{O}_{\mu,\mathcal{Z}}^\lambda$. Here $\mathcal{Z}'$ is obtained from $\mathcal{Z}$ by 
replacing $Z_i(z)$ by $Z_i(z)(1 - za)$.
\end{prop}

\begin{proof} If $V$ is in $\mathcal{O}_{\mu,\mathcal{Z}}^\lambda$, we obtain from Proposition \ref{fustunc} 
that $L_{i,a}^+*V$ is in $\mathcal{O}_{\mu+\omega_i^\vee,\mathcal{Z}'}^{\lambda + \omega_i^\vee}$ as 
$L_{i,a}^+$ is in $\mathcal{O}_{\omega_i^\vee,\mathcal{Z}'/\mathcal{Z}}^{\omega_i^\vee}$ (up to a constant).
Now, suppose that $L_{i,a}^+*V$ is in $\mathcal{O}_{\mu+\omega_i^\vee,\mathcal{Z}'}^{\lambda + \omega_i^\vee}$.
Recall that for $j\in I$, the action of $A_j^{\mathcal{Z}'/\mathcal{Z},+}(z)$ on $L_{i,a}^+$ in 
the category $\mathcal{O}_{\omega_i^\vee,\mathcal{Z}'/\mathcal{Z}}^{\omega_i^\vee}$ is given by 
a constant scalar (see \ref{fex}). Let us denote $a_{j,0}\in\mathbb{C}^*$ this scalar.
The computation in the proof of Proposition \ref{fustunc} shows that for $j\in I$, we have 
the following action on the fusion product
$$A_j^{\mathcal{Z}',+}(z).(v_{i,a}\otimes v) = a_{j,0} v_{i,a}\otimes (A_j^{\mathcal{Z},+}(z).v).$$
This implies that $V$ is in the category $\mathcal{O}_{\mu,\mathcal{Z}}^\lambda$.
\end{proof}

As a consequence, to prove Conjecture \ref{mainc} (A),  it suffices to treat the case when $V^L$ is a fundamental 
representation, that is when $\lambda$ is a fundamental coweight.

Indeed, in Equation (\ref{ldqc}), a monomial $M$ occurring in $\chi_q^L(V^L)$ is of the form
$$M = \prod_{i\in I, 1\leq s \leq N_i} M_{i,s}$$
where $M_{i,s}$ is a monomial occurring in $\chi_q^L(V_i^L(q_i^{-1}z_{i,s}^{-1}))$.
For each $i,s$, consider $\mathcal{Z}_{i,s}$ associated to $\lambda = \omega_i^\vee$ and ${\bf Z}_{i,s} =\Psib_{M_{i,s}}$.
If we know that the simple representation $L(\Psib_{M_{i,s}})$ 
is in $\mathcal{O}^{\omega_{i}^\vee}_{\mu_{M_{i,s}},\mathcal{Z}_{i,s}}$ for any $i,s$, we obtain that 
the fusion product of the $L(\Psib_{M_{i,s}})$ is in $\mathcal{O}_\lambda^{\mu_M,\mathcal{Z}}$ from Proposition \ref{fustunc}.
And so $L(\Psib_{M})$ is in this category $\mathcal{O}_\lambda^{\mu_M,\mathcal{Z}}$ as $\Psib_M = \prod_{i,s} \Psib_{M_{i,s}}$.

Consequently, from the examples above, we obtain the following.

\begin{prop}\label{tdeux} In types $A_2$, $B_2$, Conjecture \ref{mainc} (A) is true.
\end{prop}

\subsection{Finite-dimensional representations} We use Baxter polynomiality of quantum integrable systems to establish the following.

\begin{thm}\label{truncfd} A simple finite-dimensional representation of a 
shifted quantum affine algebra $\mathcal{U}_q^\mu(\hat{\Glie})$ is in a category 
$\mathcal{O}_{\mu,\mathcal{Z}}^\lambda$, as predicted by Conjecture \ref{mainc}.
\end{thm}

From Proposition \ref{fustunc} and by the classification in Theorem \ref{fdclas}, it suffices to consider the
case of $W$ simple finite-dimensional representation of $\mathcal{U}_q(\hat{\Glie})$. 

Let $w_*$ be a lowest weight vector. Let $\omega_*$ be its weight. For $i\in I$, let $T_{i,*}^\pm(z)\in \mathbb{C}[[z^{\mp 1}]]$
be the eigenvalue of $T_i^\pm(z)$ on $w_*$. By Proposition \ref{egw2}, 
$T_{i,*}^-(z)/T_{i,\Psib}^-(z)$ and $z^{ht_i(\omega(\omega_*)^{-1})}T_{i,\Psib}^+(z)/T_{i,*}^+(z)$ are polynomial
in $z$ of degree $ht_i(\omega(\omega_*)^{-1})$. 

\begin{example} This can be observed in Example \ref{longex} with the respective eigenvalues of $T^-(z)$ and $T^+(z)$ on a lowest vector : 
$$(1 - zq^{-1})^2 T_{\Psib}^-(z) = T_{\Psib}^+(z^{-1}q^2)\text{ , }T_{\Psib}^+(z) (1 - z^{-1}q)^{-2} = T_{\Psib}^-(z^{-1}q^2).$$
\end{example}

More generally we have the following.

\begin{prop}  
For $\omega'$ a weight of $W$, on $W_{\omega'}$ the operators 
$$z^{ht_i(\omega'(\omega_*)^{-1})}\frac{T_i^+(z)}{T_{i,*}^+(z)} \text{ and }\frac{T_{i,*}^-(z)}{T_{i}^-(z)}  \in (\text{End}(W_{\omega'}))[z]$$
are polynomial in $z$ of degree $ht_i(\omega'(\omega_*)^{-1})$ and coincide up to a constant operator factor. 
\end{prop}

\begin{proof} It suffices to twist the representation by the morphism $\sigma$ of \cite[Proposition 1.6]{C} as in the proof of Theorem \ref{newpol}. Then the statement follows from this Theorem.
\end{proof}

We consider $A_i^{\mathcal{Z},+}(z)$ (this is analogous for $A_i^{\mathcal{Z},-}(z)$). We have 
$$A_i^{\mathcal{Z},+}(z) = (\overline{\phi}_i^+)^{-1} \overline{Y}_{i,{\bf Z}}^+(zq_i^{-1})\frac{T_i^+(z^{-1})}{T_i^+(z^{-1} q_i^2)}
=  \overline{Y}_{i,{\bf Z}}^+(zq_i^{-1})\frac{T_{i,*}^+(z^{-1})}{T_{i,\Psib}^+(z^{-1}q_i^2)} P_i(z),$$
where 
$P_i(z) = (\overline{\phi}_i^+)^{-1} \frac{T_i^+(z^{-1})}{T_{i,*}^+(z^{-1})} \frac{T_{i,\Psib}^+(z^{-1}q_i^2)}{T_i^+(z^{-1} q_i^2)}$ is a polynomial operator of degree $ht_i(\omega(\omega')^{-1}) + ht_i(\omega' (\omega_*)^{-1}) = ht_i(\omega (\omega_*)^{-1})$. Besides
\begin{equation}\label{scalfunc}\overline{Y}_{i,{\bf Z}}^+(zq_i^{-1})\frac{T_{i,*}^+(z^{-1})}{T_{i,\Psib}^+(z^{-1}q_i^2)} 
= \overline{Y}_{i,{\bf Z}\Psib^{-1}}^+(zq_i^{-1}) \times \frac{T_{i,*}^+(z^{-1})}{T_{i,\Psib}^+(z^{-1})}.\end{equation}
Note that by Proposition \ref{egw2}, 
$$T_{i,\Psib}^+(z^{-1})/T_{i,*}^+(z^{-1}) = \prod_{a\in\mathbb{C}^*}(1 - za)^{v_{i,a}}$$ 
is a polynomial of degree $ht_i(\omega(\omega_*)^{-1})$ where we have denoted $\Psib^* = \Psib\prod_{i\in I, a\in\mathbb{C}^*}A_{i,a}^{-v_{i,a}}$ 
the lowest $\ell$-weight of $W$. We will also denote by $u_{i,a}$ the multiplicity of $Y_{i,a}$ in $\Psib$ for $i\in I$, $a\in\mathbb{C}^*$.
Then consider
$${\bf Z} = \prod_{a\in\mathbb{C}^*} \prod_{i\in I}(\Psib_{i,aq_i^{-1}} \Psib_{\overline{i},aq^{r_i + r h^\vee}})^{u_{i,a}}\prod_{i\in I, r_i = 1 \neq  r} (\Psib_{i,aq_iq^{1-r}}\Psib_{\overline{i},aq^{r_i - r +1  + rh^\vee}}\Psib_{i,aq_iq^{r-1}}\Psib_{\overline{i},aq^{r_i + r - 1  + rh^\vee}})^{u_{i,a}}$$
$$\prod_{i\in I, r_i = r \neq 1}(\Psib_{i,aq_iq^{-2}}\Psib_{\overline{i},aq^{r_i - 2  + rh^\vee}})^{u_{i,a}}
\prod_{i\in I, r_i = 3 = r} 
(\Psib_{i,aq_iq^{-4}}\Psib_{\overline{i},aq^{r_i - 4  + rh^\vee}})^{u_{i,a}}$$
%$$\prod_{i\in I, r_i = 1 = r-1}(\Psib_{i,aq_iq^{-1}}\Psib_{\overline{i},aq^{r_i - 1  + rh^\vee}})^{u_{i,a}}
%(\Psib_{i,aq_iq}\Psib_{\overline{i},aq^{r_i + 1  + rh^\vee}})^{u_{i,a}}$$
%$$\prod_{i\in I, r_i = 1 = r-2} (\Psib_{i,aq_iq^{-2}}\Psib_{\overline{i},aq^{r_i - 2  + rh^\vee}})^{u_{i,a}}
%(\Psib_{i,aq_iq^2}\Psib_{\overline{i},aq^{r_i + 2  + rh^\vee}})^{u_{i,a}}$$
and a corresponding $\mathcal{Z}$, where we use the same notations as in (v) of Remark \ref{remconj}. We claim that $\Psib\preceq_{\mathcal{Z}} {\bf Z}$, that $L(\Psib)$ is in the category $\mathcal{O}_{\mu,\mathcal{Z}}^\lambda$ for the corresponding $\lambda$.

\noindent By (v) of Remark \ref{remconj}, with monomials translated in terms of the corresponding $\ell$-weights as above, and by Proposition \ref{compord}, we obtain that for each $i\in I$, $a\in\mathbb{C}^*$ :
$$1 \preceq_{\mathcal{Z}} \Psib_{i,a}\Psib_{\overline{i},aq^{r h^\vee}}.$$
This implies that, up to constants : 
$${\bf Z}\Psib^{-1} = 
\prod_{a\in\mathbb{C}^*} \prod_{i\in I}(\Psib_{i,aq_i} \Psib_{\overline{i},aq^{r_i + r h^\vee}})^{u_{i,a}}\prod_{i\in I, r_i = 1 \neq  r} (\Psib_{i,aq_iq^{1-r}}\Psib_{\overline{i},aq^{r_i - r +1  + rh^\vee}}\Psib_{i,aq_iq^{r-1}}\Psib_{\overline{i},aq^{r_i + r - 1  + rh^\vee}})^{u_{i,a}}$$
$$\prod_{i\in I, r_i = r \neq 1}(\Psib_{i,aq_iq^{-2}}\Psib_{\overline{i},aq^{r_i - 2  + rh^\vee}})^{u_{i,a}}
\prod_{i\in I, r_i = 3 = r} 
(\Psib_{i,aq_iq^{-4}}\Psib_{\overline{i},aq^{r_i - 4  + rh^\vee}})^{u_{i,a}} \succeq_{\mathcal{Z}} 1.$$
For the second point, let $\nu_{i,a}\geq 0$ so that 
$${\bf Z}\Psib^{-1} = \prod_{i\in I, a\in\mathbb{C}^*} \Lambda_{i,a}^{\nu_{i,a}}.$$ 
We obtain as in the proof of Lemma \ref{egw} that the eigenvalue of $\overline{Y}_{i,{\bf Z}\Psib^{-1}}^+(zq_i^{-1})$ 
corresponding to a factor $\Lambda_{i,a}$ is $(1 - z a q_i^{-1})$. Hence the condition 
\begin{equation}\label{sufc}\nu_{i,aq_i}\geq v_{i,a}\text{ for any $i\in I$, $a\in\mathbb{C}^*$.}\end{equation}
 implies that the scalar function (\ref{scalfunc}) is polynomial and so that $L(\Psib)$ in $\mathcal{O}_{\mu, \mathcal{Z}}^\lambda$.
Let us establish (\ref{sufc}).

For simply-laced types, the defining formula of the $\Lambda_{j,b}$ 
in terms of the $\Psib_{k,c}$ is the same as the defining formula of the $A_{j,b}$ in terms 
of the $Z_{k,c} = Y_{k,c}$. In particular, the factorization of each $\Psib_{i,aq}\Psib_{\overline{i},aq^{1 + r h^\vee}}$ 
a product of the $\Lambda_{i,b}$ is the same as the factorization of each $Y_{i,aq}Y_{\overline{i},aq^{1 + rh^{\vee}}}$ 
in terms of the $A_{j,b}$.
This implies $\nu_{i,aq} = v_{i,a}$ for any $i\in I$, $a\in\mathbb{C}^*$, the condition (\ref{sufc}) is clear.

For general types, the factorizations of $A_{j,b}$ and $\Lambda_{j,b}$ do not match, and so the 
powers $\nu_{i,aq}$ and $v_{i,a}$ are not equal in general. 
%one may consider the Laurent monomial $B_{j,b}$ be a Laurent monomial in the $Z_{i,a}$ so that $\Psib_{B_{j,b}} = \Lambda_{j,b}$.
However, the computation in the proof of Proposition \ref{compord} shows that the 
factorization of $\Psib_{i,a} \Psib_{\overline{i},aq^{rh^\vee}}$ in terms of the $\Lambda_{j,b}$
matches the factorization of $Y_{i,a}Y_{\overline{i},aq^{rh^\vee}}$ in terms of the 
$A_{j,b}$ if $r_j = r$, in terms of the $A_{j,bq^{-1}}A_{j,bq}$ if $r_j = 1 = r-1$, in terms of the $A_{j,bq^{-2}}A_{j,b}A_{j,bq^2}$ 
if $r_j = 1 = r - 2$. Hence, the power $a_{i,a,j,b}$ of $A_{j,b}$ in $Y_{i,a}Y_{\overline{i},aq^{rh^\vee}}$ and the power 
$b_{i,a,j,b}$ of $\Lambda_{j,b}$ in $\Psib_{i,a} \Psib_{\overline{i},aq^{rh^\vee}}$ are related by :

$a_{i,a,j,b} = b_{i,a,j,b} $ if $r_j = r$,

$a_{i,a,j,b} = b_{i,a,j,bq^{-1}} + b_{i,a,j,bq}$ if $r_j = 1 = r - 1$,

$a_{i,a,j,b} = b_{i,a,j,bq^{-2}} + b_{i,a,j,b} + b_{i,a,j,bq^2} $ if $r_j = 1 = r - 2$.

\noindent Now for $j\in I$, $b\in\mathbb{C}^*$, we have : 
$$\nu_{j,bq_j} = \sum_{a\in\mathbb{C}^*}(\sum_{i\in I}u_{i,a} b_{i,a,j,bq_jq_i^{-1}} 
+ \sum_{i\in I, r_i = 1 \neq  r} u_{i,a}(b_{i,a,j,bq_jq_i^{-1}q^{r - 1}} + b_{i,a,j,bq_jq_i^{-1}q^{1 - r}})$$
$$+ \sum_{i\in I, r_i = r \neq 1} u_{i,a} b_{i,a,j,bq_jq_i^{-1}q^2}
+ \sum_{i\in I, r_i = 3 = r} u_{i,a} b_{i,a,j,bq_jq_i^{-1}q^4})$$
$$\geq \sum_{i\in I,a\in\mathbb{C}^*}u_{i,a} a_{i,a,j,b} = v_{j,b}.$$
\noindent and we obtain the inequalities (\ref{sufc}).

To complete the proof of Theorem \ref{truncfd}, we check that the truncation we found is 
 one of the possible truncations predicted by Conjecture \ref{mainc}. Let us assume this Conjecture is correct and let us check that 
our result is compatible with it. 
As we as seen in  (v) in Remark \ref{remconj}, for any $i\in I$ and $a\in\mathbb{C}^*$, the monomial $Z_{\overline{i},aq^{rh^\vee}}^{-1}$ occurs as a monomial of $\chi_q^L(V_{i,a}^L)$. Then $M_{\Psib}$ occurs in the Langlands dual 
$q$-character associated to $M_{{\bf Z}}$. Indeed, each factor $(\Psib_{i,aq_i^{-1}} \Psib_{\overline{i},aq^{r_i + r h^\vee}})^{u_{i,a}}$ contributes as $(\Psib_{i,aq_i^{-1}} \Psib_{i,aq_i}^{-1})^{u_{i,a}} = ([-\omega_i]Y_{i,a})^{u_{i,a}}$ in $\Psib$ and all
other factors contribute as $1$. Now, this implies that $L(\Psib)$ is in $\mathcal{O}_{\mu,\mathcal{Z}}^{\lambda}$, 
where $\mu$ is the weight of $M_{\Psib}$. This coincides with the result that we have established.

\begin{example} Let $\mathfrak{g}$ of type $B_2$ with $r_1 = 2$ and $r_2 = 1$. 
Let $\Psib = [-\omega_1] Y_{1,1} = \Psib_{1,q^{-2}}\Psib_{1,q^2}^{-1}$ corresponding
to a $5$-dimensional fundamental representation. We have its lowest-weight 
$$\Psib^* = [-\omega_1]Y_{1,q^6}^{-1} = [-2\omega_1]\Psib_{1,q^8}\Psib_{1,q^4}^{-1} = \Psib A_{1,q^2}^{-1}A_{2,q^4}^{-1}A_{2,q^2}^{-1}A_{1,q^4}^{-1}.$$
We have also for any $a\in\mathbb{C}^*$ :
$$\Psib_{1,a}\Psib_{1,aq^6} = \Lambda_{1,aq^2}\Lambda_{2,aq^3}\Lambda_{1,aq^4}.$$
That why we set as is the proof of Theorem \ref{truncfd} : 
$${\bf Z} = \Psib_{1,q^{-2}}\Psib_{1,q^8}\Psib_{1,1}\Psib_{1,q^6}$$
polynomial so that 
$$\Psib = {\bf Z} \Lambda_{1,q^4}^{-1}\Lambda_{2,q^5}^{-1}\Lambda_{1,q^6}^{-1}\Lambda_{1,q^2}^{-1}\Lambda_{1,q^3}^{-1}\Lambda_{1,q^4}^{-1}.$$
We see that the inequalities (\ref{sufc}) are satisfied. Hence $L(\Psib)$ is in the category $\mathcal{O}_{0,\mathcal{Z}}^{4\omega_1^\vee}$ for a certain 
$\mathcal{Z}$ compatible with ${\bf Z}$. The corresponding Langlands dual standard module is 
$$V = V_{1,q^2}^L\otimes V_{1,q^{-8}}^L \otimes V_{1,1}^L \otimes V_{1,q^{-6}}^L.$$
The Langlands dual $q$-character of the $V_{1,a}^L$ are give by formula (\ref{fundunc}). We obtain that 
$$Z_{1,q^2}Z_{1,q^{-2}}^{-1} =  Z_{1,q^2}   Z_{1,q^{-2}}^{-1}  Z_{1,1}    Z_{1,1}^{-1}$$
occurs as a monomial in the Langlands dual $q$-character $\chi_q^L(V^L)$ as in Conjecture \ref{mainc}.
\end{example}


\begin{thebibliography}{ASM}

\bibitem[BLZ]{blz} {V. Bazhanov, S. Lukyanov and A. Zamolodchikov}, 
{\it Integrable structure of conformal field theory III. The Yang-Baxter relation}, 
{Commun. Math. Phys. {\bf 200}, (1999) 297--324.}

\bibitem[Be]{bec} {J. Beck}, 
\newblock {\it Braid group action and quantum affine algebras}, 
\newblock {\em Comm. Math. Phys.} {\bf 165}, 
(1994) 555--568. 

\bibitem[BeK]{BK} {J. Beck and V. Kac}, {\it Finite-dimensional representations of quantum affine algebras at roots of unity}, 
J. Amer. Math. Soc. {\bf 9} (1996), no. 2, 391--423.

\bibitem[Bo]{bo} {J. Bowman}, {\it Irreducible modules for the quantum affine algebra $U_q(g)$ and its Borel subalgebra
$U_q(g)^{\geq 0}$}, J. Algebra {\bf 316}, (2007) 231--253.

\bibitem[BLPW]{blpw} {T. Braden, A. Licata, N. Proudfoot and B. Webster}, {\it Quantizations of conical symplectic resolutions II: category O and symplectic duality}, 
Ast\'erisque No. {\bf 384} (2016), 75--179.

\bibitem[BFN1]{bfn} {A. Braverman, M. Finkelberg and H. Nakajima}, 
\newblock {\it Coulomb branches of 3d N = 4 quiver gauge theories and slices in the affine Grassmannian} ,(with appendices by A. Braverman, M. Finkelberg, J. Kamnitzer, R. Kodera, H. Nakajima, B. Webster, A. Weekes), 
Adv. Theor. Math. Phys. {\bf 23} (2019), no. 1, 75--166.

\bibitem[BFN2]{bfn2} {A. Braverman, M. Finkelberg and H. Nakajima}, 
\newblock {\it Towards a mathematical definition of Coulomb branches of 3-dimensional N=4 gauge theories, II}, 
Adv. Theor. Math. Phys. {\bf 22} (2018), no. 5, 1071--1147.

\bibitem[BrK]{brkl} {J. Brundan and A. Kleshchev}, {\it Representations of shifted Yangians and finite W-algebras}, 
{Mem. Amer. Math. Soc. {\bf 196} (2008), no. 918.}

\bibitem[C]{C} {V. Chari}, {\it Minimal affinizations of representations of quantum groups: the rank 2 case}, 
{Publ. Res. Inst. Math. Sci. {\bf 31} (1995), no. 5, 873--911.}

\bibitem[CG]{CG} {V. Chari and J. Greenstein}, {\it Filtrations and completions of certain positive level modules of affine
algebras}, Adv. Math. {\bf 194} (2005), 296--331.

\bibitem[CP]{CP} {V. Chari and A. Pressley},
{\it Quantum affine algebras and their representations,} 
\newblock CMS Conf. Proc. 16 (1994), 59--78. 

\bibitem[Da1]{da} {I. Damiani}, 
\newblock {\it La $\mathcal{R}$-matrice pour les alg\`ebres quantiques de type affine non tordu}, 
\newblock {\em Ann. Sci. Ecole Norm. Sup.} {\bf 31},  
(1998) 493--523.

\bibitem[Da2]{da2} I. Damiani, {\it From the Drinfeld realization to the Drinfeld-Jimbo presentation of affine quantum
algebras : Injectivity}, Publ. Res. Inst. Math. Sci. 51 (2015), 131--171.

\bibitem[Dr]{dr} {V. Drinfeld}, 
\newblock {\it A new realization of Yangians and of quantum affine algebras}, 
\newblock {Soviet Math. Dokl.} {\bf 36} (1988) 212--216.

\bibitem[FJMM]{FJMM} {B. Feigin, M. Jimbo, T. Miwa and E. Mukhin}, 
{\it Finite type modules and Bethe Ansatz equations}, 
\newblock Ann. Henri Poincaré {\bf 18} (2017), no. 8, 2543--2579.

\bibitem[F]{Fi} {M. Finkelberg}, 
{\it Double affine Grassmannians and Coulomb branches of 3d N=4 quiver gauge theories}, 
{Proceedings of the ICM-Rio de Janeiro (2018), Vol. II. Invited lectures, 1283--1302.}

\bibitem[FKPRW]{fkprw} {M. Finkelberg, J. Kamnitzer, K. Pham, L. Rybnikov and A. Weekes}, 
{\it Comultiplication for shifted Yangians and quantum open Toda lattice}, 
Adv. Math. {\bf 327} (2018), 349--389.

\bibitem[FT]{FT} {M. Finkelberg and A. Tsymbaliuk}, 
{\it Multiplicative slices, relativistic Toda and shifted quantum affine algebras},
\newblock in Progr. Math. {\bf 330}(2019), 133--304.

\bibitem[FPT]{FPT} {R. Frassek, V. Pestun and A. Tsymbaliuk}, 
{\it Lax matrices from antidominantly shifted Yangians and quantum affine algebras}, 
{Preprint arXiv:2001.04929}

\bibitem[FH1]{FH0} {E. Frenkel and D. Hernandez}, {\it Langlands duality for finite-dimensional representations of quantum affine algebras}, {Lett. Math. Phys. {\bf 96} (2011), no. 1-3, 217--261.}

\bibitem[FH2]{FH} {E. Frenkel and D. Hernandez},
\newblock {\em Baxter's Relations and Spectra of Quantum Integrable
  Models},
\newblock { Duke Math. J.} {\bf 164} (2015), no. 12, 2407--2460.

\bibitem[FH3]{FH2} {E. Frenkel and D. Hernandez}, 
\newblock {\em Spectra of quantum KdV Hamiltonians, Langlands duality, and affine opers}, 
\newblock {Comm. Math. Phys. {\bf 362} (2018), no. 2, 361--414}.

\bibitem[FHR]{FHR} {E. Frenkel, D. Hernandez and N. Reshetikhin}, {\em Folded quantum integrable models and deformed $\mathcal{W}$-algebras}, Preprint arXiv:2110.14600.

\bibitem[FM]{Fre2} {E. Frenkel and E. Mukhin}, 
{\it Combinatorics of $q$-Characters of Finite-Dimensional Representations of Quantum Affine Algebras}, {Comm. Math. Phys., vol {\bf 216} (2001), no. 1, 23--57.}

\bibitem[FR1]{Wd} {E. Frenkel and N. Reshetikhin}, {\it Deformations of $W$-algebras associated to simple Lie algebras},
Comm. Math. Phys. {\bf 197} (1998), 1--32 (1998).

\bibitem[FR2]{Fre} {E. Frenkel and N. Reshetikhin}, {\it The $q$-Characters of Representations of Quantum Affine Algebras and Deformations of $W$-Algebras}, {Recent Developments in Quantum Affine Algebras and related topics, Cont. Math., vol. {\bf 248} (1999), 163--205.}

\bibitem[FZ]{FZ1} S. Fomin and A. Zelevinsky, 
Cluster algebras I: Foundations,
{\em J. Amer. Math. Soc.} {\bf 15} (2002), 497--529.

\bibitem[GTL]{GTL} {S. Gautam and V. Toledano Laredo}, {\it Yangians, quantum loop algebras, and abelian difference equations}, 
{J. Amer. Math. Soc. {\bf 29} (2016), no. 3, 775--824.}

\bibitem[He1]{H} {D. Hernandez},
{\it Representations of quantum affinizations and fusion product},
\newblock Transform. Groups {\bf 10} (2005), no. 2, 163--200. 

\bibitem[He2]{H2} {D. Hernandez},
{\it Drinfeld coproduct, quantum fusion tensor category and applications}, 
\newblock Proc. Lond. Math. Soc. (3) {\bf 95} (2007), no. 3, 567--608. 

\bibitem[He3]{H3tb} {D. Hernandez}, 
{\it On minimal affinizations of representations of quantum groups}, 
\newblock Comm. Math. Phys. {\bf 277}, no. 1, 221--259 (2007).

\bibitem[He4]{H3} {D. Hernandez},
{\it Kirillov-Reshetikhin conjecture: the general case}, Int. Math. Res. Not. IMRN 2010, no. 1, 149--193.

\bibitem[HJ]{HJ} {D. Hernandez and M. Jimbo},
\newblock {\em Asymptotic representations and Drinfeld rational
  fractions},
\newblock  Compos. Math. {\bf 148} (2012), 1593--1623.

\bibitem[HL1]{HL0}{D. Hernandez and B. Leclerc},
\newblock Cluster algebras and quantum affine algebras,
\newblock {\em Duke Math. J.} {\bf 154} (2010), no. 2, 265--341.

\bibitem[HL2]{HLJEMS}{D. Hernandez and B. Leclerc}, 
{\em A cluster algebra approach to $q$-characters of Kirillov-Reshetikhin modules},
J. Eur. Math. Soc., {\bf 18} (2016), 1113--1159.

\bibitem[HL3]{HL} {D. Hernandez and B. Leclerc}, 
\newblock {\em Cluster algebras and category O for representations of Borel subalgebras
of quantum affine algebras}, 
\newblock Algebra and Number Theory {\bf 10} (2016), 2015--2052.

\bibitem[HN]{hn} {D. Hernandez and H. Nakajima},
\newblock {\em Level 0 monomial crystals}, 
\newblock Nagoya Math. J. {\bf 184} (2006), 85--153.

\bibitem[Hi]{hik} {T. Hikita}, 
\newblock An algebro-geometric realization of the cohomology ring of Hilbert scheme
of points in the affine plane, 
\newblock Int. Math. Res. Not. {\bf 2017}, no. 8, 2538--2561.


\bibitem[J]{J} {M. Jimbo}, 
\newblock {\em A q-analogue of U(gl(N+1)), Hecke algebra, and the Yang-Baxter equation}, 
Lett. Math. Phys. {\bf 11} (1986), no. 3, 247--252.

\bibitem[Kac]{kac} {V. Kac}, {\em Infinite dimensional Lie algebras}, 
3rd Edition, Cambridge University Press (1990).


\bibitem[KTWWY1]{ktwwy} {J. Kamnitzer, P. Tingley, B. Webster, A. Weekes, and O. Yacobi}, 
\newblock {\em Highest weights for truncated shifted Yangians and product monomial crystals}, 
{J. Comb. Algebra {\bf 3} (2019), no. 3, 237--303.}

\bibitem[KTWWY2]{ktwwy2} {J. Kamnitzer, P. Tingley, B. Webster, A. Weekes, and O. Yacobi}, 
\newblock {\em On category O for affine Grassmannian slices and categorified tensor products},
{Proc. Lond. Math. Soc. {\bf 119} (2019), no. 5, 1179--1233.}

\bibitem[KWWY]{kwwy1} {J. Kamnitzer, B. Webster, A. Weekes, and O. Yacobi}, 
\newblock Yangians and quantization of slices in the affine Grassmannian, 
\newblock Algebra Number Theory {\bf 8} (2014), no. 4, 857--893.

\bibitem[KKKO]{kkko3} {S-J. Kang, M. Kashiwara, M. Kim and S-J. Oh}
\newblock {\em Symmetric quiver Hecke algebras and R-matrices of quantum affine algebras III}, 
\newblock Proc. Lond. Math. Soc. (3) {\bf 111} (2015), no. 2, 420--444.

\bibitem[Kas1]{K} {M. Kashiwara}, 
\newblock {\em On crystal bases of the q-analogue of universal enveloping algebras}, 
\newblock Duke Math. J. {\bf 63} (1991), 465--516.

\bibitem[Kas2]{K2} {M. Kashiwara}, 
\newblock {\em Realizations of crystals}, 
\newblock in Combinatorial and geometric representation theory, Contemp. Math. {\bf 325} (2003), 133--139.

\bibitem[Kas3]{K3} {M. Kashiwara},
\newblock {\em Crystal bases and categorifications—Chern Medal lecture}, 
\newblock {Proceedings of the ICM-Rio de Janeiro 2018. Vol. I. Plenary lectures, 249--258.}

\bibitem[KKOP1]{kkop} {M. Kashiwara, M. Kim, S-J. Oh and E. Park}, 
\newblock {\em Categories over quantum affine algebras and monoidal categorification}, 
{Preprint arXiv:2005.10969.}

\bibitem[KKOP2]{kkop2} {M. Kashiwara, M. Kim, S-J. Oh and E. Park}, 
\newblock {\em Monoidal categorification and quantum affine algebras II}, 
{Preprint arXiv:2103.10067.}

\bibitem[KW]{kw} {R. Kodera and K. Wada},
\newblock {\em Finite dimensional simple modules of (q,Q)-current algebras},
{J. Algebra {\bf 570} (2021), 470--530.}

\bibitem[MY]{my} {E. Mukhin and C. Young}, 
\newblock {\em Affinization of category O for quantum groups}, 
\newblock Trans. Amer. Math. Soc. {\bf 366} (2014), no. 9, 4815--4847.

\bibitem[N1]{N} {H. Nakajima} {\it Quiver varieties and finite-dimensional representations of quantum affine algebras}, 
J. Amer. Math. Soc. {\bf 14} (2001), no. 1, 145--238.

\bibitem[N2]{nacr} {H. Nakajima},
\newblock {\em  t–analogs of q–characters of quantum affine algebras of type $A_n$, $D_n$},
\newblock in Combinatorial and geometric representation theory (Seoul, 2001), 141--160, Contemp. Math., {\bf 325}, Amer. Math. Soc., Providence, RI (2003).

\bibitem[N3]{N1} {H. Nakajima}, 
\newblock {\em Quiver varieties and t-analogs of q-characters of quantum affine algebras}, 
\newblock Annals of Math. {\bf 160} (2004), 1057--1097.

\bibitem[N4]{Nreview} {H. Nakajima},
\newblock {\em Introduction to a provisional mathematical definition of Coulomb branches of 3-dimensional  $N=4$  gauge theories}, 
{in Proc. Sympos. Pure Math. {\bf 99} (2018), 193--211.}

\bibitem[N5]{N4} {H. Nakajima}, 
\newblock {\em Modules of quantized Coulomb branches}, 
\newblock in preparation.

\bibitem[NW]{nw} {H. Nakajima and A. Weekes}, {\it Coulomb branches of quiver gauge theories with symmetrizers}, 
{to appear in J. Eur. Math. Soc. (preprint arXiv:1907.06552).}

\bibitem[O]{o} {A. Okounkov}, {\it On the crossroads of enumerative geometry and geometric representation theory}, 
{Proceedings of the ICM-Rio de Janeiro (2018) Vol. I. Plenary lectures, 839--867.}

\bibitem[Q]{Q} {F. Qin}, {\it Triangular bases in quantum cluster algebras and monoidal categorification conjectures}, Duke Math. J. {\bf 166} (2017), no. 12, 2337--2442.

\bibitem[SS]{ss} {G. Schrader and A. Shapiro}, {\it K-theoretic Coulomb branches of quiver gauge theories and cluster varieties}, Preprint 
arXiv:1910.03186.  

\bibitem[T1]{t1} {Z. Tsuboi}, {\it Asymptotic representations and q-oscillator solutions of the graded Yang-Baxter equation related to Baxter 
Q-operators}, {Nuclear Phys. B {\bf 886} (2014), 1--30.}

\bibitem[T2]{t2} {Z. Tsuboi}, {\it A note on q-oscillator realizations of Uq(gl(M|N)) for Baxter Q-operators}, 
{Nuclear Phys. B {\bf 947} (2019), 114747.}

\bibitem[VV]{VV} {M. Varagnolo and E. Vasserot}, {\it Standard modules of quantum affine algebras}, 
Duke Math. J. {\bf 111} (2002), no. 3, 509--533.

%\bibitem[Y]{Y} {Y. Young}, {\it Quantum loop algebras and l-root operators}, 
%{Transform. Groups {\bf 20} (2015) no.4, 1195--1226.}

\bibitem[Z]{Z} {H. Zhang},
\newblock {\em Yangians and Baxter's relations},
\newblock Lett. Math. Phys. {\bf 110} (2020) 2113--2141.

\end{thebibliography}
\end{document}